\documentclass[12pt]{article}
\usepackage{amsmath,amssymb,amsthm,amscd}
\usepackage[shortlabels]{enumitem}
\usepackage{mathtools}
\mathtoolsset{showonlyrefs}
\usepackage[all]{xy}
\usepackage[headinclude,headlines=1]{typearea}
\typearea{20}
\usepackage[dvipdfmx,colorlinks,setpagesize=false,linktocpage=true]{hyperref}
\usepackage{titletoc}
\titlecontents{section}[3em]{}{\contentslabel{2em}}{}{\hfill \contentspage}{}
\usepackage[nottoc]{tocbibind}

\newtheorem{thm}{Theorem}[section]
\newtheorem*{thm*}{Theorem}
\newtheorem{prop}[thm]{Proposition}
\newtheorem{lem}[thm]{Lemma}
\newtheorem{cor}[thm]{Corollary}

\theoremstyle{definition}
\newtheorem{defn}[thm]{Definition}
\newtheorem{assump}[thm]{Assumption}
\newtheorem{exa}[thm]{Example}
\newtheorem{rmk}[thm]{Remark}
\newtheorem{notn}[thm]{Notation}
\newtheorem*{ack}{Acknowledgments}
\newtheorem{para}[thm]{}

\setlist[enumerate]
{leftmargin=56pt,labelsep=8pt,itemsep=0pt,label=\upshape{(\thethm.\arabic*)}}

\numberwithin{equation}{section}

\newcommand{\vo}{\boldsymbol 0}

\newcommand{\vc}{\boldsymbol c}
\newcommand{\ve}{\boldsymbol e}
\newcommand{\vj}{\boldsymbol j}
\newcommand{\vq}{\boldsymbol q}
\newcommand{\vr}{\boldsymbol r}
\newcommand{\vs}{\boldsymbol s}
\newcommand{\vu}{\boldsymbol u}
\newcommand{\vv}{\boldsymbol v}
\newcommand{\vw}{\boldsymbol w}

\newcommand{\rf}{\boldsymbol r\!}

\newcommand{\RomI}{\uppercase\expandafter{\romannumeral 1}}
\newcommand{\RomII}{\uppercase\expandafter{\romannumeral 2}}

\newcommand{\bC}{\mathbb C}
\newcommand{\bN}{\mathbb N}
\newcommand{\bQ}{\mathbb Q}
\newcommand{\bZ}{\mathbb Z}
\newcommand{\bnnZ}{\mathbb Z_{\ge 0}}
\newcommand{\bpZ}{\mathbb Z_{> 0}}
\newcommand{\bR}{\mathbb R}
\newcommand{\bpR}{\mathbb R_{> 0}}
\newcommand{\ki}{1, 2, \dots, k}
\newcommand{\ski}{\{\ki\}}

\newcommand{\cC}{\mathcal C}
\newcommand{\cL}{\mathcal L}
\newcommand{\cM}{\mathcal M}
\newcommand{\cN}{\mathcal N}
\newcommand{\cO}{\mathcal O}

\newcommand{\ca}{complex analytic }
\newcommand{\cmh}{cohomological mixed Hodge }
\newcommand{\coh}{{\rm H}}
\newcommand{\gp}{^{\rm gp}}
\newcommand{\nc}{normal crossing }
\newcommand{\snc}{simple normal crossing }
\newcommand{\ssls}{semistable log smooth }
\newcommand{\ula}{\underline{\lambda}}
\newcommand{\umu}{\underline{\mu}}
\newcommand{\unu}{\underline{\nu}}

\DeclareMathOperator{\an}{an}
\DeclareMathOperator{\dlog}{dlog}
\DeclareMathOperator{\ext}{Ext}
\DeclareMathOperator{\gr}{Gr}
\DeclareMathOperator{\id}{id}
\DeclareMathOperator{\image}{Image}
\DeclareMathOperator{\kernel}{Ker}
\DeclareMathOperator{\kos}{Kos}
\DeclareMathOperator{\rank}{rank}
\DeclareMathOperator{\rec}{rec}
\DeclareMathOperator{\res}{Res}
\DeclareMathOperator{\spec}{Spec}

\begin{document}

\title
{Limiting mixed Hodge structures
on the relative log de Rham cohomology groups
of a projective semistable log smooth degeneration}
\author{Taro Fujisawa \\
 \\
Tokyo Denki University \\
e-mail: fujisawa@mail.dendai.ac.jp}
\date{\today}
\footnotetext[0]
{\hspace{-18pt}2020 {\itshape Mathematics Subject Classification}.
Primary 14C30; Secondary 14A21, 32S35. }

\maketitle

\begin{abstract}
We prove that
the relative log de Rham cohomology groups
of a projective semistable log smooth degeneration
admit a natural \textit{limiting} mixed Hodge structure.
More precisely,
we construct a family of increasing filtrations
and a family of nilpotent endomorphisms
on the relative log de Rham cohomology groups
and show that they satisfy
a part of good properties
of a nilpotnet orbit in several variables.
\end{abstract}

\tableofcontents

\section{Introduction}

A morphism from a complex manifold
to a polydisc
is said to be semistable,
if it is locally isomorphic
to a product of semistable degenerations over the unit disc
(cf. Example \ref{exa:1}
and \cite[Lemma 3.3]{FujisawaLHSSV2}).
The notion of \ssls degeneration
is an abstraction of the central fiber of a semistable morphism
in the context of log geometry.
Namely,
a \ssls degeneration
is a log \ca space $(X, \cM_X)$ over the log point $(\ast, \bN^k)$
(i.e. a morphism of log \ca space
$f \colon (X, \cM_X) \longrightarrow (\ast,\bN^k)$),
which is locally isomorphic to the central fiber of
a semistable morphism to the $k$-dimensional polydisc
in the category of log \ca spaces
(cf. local description in \ref{para:1}).
For the precise definition of \ssls degeneration,
see Definition \ref{defn:19}.

One of the main results of this paper is the following.

\begin{thm}
\label{thm:9}
Let $f \colon (X, \cM_X) \longrightarrow (\ast, \bN^k)$
be a projective \ssls degeneration.
Then the relative log de Rham cohomology groups
$\coh^n(X, \Omega_{X/\ast}(\log (\cM_X/\bN^k)))$
admit a \text{limiting} mixed Hodge structure,
whose Hodge filtration $F$ is induced from
the stupid filtration
$($filtration b\^ete in \textup{\cite[(1.4.7)]{DeligneII})}
on $\Omega_{X/\ast}(\log (\cM_X/\bN^k))$.
\end{thm}

Here, a $\bQ$-mixed Hodge structure
$((V_{\bQ}, W), (V_{\bC}, W, F))$
is called a \textit{limiting} mixed Hodge structure,
if there exists a nilpotent endomorphism $N$ of $V_{\bQ}$
with $W=W(N)[k]$ for some $k \in \bZ$,
where $W(N)$ denotes the monodromy weight filtration of $N$
(cf. \cite[p. 90]{GreenGriffithsDTLMHS}).
Theorem \ref{thm:9} is deduced from the following theorem:

\begin{thm}[cf. Theorems \ref{thm:6} and \ref{thm:4}]
\label{thm:8}
On $\coh^n(X, \Omega_{X/\ast}(\log (\cM_X/\bN^k)))$,
we can construct a finite increasing filtration $L(I)$
for all $I \subset \ski$
and nilpotent endomorphisms $N_1, \dots, N_k$ such that
the following is satisfied$:$
\begin{enumerate}
\item
\label{item:34}
By setting $L=L(\ski)$,
the triple
$(\coh^n(X, \Omega_{X/\ast}(\log (\cM_X/\bN^k))), L[n], F)$
underlies a $\bQ$-mixed Hodge structure.
\item
\label{item:39}
$L(I)$ coincides with the monodromy weight filtration
of the nilpotent endomorphism $N_I(c_I)=\sum_{i \in I}c_iN_i$
for all $c_I=(c_i)_{i \in I} \in (\bpR)^I$.
\end{enumerate}
\end{thm}

The case of $I=\ski$ in \ref{item:39}
together with \ref{item:34}
implies Theorem \ref{thm:9}.
Moreover, \ref{item:39} claims that
the monodromy weight filtration of $N_I(c_I)$
is independent of the choice of $c_I \in (\bpR)^I$.
The following theorem states the relation between
the filtrations $L$ and $L(I)$.

\begin{thm}[cf. Theorem \ref{thm:2}]
\label{thm:10}
On $\coh^n(X, \Omega_{X/\ast}(\log (\cM_X/\bN^k)))$,
the filtration $L$ is the monodromy weight filtration
of $N(c)=\sum_{i=1}^{k}c_iN_i$
relative to $L(I)$
for all $c=(c_i)_{i=1}^k \in (\bpR)^k$.
\end{thm}

Theorems \ref{thm:8} and \ref{thm:10}
are consequences of
Theorems \ref{thm:6}, \ref{thm:4} and \ref{thm:2}.
Theorem \ref{thm:6}
follows directly from Theorem \ref{thm:1},
which will be proved in Section \ref{sec:proof-theor-refthm:1}.
Theorems \ref{thm:4} and \ref{thm:2}
will be proved together with Theorems \ref{thm:3} and \ref{thm:5}
in Section \ref{sec:proofs-main-theorems}.
Theorem \ref{thm:3}
claims the $E_2$-degeneracy
of the spectral sequence associated to the filtration $L(I)$.
This is a generalization of the result on $E_2$-degeneracy
for a projective semistable morphism
in \cite{FujisawaDWSS} and \cite{FujisawaLHSSV2}
to the case of a projective \ssls degeneration.
Theorem \ref{thm:5},
which is a by-product of the proof of Theorems  \ref{thm:4} and \ref{thm:2},
states that
the analogue of the hard Lefschetz theorem
for $\coh^{\ast}(X, \Omega_{X/\ast}(\log (\cM_X/\bN^k)))$ holds true.
In \cite{NakkajimaSWSS},
Y. Nakkajima stated the log hard Lefschetz conjecture
and proved it for a projective SNCL variety over the standard log point
(cf. Conjecture 9.5 and Theorem 9.14 in \cite{NakkajimaSWSS}).
Theorem \ref{thm:5} is the affirmative answer
to an analogue of the log hard Lefschetz conjecture
for a projective \ssls degeneration.

\begin{para}
\label{para:14}
This paper is partially motivated
by Theorems \textup{\RomI} and \textup{\RomI'} of Green and Griffiths
\cite{GreenGriffithsDTLMHS}.
Let $X$ be a reduced \ca space,
which is locally isomorphic to a product of \nc varieties
as in (\RomI.2) of \cite{GreenGriffithsDTLMHS}.
Then Green and Griffiths claimed that
a certain type of infinitesimal deformation of $X$
(cf. p.100 and p.108 of \cite{GreenGriffithsDTLMHS})
canonically yields a polarized limiting mixed Hodge structure
under the appropriate projectivity assumption.
In fact,
the existence of a good infinitesimal deformation of $X$,
which is assumed in \cite{GreenGriffithsDTLMHS} as above,
implies that there exists a log structure $\cM_X$
such that $(X, \cM_X)$ becomes a \ssls degeneration.
Thus Theorem \ref{thm:9} above
is an analogue of Theorems \textup{\RomI} and \textup{\RomI'}
of \cite{GreenGriffithsDTLMHS}
in the context of log geometry.
We note that the difference between
Theorems \textup{\RomI} and \textup{\RomI'}
of \cite{GreenGriffithsDTLMHS}
and Theorem \ref{thm:9} above
is about polarization.
We will return to this point later.
\end{para}

\begin{para}
For the case of $k=1$,
a \ssls degeneration
is called a log deformation
by Steenbrink in \cite{SteenbrinkLE}.
The relative log de Rham cohomology groups
of a projective strict log deformation
are thoroughly studied in \cite{SteenbrinkLE}, \cite{Fujisawa-Nakayama}
and \cite{FujisawaPLMHS}.
In particular, it is proved in \cite{FujisawaPLMHS}
that they admit a natural \textit{polarized} limiting mixed Hodge structure.
Thus the result of this paper is a partial generalization
of results in \cite{SteenbrinkLE}, \cite{Fujisawa-Nakayama}
and \cite{FujisawaPLMHS}
to a projective \ssls degeneration.

Study of limiting mixed Hodge structures
for a projective semistable degeneration over the unit disc
originated from Steenbrink \cite{SteenbrinkLHS},
in which he proved that such a morphism
yields a natural limiting mixed Hodge structure
on the relative log de Rham cohomology groups of the central fiber
(cf. \cite{ElZeinCE},
\cite{SaitoMorihikoMHP},
\cite{Guillen-NavarroAznarCI},
\cite{UsuiMTTS}).
His results were generalized by the author's previous works
\cite{FujisawaLHSSV}, \cite{FujisawaDWSS} and \cite{FujisawaLHSSV2}
to the case of a projective semistable morphism
over a higher dimensional polydisc.
The other motivation of this paper
is to generalize these results
to a projective \ssls degeneration.
\end{para}

\begin{para}
We briefly explain the outline of this paper.
In Section \ref{sec:preliminaries},
we fix notation
and collect several preliminary definitions and results
for the later use.
In Section \ref{sec:semist-log-smooth},
we introduce the notion of a \ssls degeneration.
Hereafter, a \ssls degeneration
$f \colon (X, \cM_X) \longrightarrow (\ast,\bN^k)$ is fixed.
In \ref{para:1},
we give a local description of a \ssls degeneration,
which is constantly used throughout this paper.
Some notation and results
on Koszul complexes are briefly recalled in \ref{para:11}.
In Definition \ref{defn:24},
we construct
$((A_{\bQ}, L(I)), (A_{\bC}, L(I), F), \alpha)$
consisting of 
a complex of $\bQ$-sheaves $A_{\bQ}$
equipped with an increasing filtration $L(I)$,
a complex of $\bC$-sheaves $A_{\bC}$
equipped with an increasing filtration $L(I)$
and an decreasing filtration $F$,
and a morphism of complexes of $\bQ$-sheaves
$\alpha \colon A_{\bQ} \longrightarrow A_{\bC}$
preserving the filtrations $L(I)$ for all $I \subset \ski$.
(We set $L=L(\ski)$ as in \ref{item:34}.)
These data play a central role in this paper.
In fact, Lemma \ref{lem:15} states that
$(A_{\bC}, F)$ is filtered quasi-isomorphic
to $(\Omega_{X/\ast}(\log (\cM_X/\bN^k)), F)$.
Therefore the filtered vector space
$(\coh^{\ast}(\Omega_{X/\ast}(\log (\cM_X/\bN^k))), F)$
is replaced by
$(\coh^{\ast}(X, A_{\bC}), F)$ in what follows
(cf. Corollary \ref{cor:3}).
Section \ref{sec:main-results}
is devoted to state the main results of this paper,
Theorems \ref{thm:1}, \ref{thm:6}, \ref{thm:3},
\ref{thm:4}, \ref{thm:2}, and \ref{thm:5}.
An endomorphism $\nu_i$ on $A_{\bQ}$ and $A_{\bC}$,
which induces the nilpotent endomorphism $N_i$ in Theorem \ref{thm:8},
is defined in Definition \ref{defn:14} for $i=\ki$.
In Section \ref{sec:proof-theor-refthm:1},
Theorem \ref{thm:1} is proved.
We first construct log complex manifolds
$(X_{\vr}, \cM_{X_{\vr}})$ in Definition \ref{defn:28}.
Then we define the residue morphism \eqref{eq:111}
for the log de Rham complex
in Definition \ref{defn:23},
and \eqref{eq:112} for the Koszul complex
in Definition \ref{defn:30} respectively.
Once these residue morphisms are obtained,
Theorem \ref{thm:1}
is a consequence of the classical Hodge theory
on $\coh^{\ast}(X_{\vr}, \varepsilon_{\vr} \otimes_{\bZ} \bC)$,
where $\varepsilon_{\vr}$ is a locally free $\bZ$-module of rank one
admitting a positive definite symmetric bilinear form
(see Definition \ref{defn:6} for $\varepsilon_{\vr}$).

To prove the remaining theorems,
we will apply a result
on a polarized differential multi-graded Hodge-Lefschetz module
to $V_{\bC}=\bigoplus_{a,b}E_1^{a,b}(A_{\bC}, L)$.
(Precisely, we will apply Proposition \ref{prop:1}
to the real form $V_{\bR}$ of $V_{\bC}$.)
To prove that $V_{\bR}$ is
a polarized differential
$\bZ \oplus \bZ^k$-graded Hodge-Lefschetz module,
the most subtle point is to construct
a polarization on $V_{\bR}$.
Apparently it looks possible
to obtain such a polarization directly
from the fact that $V_{\bC}$ is expressed as
a direct sum of cohomology groups
$\coh^{\ast}(X_{\vr}, \varepsilon_{\vr} \otimes_{\bZ} \bC)$.
Actually we can obtain a bilinear form
on $V_{\bC}$ by using the polarization
on $\coh^{\ast}(X_{\vr}, \varepsilon_{\vr} \otimes_{\bZ} \bC)$
as in \cite[(3.4)]{Guillen-NavarroAznarCI}.
However, it is difficult to prove
that the bilinear form obtained in this way
is compatible with
the morphism $d_1$ of $E_1$-terms
of the spectral sequence $E_r^{\ast,\ast}(A_{\bC}, L)$

The idea to avoid this difficulty
is to construct a product on the \textit{filtered complex} $(A_{\bC}, L)$,
which induces the desired bilinear form on $V_{\bC}$.
The fact that the bilinear form comes from
a product on $(A_{\bC}, L)$
enables us to analyze the relation
between this bilinear form
and the morphism $d_1$ of the $E_1$-terms.
To carry out this idea,
we follow the arguments in \cite{FujisawaPLMHS}
and adapt them to the case of a \ssls degeneration.
In Section \ref{sec:a Cech type complex},
we construct
a \v{C}ech type filtered complex
$(\cC(\Omega_{X_{\bullet}}(\log \cM_{X_{\bullet}})),\delta W)$
and a product on it.
Moreover, under the assumption that $X$ is of pure dimension,
we construct a morphism
$\Theta \colon
E_1^{-k, 2\dim X+2k} \longrightarrow \bC$,
where $E_1^{p,q}$ denotes
the $E_1$-terms of the spectral sequence associated to
the filtered complex
$(R\Gamma_c(X, \cC(\Omega_{X_{\bullet}}(\log \cM_{X_{\bullet}}))), \delta W)$.
In Section \ref{sec:constr-bilin-forms},
a product
$A_{\bC} \otimes_{\bC} A_{\bC} \longrightarrow
\cC(\Omega_{X_{\bullet}}(\log \cM_{X_{\bullet}}))[k]$
is constructed
by using the residue morphism on $A_{\bC}$
and the product on $\cC(\Omega_{X_{\bullet}}(\log \cM_{X_{\bullet}}))$.
This product induces a product on
$E_1^{\ast,\ast}(A_{\bC}, L)$
with values in
$E_1^{\ast,\ast}(
\cC(\Omega_{X_{\bullet}}(\log \cM_{X_{\bullet}})), \delta W)$.
In Section \ref{sec:bilin-form},
under the assumption that $X$ is projective and of pure dimension,
a bilinear form
on $V_{\bC}=\bigoplus_{a,b}E_1^{a,b}(A_{\bC}, L)$
is constructed as follows.
For two elements of $V_{\bC}$,
the product of these elements
is contained in
$\bigoplus_{p,q}E_1^{p,q}(
\cC(\Omega_{X_{\bullet}}(\log \cM_{X_{\bullet}})), \delta W)$.
Then taking its image by the projection
to the direct summand
$E_1^{-k,2\dim X+2k}(
\cC(\Omega_{X_{\bullet}}(\log \cM_{X_{\bullet}})), \delta W)$,
and evaluate it by the morphism $\Theta$ above
with an appropriate sign.
(For the precise definition,
see Definition \ref{defn:2}.)
Lemma \ref{lem:22} shows the compatibility
of this bilinear form
with the morphism $d_1$ of the $E_1$-terms
of the spectral sequence $E_r^{p,q}(A_{\bC}, L)$.
This lemma follows form the fact that
the product on $V_{\bC}$ is induced from
the product on the
\textit{filtered complex} $(A_{\bC},L)$
and from the equality $\Theta \cdot d_1=0$ in Lemma \ref{lem:13}.
Restricting this bilinear form
to the real form $V_{\bR}$,
it turns out to be 
a polarized differential $\bZ \oplus \bZ^k$-graded
Hodge-Lefschetz module as expected.
Section \ref{sec:multi-graded-hodge}
is devoted to the arguments on
polarized differential multi-graded Hodge-Lefschetz modules,
which is a slight generalization
of polarized differential bigraded Hodge-Lefschetz modules
in \cite{Guillen-NavarroAznarCI}
(cf. \cite[Section 4]{SaitoMorihikoMHP}).
As already mentioned above,
we prove Theorems \ref{thm:3}, \ref{thm:4}, \ref{thm:2}, and \ref{thm:5}
all together in Section \ref{sec:proofs-main-theorems}
by applying Proposition \ref{prop:1}
to $V_{\bR}$.
\end{para}

\begin{para}
We discuss about remaining problems.
Compared to the results of Green and Griffiths mentioned in \ref{para:14},
the limiting mixed Hodge structure in Theorem \ref{thm:9}
is expected to be polarized.
To this end,
we have to lift the polarization on $V_{\bC}$
to a bilinear form on $\coh^{\ast}(X, A_{\bC})$.
It seems possible
to obtain such a lifting
by following the arguments in \cite{FujisawaPLMHS}.

Theorem \ref{thm:10} states the relation between
the filtrations $L$ and $L(I)$.
It is natural to consider the relation between $L(I)$ and $L(J)$
for $J \subset I \subset \ski$
as in the theory of nilpotent orbits in several variables
(see e.g. \cite{CKS}).
Namely,
$L(I)$ is expected to be the monodromy weight filtration of $N_I(c_I)$
relative to the filtration $L(J)$ for every $c_I \in (\bpR)^I$.
Furthermore,
Theorems \ref{thm:8} and \ref{thm:10}
show that the limiting mixed Hodge structure
$(\coh^n(X, A), L, F)$
equipped with the nilpotent endomorphisms $N_1, \dots, N_k$
satisfy a part of good properties
of a nilpotent orbit in several variables.
Thus it is hoped that 
$(\coh^n(X, A_{\bC}), F, N_1, \dots, N_k)$
generates a nilpotent orbit
in $k$-variables.
If this is the case,
then we can prove
that a projective \ssls degeneration
yields polarized log Hodge structures
on the log point $(\ast,\bN^k)$.
This partially generalize the result in \cite{FN2020}
to the case over a base with higher log rank.

Since a \ssls degeneration
is a special case of a log smooth degeneration
defined in \cite[Definition 4.3]{FujisawaMHSLSD},
it is already proved that
the relative log de Rham cohomology groups
of a projective \ssls degeneration
carries a $\bQ$-mixed Hodge structure
if all the irreducible components of $X$ are smooth.
Although the two constructions,
one in \cite{FujisawaMHSLSD} and the other in this paper,
are rather different,
it should be proved that
these two mixed Hodge structures
are the same.
\end{para}

\begin{ack}
The author would like to thank Y. Nakkajima
for stimulating and helpful discussion.
The author was partially supported
by JSPS KAKENHI Grant Number JP16K05107.
\end{ack}

\section{Preliminaries}
\label{sec:preliminaries}

\begin{para}
The cardinality of a finite set $A$
is denoted by $|A|$.
\end{para}

\begin{para}
The set of the positive integers (resp. the positive real numbers)
is denoted by $\bpZ$ (resp. $\bpR$).
\end{para}

\begin{para}
For two sets $A$ and $B$,
the set of all maps from $A$ to $B$
is denoted by $B^A$.
\end{para}

\begin{para}
\label{para:7}
Let $A$ be a finite set.
Then $\bZ^A$ is a free $\bZ$-module of rank $|A|$,
whose canonical $\bZ$-basis is denoted by $\{\ve_a\}_{a \in A}$.
For a subset $B \subset A$,
we have the canonical direct sum decomposition
$\bZ^A=\bZ^B \oplus \bZ^{A \setminus B}$,
which induces the canonical surjection
$\bZ^A \longrightarrow \bZ^B$.
For an element $\vq \in \bZ^A$,
its image by this canonical surjection
is denoted by $\vq_B \in \bZ^B$.
We set $\ve=\sum_{a \in A}\ve_a \in \bZ^A$.
Then $\ve_B=\sum_{a \in B}\ve_a \in \bZ^B$.
For $\vq=\sum_{a \in A}q_a\ve_a \in \bZ^A$,
we set $|\vq|=\sum_{a \in A}q_a \in \bZ$.
For the case of $A=\ski$,
we use $\bZ^k$ instead of $\bZ^A$.
As usual, we write $\vq=(q_1,q_2, \dots, q_k)$
for $\vq=\sum_{i=1}^{k}q_i\ve_i \in \bZ^k$.

A partial order $\ge$ on $\bZ^A$ is defined by
\begin{equation}
\label{eq:3}
\vq=\sum_{a \in A}q_a\ve_a \ge \vq'=\sum_{a \in A}q'_a\ve_a
\Longleftrightarrow
q_a \ge q'_a \quad \text{for all $a \in A$}.
\end{equation}
We set
$\bZ^A_{\ge \vq}=\{\vr \in \bZ^A \mid \vr \ge \vq\}$
for $\vq \in \bZ^A$.
For the case of $\vq=\vo$,
we use $\bN^A=\bZ^A_{\ge \vo}$.
Then $\bN^A$ is a monoid admitting a direct sum decomposition
$\bN^A=\bigoplus_{a \in A}\bN \ve_a$ as monoids.
\end{para}

\begin{para}
Let $\Lambda$ be a finite set.
We sometimes use
$\ula, \umu, \unu, \dots$ for subsets of $\Lambda$.
The set of all subsets of $\Lambda$ 
is denoted by $S(\Lambda)$.
For the case where a partition into disjoint union
\begin{equation}
\label{eq:18}
\Lambda=\coprod_{i=1}^k\Lambda_i
\end{equation}
is given,
we set
\begin{equation}
S_{\vr}(\Lambda)
=\{\ula \in S(\Lambda); |\ula \cap \Lambda_i|=r_i
\text{ for all $i=\ki$}\}
\end{equation}
for $\vr=(r_i)_{i=1}^k \in \bZ^k$.
Note that $|\ula|=|\vr|$
for $\ula \in S_{\vr}(\Lambda)$.
\end{para}


\begin{para}
\label{para:13}
For a finite set $\Lambda$,
we set
$\varepsilon(\Lambda)=\bigwedge^{|\Lambda|}\bZ^{\Lambda}$,
which is a free $\bZ$-module of rank one.
We note $\varepsilon(\emptyset)=\bZ$ by definition.
Moreover, we set
$\bigwedge\bZ^{\Lambda}
=\bigoplus_{m \ge 0}\bigwedge^m\bZ^{\Lambda}$.
Then the equality
\begin{equation}
\label{eq:17}
\bigwedge \bZ^{\Lambda}
=\bigoplus_{\ula \in S(\Lambda)}\varepsilon(\ula)
\end{equation}
holds.
\end{para}

\begin{para}[Two products $\chi$ and $\overline{\chi}$ on $\bigwedge \bZ^\Lambda$]
Let $\Lambda$ be a finite set.
A morphism
$\chi(\Lambda)
\colon \bigwedge\bZ^{\Lambda} \otimes_{\bZ} \bigwedge\bZ^{\Lambda}
\longrightarrow
\bigwedge\bZ^{\Lambda}$
is defined by
$\chi(\Lambda)(\vv \otimes \vw)=\vv \wedge \vw$
for $\vv, \vw \in \bigwedge\bZ^{\Lambda}$.
Via the direct sum decomposition \eqref{eq:17},
the restriction of $\chi(\Lambda)$
on the direct summand $\varepsilon(\ula) \otimes_{\bZ} \varepsilon(\umu)$
induces an isomorphism
\begin{equation}
\label{eq:44}
\chi(\ula, \umu) \colon
\varepsilon(\ula) \otimes_{\bZ} \varepsilon(\umu)
\overset{\simeq}{\longrightarrow}
\varepsilon(\ula \cup \umu)
\end{equation}
if $\ula \cap \umu = \emptyset$.
Similarly, an isomorphism
\begin{equation}
\label{eq:97}
\ve_{\lambda} \wedge \colon
\varepsilon(\umu) \longrightarrow \varepsilon(\{\lambda\} \cup \umu)
\end{equation}
is defined by sending $\vv \in \varepsilon(\umu)$ to
$\ve_{\lambda} \wedge \vv \in \varepsilon(\{\lambda\} \cup \umu)$
for $\lambda \in \Lambda \setminus \umu$.

Now, we consider the case where $\Lambda$
is equipped with a partition \eqref{eq:18}.
For $\ula, \umu \subset \Lambda$
with $|\ula \cap \umu \cap \Lambda_i|=1$
for all $i=\ki$,
we set $\{\lambda_i\}=\ula \cap \umu \cap \Lambda_i$
for each $i$,
and obtain an isomorphism
$(\ve_{\lambda_k}\wedge)^{-1} \cdots (\ve_{\lambda_1}\wedge)^{-1}
\colon
\varepsilon(\umu)
\longrightarrow
\varepsilon(\umu \setminus \{\lambda_1, \dots, \lambda_k\})$,
where $\ve_{\lambda_i}\wedge$ is the isomorphism \eqref{eq:97}.
Then a morphism
$\overline{\chi}(\ula, \umu) \colon
\varepsilon(\ula) \otimes \varepsilon(\umu)
\longrightarrow
\bigwedge\bZ^{\Lambda}$
is defined by
\begin{equation}
\overline{\chi}(\ula, \umu)=
\chi(\ula, \umu \setminus \{\lambda_1, \lambda_2, \dots, \lambda_k\})
\cdot (\id \otimes (\ve_{\lambda_k}\wedge)^{-1}
\cdots (\ve_{\lambda_1}\wedge)^{-1}).
\end{equation}
For the case where $|\ula \cap \umu \cap \Lambda_i| \not= 1$
for some $i \in \ski$,
we set
$\overline{\chi}(\ula, \umu)=0$
as a morphism from $\varepsilon(\ula) \otimes \varepsilon(\umu)$
to $\bigwedge\bZ^{\Lambda}$.
Thus we obtain a morphism
\begin{equation}
\label{eq:45}
\overline{\chi}(\Lambda)
=\bigoplus\overline{\chi}(\ula, \umu) \colon
\bigwedge\bZ^{\Lambda} \otimes \bigwedge\bZ^{\Lambda}
\longrightarrow
\bigwedge\bZ^{\Lambda}
\end{equation}
via the direct sum decomposition \eqref{eq:17}.
For $\vv \in \bigwedge^p\bZ^{\Lambda}, \vw \in \bigwedge^q\bZ^{\Lambda}$,
the equality
\begin{equation}
\label{eq:20}
\overline{\chi}(\Lambda)(\vw \otimes \vv)
=(-1)^{(p-k)(q-k)}\overline{\chi}(\Lambda)(\vv \otimes \vw)
\end{equation}
can be easily checked.
\end{para}

\begin{rmk}
\label{rmk:2}
Let $\Lambda$ be as above
and $\Gamma \subset \Lambda$.
Then $\Gamma$ has a partition $\Gamma=\coprod_{i=1}\Gamma \cap \Lambda_i$.
For the morphisms $\overline{\chi}(\Lambda)$
and $\overline{\chi}(\Gamma)$ defined above,
the diagram
\begin{equation}
\begin{CD}
\bigwedge \bZ^{\Lambda} \otimes \bigwedge \bZ^{\Lambda}
@>{\overline{\chi}(\Lambda)}>>
\bigwedge \bZ^{\Lambda} \\
@VVV @VVV \\
\bigwedge \bZ^{\Gamma} \otimes \bigwedge \bZ^{\Gamma}
@>{\overline{\chi}(\Gamma)}>>
\bigwedge \bZ^{\Gamma},
\end{CD}
\end{equation}
is commutative,
where the vertical arrows
are the morphisms induced from the canonical surjection
$\bZ^{\Lambda} \longrightarrow \bZ^{\Gamma}$
in \ref{para:7}.
\end{rmk}

\subsection*{Finitely generated free monoids}

\begin{defn}[A finitely generated free monoid]
\label{notn:3}
In this paper,
a monoid $P$ is called a finitely generated free monoid
if there exists an isomorphism of monoids
$P \simeq \bN^{\Lambda}$ for some {\itshape finite} set $\Lambda$.
\end{defn}

\begin{rmk}
\label{rmk:1}
In the situation above,
the finite set $\Lambda$
is uniquely determined by $P$
up to the unique isomorphism
in the following sense.
Let $\Lambda$ and $\Gamma$ be finite sets,
and $\xi_1 \colon P \overset{\simeq}{\longrightarrow} \bN^{\Lambda}$
and $\xi_2 \colon P \overset{\simeq}{\longrightarrow} \bN^{\Gamma}$
isomorphisms of monoids.
Then there exists a unique bijection
$\sigma \colon \Lambda \longrightarrow \Gamma$
such that
$(\xi_2 \cdot \xi_1^{-1})(\ve_{\lambda})=e_{\sigma(\lambda)}$
for all $\lambda \in \Lambda$.
\end{rmk}

\begin{defn}[The canonical bilinear form on a finitely generated free monoid]
\label{defn:9}
Let $P$ be a finitely generated free monoid.
Fix an isomorphism $\xi \colon P \overset{\simeq}{\longrightarrow} \bN^{\Lambda}$
for a finite set $\Lambda$.
Then $\xi$ induces an isomorphism $\xi\gp \colon P\gp \simeq \bZ^{\Lambda}$.
On $\bZ^{\Lambda}$, there exists the canonical bilinear form
$(\ ,\ ) \colon \bZ^{\Lambda} \otimes_{\bZ} \bZ^{\Lambda} \longrightarrow \bZ$
defined by $(\ve_{\lambda}, \ve_{\lambda})=1$
and $(\ve_{\lambda}, \ve_{\mu})=0$ for $\lambda \not= \mu$.
Via the isomorphism $\xi\gp$ above,
a symmetric bilinear form
$P\gp \otimes_{\bZ} P\gp \longrightarrow \bZ$ is induced.
By Remark \ref{rmk:1} above,
this bilinear form is independent
of the isomorphism $\xi$.
This bilinear form
$P\gp \otimes_{\bZ} P\gp \longrightarrow \bZ$
is called the canonical bilinear form associated to $P$.
Trivially, the induced bilinear form on $\bR \otimes_{\bZ} P\gp$
is symmetric and positive definite.
\end{defn}

\begin{defn}[A semistable morphism to a finitely generated free monoid]
\label{defn:1}
Let $\Lambda$ be a finite set.
A morphism of monoids
$\varphi \colon \bN^k \longrightarrow \bN^\Lambda$
is said to be semistable
if there exists a partition
$\Lambda=\coprod_{i=1}^k\Lambda_i$
such that
$\varphi(\ve_i)=\sum_{\lambda \in \Lambda_i}\ve_{\lambda}$
for all $i=\ki$.
The partition $\Lambda=\coprod_{i=1}^k\Lambda_i$
is called the partition associated to $\varphi$.
More generally, a morphism of monoids
$\varphi \colon \bN^k \longrightarrow P$
to a finitely generated free monoid $P$
is said to be semistable
if there exist a finite set $\Lambda$
and an isomorphism
$\xi \colon P \overset{\simeq}{\longrightarrow} \bN^{\Lambda}$
such that the composite
$\xi \cdot \varphi$
is semistable in the sense defined above.
\end{defn}

\begin{rmk}
\label{rmk:3}
For a semistable morphism $\varphi \colon \bN^k \longrightarrow P$,
the finite set $\Lambda$ equipped with the partition
$\Lambda=\coprod_{i=1}^k\Lambda_i$ above
is uniquely determined by $\varphi$ in the following sense.
Let $\Lambda$ and $\Gamma$ be finite sets,
and $\xi_1 \colon P \overset{\simeq}{\longrightarrow} \bN^{\Lambda}$
and $\xi_2 \colon P \overset{\simeq}{\longrightarrow} \bN^{\Gamma}$
isomorphisms
such that
$\xi_1 \cdot \varphi$ and $\xi_2 \cdot \varphi$
are semistable.
Then the bijection
$\sigma \colon \Lambda \longrightarrow \Gamma$
in Remark \ref{rmk:1}
preserves the partitions of $\Lambda$ and $\Gamma$
associated to $\xi_1 \cdot \varphi$ and $\xi_2 \cdot \varphi$
respectively.
\end{rmk}

\begin{defn}[The direct sum decomposition associated to a semistable morphism]
\label{defn:8}
Let $P$ be a finitely generated free monoid
and $\varphi \colon \bN^k \longrightarrow P$
a semistable morphism.
Take a finite set $\Lambda$,
an isomorphism $\xi \colon P \longrightarrow \bN^{\Lambda}$
and a partition $\Lambda=\coprod_{i=1}^k\Lambda_i$
associated to $\xi \cdot \varphi$
as in Definition \ref{defn:1}.
Then a finitely generated free monoid
$P_i=\xi^{-1}(\bN^{\Lambda_i})$
is independent of the choice of $\xi$
by the remark above.
Thus we obtain a direct sum decomposition of monoids
$P=\bigoplus_{i=1}^kP_i$,
called the decomposition associated to $\varphi$.
\end{defn}

\begin{defn}[The product $\overline{\chi}$ associated to a semistable morphism]
\label{defn:7}
Let $P$ be a finitely generated free monoid
and $\varphi \colon \bN^k \longrightarrow P$
a semistable morphism.
Take $\xi \colon P \longrightarrow \bN^{\Lambda}$
and $\Lambda=\coprod_{i=1}^k\Lambda_i$ as above.
Via the isomorphism
$\bigwedge \xi\gp:
\bigwedge P\gp \overset{\simeq}{\longrightarrow} \bigwedge\bZ^{\Lambda}$,
the morphism $\overline{\chi}(\Lambda)$ in \eqref{eq:45}
gives us a morphism
$\bigwedge P\gp \otimes_{\bZ} \bigwedge P\gp \longrightarrow \bigwedge P\gp$
which is independent of $\xi$
by Remark \ref{rmk:3}.
This morphism is denoted by
$\overline{\chi}(\varphi)$.
\end{defn}

\subsection*{Filtered complexes}

\begin{notn}[Finiteness for filtrations]
Because we mainly use finite filtrations in this paper,
we usually omit the adjective ``finite''
for filtrations.
\end{notn}

\begin{notn}[Spectral sequences]
We follow the notation in \cite[(1.3.1)]{DeligneII}
for the spectral sequence associated to a filtered complex.
Let $(K_1, F)$ and $(K_2, F)$ be decreasingly filtered complexes.
A morphism
$f \colon (K_1, F) \longrightarrow (K_2, F)$
in the filtered derived category
induces a morphism of spectral sequences
$E_r^{p,q}(K_1, F) \longrightarrow E_r^{p,q}(K_2, F)$,
denoted by $E_r^{p,q}(f)$,
for all $p,q$ and for all $r$ with $1 \le r \le \infty$.
We often use $E_r(f)$ instead of $E_r^{p,q}(f)$ for short.
The morphism $E_{\infty}^{p,q}(f)$ coincides with
$\gr_F^p\coh^{p+q}(f)$
via the isomorphisms
$E_{\infty}^{p,q}(K_i, F) \simeq \gr_F^p\coh^{p+q}(K_i)$
for $i=1, 2$.
\end{notn}

\begin{para}[Tensor product of complexes]
\label{para:6}
For two complexes $K$ and $L$,
the differential of the complex $K \otimes L$
is given by
$d=d \otimes \id +(-1)^p\id \otimes d$
on the direct summand $K^p \otimes L^q$ of
$(K \otimes L)^{p+q}$.
An identification
$K \otimes L
\overset{\simeq}{\longrightarrow}
L \otimes K$
is given by
$x \otimes y \mapsto (-1)^{pq} y \otimes x$
on $K^p \otimes L^q$
as in \cite[p. 11]{Conrad}.
For $a, b \in \bZ$,
an identification
$K[a] \otimes L[b] \overset{\simeq}{\longrightarrow} (K \otimes L)[a+b]$
is given by
\begin{equation}
\label{eq:80}
x \otimes y \mapsto (-1)^{pb} x \otimes y
\end{equation}
on $K[a]^p \otimes L[b]^q =K^{p+a} \otimes L^{q+b}$
as in \cite[(1.3.6)]{Conrad}.
\end{para}

\begin{defn}
\label{defn:10}
For two complexes $K_1$ and $K_2$,
a morphism
$\coh^a(K_1) \otimes \coh^b(K_2) \longrightarrow \coh^{a+b}(K_1 \otimes K_2)$
is canonically induced
for all $a, b \in \bZ$.
For a morphism of complexes
$f \colon K_1 \otimes K_2 \longrightarrow K_3$,
the composite
\begin{equation}
\coh^a(K_1) \otimes \coh^b(K_2)
\longrightarrow
\coh^{a+b}(K_1 \otimes K_2)
\xrightarrow{\coh^{a+b}(f)}
\coh^{a+b}(K_3)
\end{equation}
is denoted by $\coh^{a,b}(f)$ in this paper.
For the case where $K_1, K_2, K_3$ are complexes of abelian sheaves
on a topological space $X$,
morphisms
$\coh^a(X, K_1) \otimes \coh^b(X, K_2)
\longrightarrow \coh^{a+b}(X, K_1 \otimes K_2)$
and
$\coh^{a,b}(X,f) \colon
\coh^a(X, K_1) \otimes \coh^b(X,K_2)
\longrightarrow \coh^{a+b}(X, K_3)$
are defined similarly.
\end{defn}

\begin{defn}[Filtration on the tensor product]
\label{defn:11}
Let $(K_1, F)$ and $(K_2, F)$
be two decreasingly filtered complexes.
A decreasing filtration $F$ on $K_1 \otimes K_2$
is defined by
\begin{equation}
F^r(K_1^p \otimes K_2^q)
=\sum_{a+b=r}
\image(F^aK_!^p \otimes F^bK_2^q \longrightarrow K_1^p \otimes K_2^q)
\end{equation}
for all $p,q \in \bZ$.
There exists the canonical morphism
$\gr_F^aK_1 \otimes \gr_F^bK_2 \longrightarrow \gr_F^{a+b}(K_1 \otimes K_2)$
for $a, b \in \bZ$.
For a morphism of filtered complexes
$f \colon (K_1 \otimes K_2, F) \longrightarrow (K_3, F)$,
the composite
\begin{equation}
\gr_F^aK_1 \otimes \gr_F^bK_2
\longrightarrow
\gr_F^{a+b}(K_1 \otimes K_2)
\xrightarrow{\gr_F^{a+b}f}
\gr_F^{a+b}K_3
\end{equation}
is denoted by $\gr_F^{a,b}f$.
\end{defn}

\begin{defn}
\label{defn:12}
For two decreasingly filtered complex $(K_1, F)$ and $(K_2, F)$,
a morphism
\begin{equation}
\rho_r^{a,b,c,d} \colon
E_r^{a,b}(K_1, F) \otimes E_r^{c,d}(K_2, F)
\longrightarrow
E_r^{a+c,b+d}(K_1 \otimes K_2, F)
\end{equation}
is canonically induced for all $0 \le r \le \infty$
and for all $a,b,c,d \in \bZ$.
For the morphism $d_r$ of $E_r$-terms,
the equality
\begin{equation}
\label{eq:106}
d_r \cdot \rho_r^{a,b,c,d}
=\rho_r^{a+r,b-r+1,c,d}
\cdot
(d_r \otimes \id)
+(-1)^{a+b}\rho_r^{a,b,c+r,d-r+1}
\cdot
(\id \otimes d_r)
\end{equation}
holds on $E_r^{a,b}(K_1,W) \otimes E_r^{c,d}(K_2,W)$
for all $a,b,c,d \in \bZ$.
%
For a morphism $f \colon (K_1 \otimes K_2, F) \longrightarrow (K_3, F)$
in the filtered derived category,
the composite
\begin{equation}
E_r^{a,b}(K_1, F) \otimes E_r^{c,d}(K_2, F)
\xrightarrow{\rho_r^{a,b,c,d}}
E_r^{a+c,b+d}(K_1 \otimes K_2, F)
\xrightarrow{E_r(f)}
E_r^{a+c, b+d}(K_3, F)
\end{equation}
is simply denoted by $E_r(f)$ for $1 \le r \le \infty$ by abuse of notation.
\end{defn}

\begin{para}[Gysin morphism for a bifiltered complex]
\label{para:10}
Let $F$ and $G$ be two decreasing filtrations on a complex $K$.
The short exact sequence
\begin{equation}
\begin{CD}
0 @>>> \gr_G^{a+1}K
@>>> G^aK\bigl/G^{a+2}K
@>>> \gr_G^aK
@>>> 0
\end{CD}
\end{equation}
defines a morphism
\begin{equation}
\gamma_G \colon \gr_G^aK \longrightarrow \gr_G^{a+1}K[1]
\end{equation}
in the derived category for all $a \in \bZ$.
In fact, this morphism $\gamma_G$ underlies
a morphism
\begin{equation}
\gamma_G \colon (\gr_G^aK, F) \longrightarrow (\gr_G^{a+1}K[1], F),
\end{equation}
denoted by the same letter $\gamma_G$,
in the filtered derived category
because
\begin{equation}
\begin{CD}
0 @>>> \gr_F^p\gr_G^{a+1}K
@>>> \gr_F^p(G^aK\bigl/G^{a+2}K)
@>>> \gr_F^p\gr_G^aK
@>>> 0
\end{CD}
\end{equation}
is exact for all $p$.
Thus $\gamma_G$ induces a morphism of spectral sequences
\begin{equation}
\label{eq:92}
E_r(\gamma_G) \colon
E_r^{p,q}(\gr_G^aK, F) \longrightarrow E_r^{p,q+1}(\gr_G^{a+1}K,F)
\end{equation}
for $1 \le r \le \infty$
by the identification
$E_r^{p,q}(\gr_G^{a+1}K[1], F) \simeq E_r^{p,q+1}(\gr_G^{a+1}K, F)$.
Here we note that $E_r(\gamma_G)$ is anti-commutative
with $d_r$
because the morphism $d_r$ on $E_r^{p,q}(\gr_G^{a+1}K[1], F)$
is identified with the morphism $-d_r$  on $E_r^{p,q+1}(\gr_G^{a+1}K,F)$.
\end{para}

\begin{para}[Convolution of two filtrations]
\label{para:9}
Let $K$ be a complex.
For two decreasing filtrations
$F$ and $G$ on $K$,
a decreasing filtration $F \ast G$ on $K$
is defined by
\begin{equation}
(F \ast G)^pK^n
=\sum_{a+b=p}F^aK^n \cap G^bK^n
\end{equation}
for all $n,p \in \bZ$
as in \cite[(1.4) Definition]{Steenbrink-Zucker}
and \cite[Definition 1.3.1]{KashiwaraABVPHS}.
Then the canonical injection
$F^aK \cap G^bK \hookrightarrow (F \ast G)^{a+b}K$
induces an isomorphism of complexes
\begin{equation}
\bigoplus_{a+b=p}\gr_F^a\gr_G^bK
\overset{\simeq}{\longrightarrow}
\gr_{F \ast G}^pK
\end{equation}
for all $p$,
under which we have the identification
\begin{equation}
\bigoplus_{\substack{a+b=p \\ a \ge k}}\gr_F^a\gr_G^bK
\overset{\simeq}{\longrightarrow}
F^k\gr_{F \ast G}^pK
\end{equation}
for all $k$.
Thus we obtain identifications
\begin{equation}
E_1^{p,q}(K, F \ast G)
\simeq
\coh^{p+q}(\gr_{F \ast G}^pK)
\simeq
\bigoplus_{a+b=p}\coh^{p+q}(\gr_F^a\gr_G^bK)
\simeq
\bigoplus_{a+b=p}E_1^{a,b+q}(\gr_G^bK, F)
\end{equation}
for all $p,q$,
under which $F^kE_1^{p,q}(K, F \ast G)$
is identified with $\bigoplus_{a+b=p,a \ge k}E_1^{a,b+q}(\gr_G^bK, F)$.
A morphism
$d_1' \colon
E_1^{p,q}(K, F \ast G) \longrightarrow E_1^{p+1,q}(K, F \ast G)$
is defined to be a direct sum of the morphisms of $E_1$-terms
$E_1^{a,b+q}(\gr_G^bK, F) \longrightarrow E_!^{a+1,b+q}(\gr_G^bK, F)$.
Similarly, a morphism
$d_1'' \colon
E_1^{p,q}(K, F \ast G) \longrightarrow E_1^{p+1,q}(K, F \ast G)$
is defined to be a direct sum of the morphisms
$E_1(\gamma_G) \colon
E_1^{a,b+q}(\gr_G^bK, F) \longrightarrow E_1^{a,b+q+1}(\gr_G^{b+1}K, F)$
in \eqref{eq:92}.
\end{para}

The following lemma is easily checked by definition.

\begin{lem}
\label{lem:23}
For the morphism of $E_1$-terms
$d_1 \colon E^{p,q}(K, F \ast G) \longrightarrow E_1^{p+1,q}(K, F \ast G)$,
the equality
$d_1=d_1'+d_1''$
holds for all $p,q$.
\end{lem}

\begin{notn}[Decreasing filtration and increasing filtration]
A decreasing filtration $F$ induces an increasing filtration $W$
by $W_m=F^{-m}$ for all $m \in \bZ$, and vice versa.
We interchanges decreasing and increasing filtrations
by this rule.
For a decreasing filtration $F$,
we use the notation $F[n]=F^{p+n}$.
Hence, we use $W[n]_m=W_{m-n}$
for an increasing filtration $W$.
Note that this notation for the shift of an increasing filtration
coincides with the one
in \cite{DeligneII} and \cite{ElZeinbook},
and different from the one in \cite{CKS}.
\end{notn}

\subsection*{Log \ca spaces}
%

\begin{notn}
\label{notn:1}
Let $(X, \cM_X)$ be a log \ca space.
For an open subset $V \subset X$, 
the restriction $\cM_X|_V$ is denoted by $\cM_V$ for short.
The monoid sheaf $\cM_X/\cO^{\ast}_X$ is denoted by $\overline{\cM}_X$
as in \cite{OgusLogBook}.
The canonical morphism $\cM_X \longrightarrow \overline{\cM}_X$
is denoted by $\pi_X$.
The log de Rham complex of $(X, \cM_X)$
is denoted by $\Omega_X(\log \cM_X)$.

For an effective divisor $D$
on a complex manifold $X$,
a log structure $\cM_X(D)$
is defined by
$\cM_X(D)=j_{\ast}\cO^{\ast}_{X \setminus D} \cap \cO_X$,
where $j \colon X \setminus D \hookrightarrow X$ is the open immersion
(cf. \cite[(1.5)]{KazuyaKato}).
For the case where $D$ is a \nc divisor on $X$,
the log de Rham complex $\Omega_X(\log \cM_X(D))$
coincides with the usual log de Rham complex $\Omega_X(\log D)$.
\end{notn}

\begin{notn}
Let $f \colon (X, \cM_X) \longrightarrow (Y, \cM_Y)$
be a morphism of log \ca spaces.
Then the morphism of monoid sheaves
$f^{\flat} \colon f^{-1}\cM_Y \longrightarrow \cM_X$
induces a morphism of monoid sheaves
$f^{-1}\overline{\cM}_Y \longrightarrow \overline{\cM}_X$ on $X$,
denoted by $\overline{f^{\flat}}$ in this paper.
Thus a morphism of monoids
$\overline{f^{\flat}_x} \colon
\overline{\cM}_{Y,f(x)} \longrightarrow \overline{\cM}_{X,x}$
is induced for every $x \in X$.
\end{notn}

\begin{notn}
For two morphisms of log \ca spaces
$(X, \cM_X) \longrightarrow (Z, \cM_Z)$
and $(Y, \cM_Y) \longrightarrow (Z, \cM_Z)$,
we denote by
$(X, \cM_X) \times_{(Z, \cM_Z)} (Y, \cM_Y)$
the fiber product in the category of log \ca spaces.
\end{notn}

\begin{defn}[Weight filtration on log de Rham complex]
\label{defn:29}
Let $(X, \cM_X)$ be a log \ca space.
A monoid subsheaf $\cN$ with $\cO^{\ast}_X \subset \cN \subset \cM_X$
defines a log structure on $X$
by restricting the structure morphism
$\cM_X \longrightarrow \cO_X$ to $\cN$.
Then the identity map of $X$ induces a morphism of log \ca spaces
$(X, \cM_X) \longrightarrow (X, \cN)$,
which gives us the canonical morphism
of the log de Rham complexes
$\Omega_X(\log \cN) \longrightarrow \Omega_X(\log \cM_X)$.
For $m \in \bZ$, an $\cO_X$-submodule
$W(\cN)_m\Omega^n_X(\log \cM_X)$ is defined by
\begin{equation*}
W(\cN)_m\Omega^n_X(\log \cM_X)
=\image(\Omega^{n-m}_X(\log \cN) \otimes_{\cO_X}
\Omega^m_X(\log \cM_X)
\overset{\wedge}{\longrightarrow}
\Omega^n_X(\log \cM_X)),
\end{equation*}
where the morphism $\wedge$ on the right hand side
is induced from the wedge product on
$\Omega_X(\log \cM_X)$.
It is easy to see that $W(\cN)$ defines an increasing filtration
on the complex $\Omega_X(\log \cM_X)$.
By definition $W(\cM_X)$ is the trivial filtration.
For the case of $\cN=\cO^{\ast}_X$,
we use $W$ instead of $W(\cO^{\ast}_X)$.
\end{defn}


\section{Semistable log smooth degenerations}
\label{sec:semist-log-smooth}

In this section,
we first introduce the notion of a \ssls degeneration.
Then, we construct
$((A_{\bQ}, L(I), L), (A_{\bC}, L(I), L, F), \alpha)$,
which is the object to be studied throughout this paper,
for a \ssls degeneration.

\begin{notn}
Let $k$ be a positive integer.
A pre-log structure
$\beta: \bN^k \longrightarrow \bC$ over the point $(\spec\bC)_{\an}$
is given by $\beta(\vo)=1$
and $\beta(\vv)=0$ for $\vv \in \bN^k \setminus \{\vo\}$.
The log structure associated to the pre-log structure $\beta$
is $\bC^{\ast} \oplus \bN^k \longrightarrow \bC$
sending $(a, \vv) \in \bC^{\ast} \oplus \bN^k$ to $a\beta(\vv) \in \bC$.
The point equipped with
this log structure is called the $\bN^k$-log point
and simply denoted by $(\ast, \bN^k)$.
The $\bN$-log point $(\ast, \bN)$ is called the standard log point
in \cite{SteenbrinkLE}.
\end{notn}

\begin{notn}
For a finitely generated monoid $P$,
the \ca space $(\spec \bC[P])_{\an}$
carries the log structure associated to the pre-log structure
induced by the morphism $P \longrightarrow \bC[P]$.
This log \ca space is denoted by
$((\spec \bC[P])_{\an}, P)$ for short.
For a finite set $\Lambda$, 
the log \ca space $((\spec \bC[\bN^{\Lambda}])_{\an},\bN^{\Lambda})$
is simply denoted by $(\bC^{\Lambda}, \bN^{\Lambda})$.
For the case of $\Lambda=\ski$,
we use $(\bC^k, \bN^k)$ instead of $(\bC^{\Lambda}, \bN^{\Lambda})$.
We have the canonical strict closed immersion
$\iota \colon (\ast, \bN^k) \longrightarrow (\bC^k, \bN^k)$,
which sends the point $\ast$ to the origin $0 \in \bC^k$.

A morphism of finitely generated monoid
$h \colon Q \longrightarrow P$
induces a morphism of log \ca spaces
$((\spec \bC[P])_{\an}, P)
\longrightarrow
((\spec \bC[Q])_{\an}, Q)$
denoted by $\tilde{h}$ throughout this paper.
\end{notn}

\begin{defn}[Semistable log smooth degeneration]
\label{defn:19}
Let $(X, \cM_X)$ be an fs log \ca space.
A morphism of log \ca spaces
$f \colon (X, \cM_X) \longrightarrow (\ast, \bN^k)$
is called a \ssls degeneration
if the following three conditions are satisfied:
\begin{enumerate}
\item
$f$ is log smooth.
\item
$\overline{\cM}_{X,x}$ is a finitely generated free monoid
for all $x \in X$
(cf. Definition \ref{notn:3}).
\item
The morphism
$\overline{f^{\flat}_x} \colon \bN^k \longrightarrow \overline{\cM}_{X,x}$
is semistable for all $x \in X$
(cf. Definition \ref{defn:1}).
\end{enumerate}
Moreover,
a \ssls degeneration $f$
is said to be projective (resp. proper),
if $X$ is projective (resp. compact).
\end{defn}

\begin{notn}
\label{notn:4}
Let $f \colon (X, \cM_X) \longrightarrow (\ast, \bN^k)$
be a \ssls degeneration.
The relative log de Rham complex of $f$
is denoted by $\Omega_{X/\ast}(\log \cM_X/\bN^k)$
as in \cite[(1.7)]{KazuyaKato}.
The image of $\ve_i \in \Gamma(X, \bN^k_X)$
by the morphism
$f^{\flat} \colon f^{-1}\bN^k=\bN^k_X \longrightarrow \cM_X$
is denoted by $t_i \in \Gamma(X, \cM_X)$
for $i=\ki$.
This gives us a global section
$\dlog t_i \in \Gamma(X, \Omega^1_X(\log \cM_X))$.
\end{notn}

\begin{exa}
\label{exa:1}
Let $\Delta^k$ be the $k$-dimensional polydisc
with the coordinates $(t_1, \dots, t_k)$
and $g \colon \mathcal{X} \longrightarrow \Delta^k$
be a surjective morphism of complex manifolds.
Let $E$ be the divisor on $\Delta^k$ defined by $t_1 \cdots t_k$.
Assume that $D=f^{\ast}E$ is \textit{reduced} \snc divisor on $\mathcal{X}$.
Then $X=g^{-1}(0) \longrightarrow \{0\}$
underlies a \ssls degeneration
once we equip the log structures on $X$ and $\{0\}$
induced from $\cM_\mathcal{X}(D)$ and $\cM_{\Delta^k}(E)$ respectively.
Here we remark that the morphism $g$ as above
is called a semistable morphism in \cite{FujisawaLHSSV2}.
\end{exa}

The following proposition
shows that a \ssls degeneration is
locally isomorphic to the one obtained in the example above.

\begin{prop}
\label{prop:2}
Let $f \colon (X, \cM_X) \longrightarrow (\ast, \bN^k)$
be a \ssls degeneration.
For every $x \in X$,
there exist
\begin{enumerate}
\item
\label{item:2}
an open neighborhood $V$ of $x$,
\item
a finite set $\Lambda$,
\item
\label{item:3}
a semistable morphism of monoids
$\varphi \colon \bN^k \longrightarrow \bN^{\Lambda}$, and
\item
\label{item:4}
a commutative diagram of log \ca spaces
\begin{equation}
\label{eq:114}
\vcenter{
\xymatrix{
(V, \cM_V)
\ar[r]^-{(\dagger)} \ar[rd]_-{f|_V}
& (\ast, \bN^k) \times_{(\bC^k, \bN^k)} (\bC^{\Lambda}, \bN^{\Lambda})
\ar[r]^-{(\dagger \dagger)} \ar[d]
& (\bC^{\Lambda}, \bN^{\Lambda})
\ar[d]^-{\widetilde{\varphi}} \\
& (\ast, \bN^k) \ar[r]_-{\iota}
& (\bC^k, \bN^k)
}
}
\end{equation}
in which the morphism $(\dagger)$
on the top horizontal line
is strict and log smooth.
\end{enumerate}
Moreover, these data can be taken
such that the composite of $(\dagger)$ and $(\dagger \dagger)$
in \eqref{eq:114}
sends $x \in V$ to the origin of $\bC^{\Lambda}$.
\end{prop}
\begin{proof}
This is an analogue of Theorem 1.2.7
of \cite{OgusLogBook}
in the analytic context.
By definition,
there exist a finite set $\Lambda$
and an isomorphism
$\xi \colon \overline{\cM}_{X,x}
\overset{\simeq}{\longrightarrow} \bN^{\Lambda}$
such that
$\xi \cdot \overline{f^{\flat}_x}$
is semistable.
On the other hand,
$\ext^1(G, \cO_{X,x}^{\ast})=0$
for any finitely generated abelian group $G$
because $\cO_{X,x}^{\ast}$ is $n$-divisible for all $n \in \bpZ$.
Then the proof is similar to the argument in \cite{OgusLogBook}.
\end{proof}

\begin{para}[\textbf{Local description of a \ssls degeneration}]
\label{para:1}
From the proposition above,
we obtain a local description of a \ssls degeneration
$f \colon (X, \cM_X) \longrightarrow (\ast, \bN^k)$
as follows.

For any $x \in X$,
take the data in \ref{item:2}--\ref{item:4}.
Moreover,
the partition associated to $\varphi$
is denoted by $\Lambda=\coprod_{i=1}^k\Lambda_i$.
The morphism $(\dagger)$ in \ref{item:4} is smooth
in the usual sense
because it is strict and log smooth.
Therefore
by shrinking $V$ sufficiently small,
the morphism $(\dagger)$
induces an strict open immersion
\begin{equation}
(V, \cM_V)
\longrightarrow
(U, \cM_U)
=
(\ast,\bN^k)
\times_{(\bC^k,\bN^k)} (\bC^{\Lambda}, \bN^{\Lambda})
\times (\bC^l, \cO^{\ast}_{\bC^l})
\end{equation}
for some $l \in \bnnZ$.
We may assume that the natural morphism
$(V, \cM_V) \longrightarrow
(\bC^{\Lambda}, \bN^{\Lambda}) \times (\bC^l, \cO^{\ast}_{\bC^l})$
sends $x \in V$ to the origin of $\bC^{\Lambda} \times \bC^l$
because of the latter part of Proposition \ref{prop:2}.
Such $(V, \cM_V)$ (or $(U, \cM_U)$)
is called a local model of
$f \colon (X, \cM_X) \longrightarrow (\ast, \bN^k)$.

The coordinate function of $\bC^k$
corresponding to $\ve_i \in \bN^k$
is denoted by $t_i$ for $i=\ki$.
Then the log \ca space $(\bC^k,\bN^k)$
is the \ca space $\bC^k$
equipped with the log structure
associated to the divisor $E=\{t_1t_2 \cdots t_k=0\}$.
The coordinate function of $\bC^{\Lambda} \times \bC^l$
corresponding to $\ve_{\lambda} \in \bN^{\Lambda}$
is denoted by $x_{\lambda}$
and the divisor on $\bC^{\Lambda} \times \bC^l$
defined by $x_{\lambda}$ is denoted
by $D_{\lambda}$ for $\lambda \in \Lambda$.
For $i=\ki$ and $I \subset \ski$,
we set
\begin{gather}
D_i=\sum_{\lambda \in \Lambda_i}D_{\lambda} \quad (i=\ki),
\qquad
D_I=\sum_{i \in I}D_i.
\end{gather}
We use $D$ instead of $D_{\ski}$.
Then
$(\bC^{\Lambda}, \bN^{\Lambda}) \times (\bC^l, \cO^{\ast}_{\bC^l})
=(\bC^{\Lambda} \times \bC^l, \cM_{\bC^{\Lambda} \times \bC^l}(D))$
where $\cM_{\bC^{\Lambda} \times \bC^l}(D)$ denotes the log structure
associated to the divisor $D$
(cf. Notation \ref{notn:1}).
The composite of the projection
$\bC^{\Lambda} \times \bC^l \longrightarrow \bC^{\Lambda}$
and the morphism
$\widetilde{\varphi} \colon \bC^{\Lambda} \longrightarrow \bC^k$
coincides with the morphism given by
$t_i=\prod_{\lambda \in \Lambda_i}x_{\lambda}$ for $i=\ki$.
Therefore
\begin{equation}
U=\bigcap_{i=1}^k\{\prod_{\lambda \in \Lambda_i}x_{\lambda}=0\}
=\bigcap_{i=1}^k D_i
\end{equation}
by definition.
Because $\iota \colon (\ast, \bN^k) \longrightarrow (\bC^k, \bN^k)$
is strict,
$\cM_U$ coincides
with the pull-back of
$\cM_{\bC^{\Lambda} \times \bC^l}(D)$
by the closed immersion $U \hookrightarrow \bC^{\Lambda} \times \bC^l$.
Then $V$ is identified
with an open neighborhood of $0$ in $U$.
\end{para}

\begin{rmk}
\label{rmk:12}
Let $f \colon (X, \cM_X) \longrightarrow (\ast, \bN^k)$ be
a \ssls degeneration.
Then the local description above
shows that
$f$ is a log smooth degeneration
defined in \cite[Definition 4.3]{FujisawaMHSLSD}.
Moreover, we can see that
the underlying \ca space $X$ is locally isomorphic to
a product of normal crossing varieties
as in \cite[(\RomI.2)]{GreenGriffithsDTLMHS}.
\end{rmk}

\begin{lem}
\label{lem:1}
For a \ssls degeneration
$f \colon (X, \cM_X) \longrightarrow (\ast, \bN^k)$,
there exists a unique monoid sheaf $\cM(i)_X$
with $\cO^{\ast}_X \subset \cM(i)_X \subset \cM_X$
for every $i=\ki$,
such that the direct sum decomposition
of $\overline{\cM}_{X,x}$ associated
to $\overline{f^{\flat}_x}$
\textup{(}cf. Definition \textup{\ref{defn:8})}
is given by
$\overline{\cM}_{X,x}=\bigoplus_{i=1}^k\overline{\cM(i)}_{X,x}$
for all $x \in X$.
\end{lem}
\begin{proof}
The uniqueness is clear.
Therefore we may assume that
$(X, \cM_X)$ is an open neighborhood of the origin of a local model
$(U, \cM_U)$ as in \ref{para:1}.
Then the pull-back of the log structure
$\cM_{\bC^{\Lambda} \times \bC^l}(D_i)$
by the closed immersion $U \hookrightarrow \bC^{\Lambda} \times \bC^l$
gives us the desired monoid sheaf $\cM(i)_X$ for $i=\ki$.
\end{proof}

\begin{rmk}
\label{rmk:4}
We have
$\overline{\cM}_X=\bigoplus_{i=1}^k\overline{\cM(i)}_X$
by definition.
Moreover, $t_i$ in Notation \ref{notn:4}
is contained in $\Gamma(X, \cM(i)_X)$,
because $\overline{f^{\flat}_x}(\ve_i) \in \overline{\cM(i)}_{X,x}$
for all $x \in X$.
\end{rmk}

\begin{defn}
For $I \subset \ski$,
a monoid subsheaf $\cM(I)_X$ of $\cM_X$
is defined by
$\cM(I)_X=\pi_X^{-1}(\bigoplus_{i \in I}\overline{\cM(i)}_X)$.
We set $\cM(\emptyset)_X=\cO^{\ast}_X$.
\end{defn}

\begin{defn}
\label{defn:3}
Let $I \subset \ski$.
By setting $J=\ski \setminus I$,
the monoid sheaf $\cM(J)_X$
satisfies the condition
$\cO_X^{\ast} \subset \cM(J)_X \subset \cM_X$
and gives us the filtration $W(\cM(J)_X)$ on $\Omega_X(\log \cM_X)$
as in Definition \ref{defn:29}.
We denote it by $W(I)$ for short.
By definition $W(\ski)$ coincides with $W$
in Definition \ref{defn:29}.
We use $W(i)$ instead of $W(\{i\})$ for $i=\ki$.
The properties
\begin{equation}
\label{eq:104}
\dlog t_i \wedge W(I)_m\Omega^n_X(\log \cM_X)
\subset
\begin{cases}
W(I)_{m+1}\Omega^{n+1}_X(\log \cM_X)
&\quad \text{if } i \in I \\
W(I)_m\Omega^{n+1}_X(\log \cM_X)
&\quad \text{if } i \notin I
\end{cases}
\end{equation}
can be easily seen
from the fact $t_i \in \Gamma(X, \cM(i)_X)$.
\end{defn}

\begin{para}[\textbf{Local description of the log de Rham complex}]
We consider a local model $(U, \cM_U)$
and use the notation in \ref{para:1}.
Then the log structure $\cM(I)_U$
coincides with the pull-back of
$\cM_{\bC^{\Lambda} \times \bC^l}(D_I)$
by the closed immersion $U \hookrightarrow \bC^{\Lambda} \times \bC^l$.
On the other hand
\begin{equation}
\label{eq:10}
\Omega^n_U(\log \cM_U)
\simeq
\cO_U
\otimes_{\cO_{\bC^{\Lambda} \times \bC^l}}
\Omega^n_{\bC^{\Lambda} \times \bC^l}(\log D)
\end{equation}
for every $n$
by Lemma (3.6)(2) of \cite{KatoNakayama}.
Via the identification above,
$W(I)_m\Omega^n_U(\log \cM_U)$
coincides with the image of 
$\cO_U
\otimes_{\cO_{\bC^{\Lambda} \times \bC^l}}
W(D_I)_m\Omega^n_{\bC^{\Lambda} \times \bC^l}(\log D)$
in
$\cO_U
\otimes_{\cO_{\bC^{\Lambda} \times \bC^l}}
\Omega^n_{\bC^{\Lambda} \times \bC^l}(\log D)$
for all $m$.
\end{para}



\begin{para}
\label{para:11}
A rational structure on $\Omega_X(\log \cM_X)$
can be constructed by using the Koszul complex
as in \cite{SteenbrinkLE}, \cite{FujisawaLHSSV}, \cite{FujisawaMHSLSD}
and \cite{Peters-SteenbrinkMHS}.
Here, we make a list of 
definitions and elementary properties
about Koszul complexes,
which will be used throughout this paper.
The main reference is Sections 1 and 2 of \cite{FujisawaMHSLSD}.
Let $(X, \cM_X)$ be a log \ca space.
\begin{enumerate}
\item
A complex of $\bQ$-sheaves $\kos_X(\cM_X)$
is defined in \cite[(2.3)]{FujisawaMHSLSD}.
For the case of $\cM_X=\cM_X(D)$ as in Notation \ref{notn:1},
we use $\kos_X(D)$ instead of $\kos_X(\cM_X(D))$.
\item
For a morphism of log \ca spaces
$f \colon (X, \cM_X) \longrightarrow (Y, \cM_Y)$,
there exists the canonical morphism
$f^{-1}\kos_Y(\cM_Y) \longrightarrow \kos_X(\cM_X)$
of the complexes of $\bQ$-sheaves.
\item
\label{item:5}
A morphism of complexes of $\bQ$-sheaves
$\psi_{(X, \cM_X)} \colon \kos_X(\cM_X) \longrightarrow \Omega_X(\log \cM_X)$
is defined in \cite[(2.4)]{FujisawaMHSLSD}.
For the case of $\cM_X=\cM_X(D)$,
we use $\psi_{(X, D)}$ instead of $\psi_{(X, \cM_X(D))}$.
Moreover, we use $\psi_X$ instead of $\psi_{(X,\cM_X)}$
if there is no danger of confusion.
\item
\label{item:41}
For the case of trivial log structure $\cM_X=\cO^{\ast}_X$,
there exists a quasi-isomorphism
$\bQ_X \longrightarrow \kos_X(\cO^{\ast}_X)$
such that the diagram
\begin{equation}
\begin{CD}
\bQ_X @>>> \kos_X(\cO^{\ast}_X) \\
@VVV @VV{\psi_{(X, \cO^{\ast}_X)}}V \\
\cO_X @>>> \Omega_X
\end{CD}
\end{equation}
is commutative
(cf. \cite[Lemma 3.12]{FujisawaPLMHS}).
\end{enumerate}
For a \ssls degeneration
$f \colon (X, \cM_X) \longrightarrow (\ast, \bN^k)$,
we have the following:
\begin{enumerate}[resume]
\item
For every $I \subset \ski$,
a finite increasing filtration $W(I)$ on $\kos_X(\cM_X)$
is defined as
$W(\cM(J)_{X,\bQ}\gp)$ in \cite[Definition 1.8]{FujisawaMHSLSD},
where $J=\ski \setminus I$.
\item
\label{item:6}
The morphism $\psi_X$ in \ref{item:5}
preserves the filtration $W(I)$.
\item
\label{item:7}
A morphism of complexes of $\bQ$-sheaves
$t_i \wedge \colon \kos_X(\cM_X) \longrightarrow \kos_X(\cM_X)[1]$
is defined in \cite[(1.11)]{FujisawaMHSLSD}.
Then
$(t_i \wedge) \cdot (t_J \wedge)+(t_j \wedge) \cdot (t_i \wedge)=0$
for all $i,j \in \ski$
(cf. \cite[(3.29)]{FujisawaMHSLSD}).
For $I \subset \ski$,
\begin{equation}
\label{eq:65}
(t_i \wedge)(W(I)_m\kos_X(\cM_X))
\subset
\begin{cases}
W(I)_{m+1}\kos_X(\cM_X)[1]
&\quad \text{if $i \in I$} \\
W(I)_m\kos_X(\cM_X)[1]
&\quad \text{if $i \notin I$}
\end{cases}
\end{equation}
for all $m$
(cf. \cite[(1.12)]{FujisawaMHSLSD}).
Moreover, the diagram
\begin{equation}
\label{eq:66}
\begin{CD}
\kos_X(\cM_X) @>{t_i \wedge}>> \kos_X(\cM_X)[1] \\
@V{\psi_X}VV @VV{(2\pi\sqrt{-1})\psi_X}V \\
\Omega_X(\log \cM_X) @>>{\dlog t_i \wedge}> \Omega_X(\log \cM_X)[1]
\end{CD}
\end{equation}
is commutative
(cf. \cite[3.13]{FujisawaPLMHS}).
\end{enumerate}
\end{para}

\begin{para}[\textbf{The stalk of the Koszul complex}]
Here we look at $\kos_X(\cM_X)$ and $\psi_X$ stalkwise.
It is enough to consider the origin $x=0$
of a local model $(U, \cM_U)$
in \ref{para:1}.
We use the same notation as in \ref{para:1}.
In particular,
the partition $\Lambda=\coprod_{i=1}^k\Lambda_i$ is given.
The global sections $x_{\lambda} \in \Gamma(U, \cM_U)$
for all $\lambda \in \Lambda$
gives us a decomposition
$\cM_{U,x}=\cO^{\ast}_{U,x} \oplus \bN^{\Lambda}$ at the origin $x$.
Then
\begin{equation}
\label{eq:23}
\kos_U(\cM_U)^n_x
\simeq
\bigoplus_{a \in \bZ}
\bigwedge^a \bZ^{\Lambda} \otimes_{\bZ} \kos(\cO^{\ast}_{U,x})^{n-a}
\simeq
\bigoplus_{\ula \in S(\Lambda)}
\varepsilon(\ula)
\otimes_{\bZ}
\kos(\cO^{\ast}_{U,x})^{n-|\ula|}
\end{equation}
for all $n \in \bZ$
by the definition of $\kos_U(\cM_U)$
and by \eqref{eq:17}.
Under the identification above,
\begin{equation}
\label{eq:109}
W(I)_m\kos_U(\cM_U)^n_x
\simeq
\bigoplus_{|\ula \cap \Lambda_I| \le m}
\varepsilon(\ula)
\otimes_{\bZ}
\kos(\cO^{\ast}_{U,x})^{n-|\ula|}
\end{equation}
for all $m$, where $\Lambda_I=\coprod_{i \in I}\Lambda_i$.
Therefore
$W(I)_m\kos_U(\cM_U)^n_x=\kos_U(\cM_U)^n_x$
if $m \ge |\Lambda_I|$.
Via the identification \eqref{eq:23},
the restriction of $\psi_{U,x}$
on the dirct summand
$\varepsilon(\ula)
\otimes_{\bZ}
\kos(\cO^{\ast}_{U,x})^{n-|\ula|}$
is given by
\begin{equation}
\ve_{\lambda_1} \wedge \ve_{\lambda_2} \wedge
\cdots \wedge \ve_{\lambda_p} \otimes \eta
\mapsto
(2\pi\sqrt{-1})^{-p}\dlog x_{\lambda_1} \wedge \dlog x_{\lambda_2} \wedge
\cdots \wedge \dlog x_{\lambda_p}
\wedge \psi_{(U,\cO^{\ast}_U),x}(\eta),
\end{equation}
where $p=|\ula|$ and $\ula=\{\lambda_1, \lambda_2, \dots, \lambda_p\}$.
Note that $\psi_{(U,\cO^{\ast}_U),x}(\eta) \in \Omega^{n-p}_{X,x}$.
\end{para}

\begin{notn}
Throughout this paper,
the polynomial rings $\bQ[u_1, u_2, \dots, u_k]$
and $\bC[u_1, u_2, \dots, u_k]$
are simply denoted by $\bQ[\vu]$ and $\bC[\vu]$ respectively.
We use the multi-index notation as usual.
\end{notn}

\begin{defn}
We set
$d_0=\id \otimes d \colon
\bC[\vu] \otimes_{\bC} \Omega^n_X(\log \cM_X)
\longrightarrow
\bC[\vu] \otimes_{\bC} \Omega^{n+1}_X(\log \cM_X)$
for all $n \in \bZ$.
For $i=\ki$, a morphism
$d_i \colon
\bC[\vu] \otimes \Omega^n_X(\log \cM_X)
\longrightarrow
\bC[\vu] \otimes_{\bC} \Omega^{n+1}_X(\log \cM_X)$
is defined by
\begin{equation}
d_i(P \otimes \omega)=u_iP \otimes \dlog t_i \wedge \omega
\end{equation}
for $P \in \bC[\vu]$ and $\omega \in \Omega^n_X(\log \cM_X)$.
Then these morphisms satisfy
\begin{equation}
\label{eq:25}
d_id_j+d_jd_i=0
\end{equation}
for all $i, j \in \{0,\ki\}$.
Thus a complex
$\bC[\vu] \otimes_{\bC} \Omega_X(\log \cM_X)$
of $\bC$-sheaves on $X$
is obtained
by setting $d=\sum_{i=0}^{k}d_i$.
\end{defn}

\begin{defn}
\label{defn:4}
A decreasing filtration $F$
on $\bC[\vu] \otimes_{\bC} \Omega_X(\log \cM_X)$
is defined by
\begin{equation}
F^p(\bC[\vu] \otimes_{\bC} \Omega^n_X(\log \cM_X))
=\bigoplus_{\vq \in \bN^k}
\bC \vu^{\vq} \otimes_{\bC} F^{p+|\vq|+k}\Omega^n_X(\log \cM_X)
\end{equation}
for all $n, p$,
where $F$ denotes the stupid filtration
on $\Omega_X(\log \cM_X)$.
By definition,
we have
\begin{equation}
\label{eq:1}
F^p(\bC[\vu] \otimes_{\bC} \Omega^n_X(\log \cM_X))
=\bigoplus_{|\vq| \le n-p-k}
\bC \vu^{\vq} \otimes_{\bC} \Omega^n_X(\log \cM_X)
\end{equation}
for all $n, p$.
For every $I \subset \ski$,
increasing filtrations
$W(I)$ and $L(I)$
on $\bC[\vu] \otimes_{\bC} \Omega_X(\log \cM_X)$ are defined by
\begin{gather}
W(I)_m(\bC[\vu] \otimes_{\bC} \Omega^n_X(\log \cM_X))
=\bigoplus_{\vq \in \bN^k}
\bC \vu^{\vq} \otimes_{\bC} W(I)_{m+|\vq_I|+|I|}\Omega^n_X(\log \cM_X) \\
\label{eq:2}
L(I)_m(\bC[\vu] \otimes_{\bC} \Omega^n_X(\log \cM_X))
=\bigoplus_{\vq \in \bN^k}
\bC \vu^{\vq} \otimes_{\bC} W(I)_{m+2|\vq_I|+|I|}\Omega^n_X(\log \cM_X)
\end{gather}
for all $m,n$,
where $\vq_I \in \bN^I$ denotes the image of $\vq \in \bN^k$
by the projection $\bZ^k \longrightarrow \bZ^I$.
Actually, they define subcomplexes
of $\bC[\vu] \otimes_{\bC} \Omega_X(\log \cM_X)$
because of \eqref{eq:104}.
We use
$W=W(\ski), L=L(\ski), W(i)=W(\{i\})$ and $L(i)=L(\{i\})$
for short.
\end{defn}

\begin{defn}
We set
$d_0=\id \otimes d \colon
\bQ[\vu] \otimes_{\bQ} \kos_X(\cM_X)^n
\longrightarrow
\bQ[\vu] \otimes_{\bQ} \kos_X(\cM_X)^{n+1}$.
For $i=\ki$, a morphism
$d_i \colon
\bQ[\vu] \otimes_{\bQ} \kos_X(\cM_X)^n
\longrightarrow
\bQ[\vu] \otimes_{\bQ} \kos_X(\cM_X)^{n+1}$
is defined by
\begin{equation}
d_i(P \otimes \eta)=u_iP \otimes t_i \wedge \eta
\end{equation}
for $P \in \bQ[\vu]$ and $\eta \in \kos_X(\cM_X)^n$,
where $t_i \wedge$ is the morphism in \ref{item:7}.
Then these morphisms satisfy the same equalities as \eqref{eq:25}.
Thus a complex of $\bQ$-sheaves
$\bQ[\vu] \otimes_{\bQ} \kos_X(\cM_X)$
is obtained
by setting $d=\sum_{i=0}^{k}d_i$.
For a subset $I \subset \ski$,
increasing filtrations
$W(I)$ and $L(I)$
on $\bQ[\vu] \otimes_{\bQ} \kos_X(\cM_X)$ are defined by
\begin{gather}
W(I)_m(\bQ[\vu] \otimes_{\bQ} \kos_X(\cM_X)^n)
=\bigoplus_{\vq \in \bN^k}
\bQ \vu^{\vq} \otimes_{\bQ} W(I)_{m+|\vq_I|+|I|}\kos_X(\cM_X)^n \\
L(I)_m(\bQ[\vu] \otimes_{\bQ} \kos_X(\cM_X)^n)
=\bigoplus_{\vq \in \bN^k}
\bQ \vu^{\vq} \otimes_{\bQ} W(I)_{m+2|\vq_I|+|I|}\kos_X(\cM_X)^n
\end{gather}
for all $m,n$.
By \eqref{eq:65},
these are actually increasing filtrations on the complex
$\bQ[\vu] \otimes_{\bQ} \kos_X(\cM_X)$.
We use $W=W(\ski)$, $L=L(\ski)$, $W(i)=W(\{i\})$ and $L(i)=L(\{i\})$
as in Definition \ref{defn:4}.
\end{defn}

\begin{defn}
\label{defn:5}
A morphism of $\bQ$-sheaves
$\alpha \colon \bQ[\vu] \otimes_{\bQ} \kos_X(\cM_X)^n
\longrightarrow
\bC[\vu] \otimes_{\bC} \Omega^n_X(\log \cM_X)$
is defined by
\begin{equation}
\alpha(u^{\vq} \otimes \eta)
=(2\pi\sqrt{-1})^{|\vq|+k}u^{\vq} \otimes \psi_X(\eta),
\end{equation}
which turns out to be a morphism of complexes
by the commutativity of \eqref{eq:66}.
The morphism $\alpha$ preserves
the filtrations $W(I)$ and $L(I)$
for any $I \subset \ski$ by \ref{item:6}.
\end{defn}

\begin{defn}
\label{defn:24}
Complexes of $\bC$-sheaves
$A_{\bC}$
and of $\bQ$-sheaves $A_{\bQ}$ on $X$
are defined by
\begin{gather}
A_{\bC}
=\bigl(\bC[\vu] \otimes_{\bC} \Omega_X(\log \cM_X)
/\sum_{i=1}^{k}W(i)_{-1}\bigr)[k], \\
A_{\bQ}
=\bigl(\bQ[\vu] \otimes_{\bQ} \kos_X(\cM_X)
/\sum_{i=1}^{k}W(i)_{-1}\bigr)[k].
\end{gather}
The filtrations on $A_{\bC}$ and on $A_{\bQ}$
induced by $L(I)$ on $\bC[\vu] \otimes_{\bC} \Omega_X(\log \cM_X)$
and on $\bQ[\vu] \otimes_{\bQ} \kos_X(\cM_X)$
are denoted by $L(I)$ again.
We use $L=L(\ski)$ as before.
The filtration on $A_{\bC}$
induced by $F$ on $\bC[\vu] \otimes_{\bC} \Omega_X(\log \cM_X)$
is denoted by $F$ again.
The morphism $\alpha$ in Definition \ref{defn:5}
induces a morphism of complexes
$A_{\bQ} \longrightarrow A_{\bC}$,
which is denoted by the same letter $\alpha$.
The morphism $\alpha$ preserves the filtrations $L(I)$
for any $I \subset \ski$.
\end{defn}

\begin{rmk}
\label{rmk:8}
By definition, we have
\begin{equation}
\label{eq:58}
A_{\bC}^n
=\bigoplus_{\vq \in \bN^k}
\bC \vu^{\vq} \otimes_{\bC}
\bigl(\Omega^{n+k}_X(\log \cM_X)
/\sum_{i=1}^{k}W(i)_{q_i}\bigr)
\simeq
\bigoplus_{\vq \in \bN^k}
\bigl(\Omega^{n+k}_X(\log \cM_X)
/\sum_{i=1}^{k}W(i)_{q_i}\bigr)
\end{equation}
for all $n$.
For the later use, we set
\begin{gather}
(A^n_{\bC})_{\vq}=\Omega^{n+k}_X(\log \cM_X)/\sum_{i=1}^{k}W(i)_{q_i} \\
\label{eq:16}
L(I)_m(A^n_{\bC})_{\vq}
=W(I)_{m+2|\vq_I|+|I|}
\bigl(\Omega^{n+k}_X(\log \cM_X)/\sum_{i=1}^{k}W(i)_{q_i}\bigr)
\end{gather}
for $I \subset \ski$, $\vq \in \bN^k$ and $m,n \in \bZ$.
Then we simply have
\begin{gather}
\label{eq:105}
A^n_{\bC}=
\bigoplus_{\vq \in \bN^k}
\bC\vu^{\vq} \otimes_{\bC}(A^n_{\bC})_{\vq} \\
\label{eq:50}
F^pA_{\bC}^n
=
\bigoplus_{|\vq| \le n-p}
\bC \vu^{\vq} \otimes_{\bC} (A^n_{\bC})_{\vq} \\
L(I)_mA^n_{\bC}
=
\bigoplus_{\vq \in \bN^k}
\bC\vu^{\vq} \otimes_{\bC}
L(I)_m(A^n_{\bC})_{\vq}
\end{gather}
for all $m,n,p$.
In the notation above,
we use $L$ instead of $L(\ski)$ as before.

If $X$ is of finite dimension,
then $(A_{\bC}^n)_{\vq} \not= 0$ implies
$|\vq| \le n \le \dim X$.
Therefore $F$ and $L(I)$ on $A_{\bC}$
are finite filtration for any $I \subset \ski$.
Similarly, we can check that $L(I)$ on $A_{\bQ}$ is finite
by using \eqref{eq:109}
if $X$ is of finite dimension.
\end{rmk}

\begin{assump}
In the remainder of this paper,
we assume that $X$ is of finite dimension.
\end{assump}

\begin{rmk}
For the case of $k=1$,
the bifiltered complex $(A_{\bC}, L, F)$
coincides with $(A^{\bullet}, L, F)$ in \cite[(5.3)]{SteenbrinkLE}
except for the sign of the differentials.
For the case where $f \colon (X, \cM_X) \longrightarrow (\ast, \bN^k)$
is the central fiber
of a morphism $g \colon \mathcal{X} \longrightarrow \Delta^k$
as in Example \ref{exa:1},
the complex $A_{\bC}$ is isomorphic
to the complex $sB(g)$ defined in \cite[Definition 4.3]{FujisawaLHSSV2}.
\end{rmk}

\begin{defn}
\label{defn:27}
A morphism of $\cO_X$-modules
\begin{equation}
\dlog t_1 \wedge \dlog t_2 \wedge \cdots \wedge \dlog t_k \wedge
\colon
\Omega^n_X(\log \cM_X)
\longrightarrow
\Omega^{n+k}_X(\log \cM_X)
\end{equation}
is defined by
\begin{equation}
\Omega^n_X(\log \cM_X) \ni \omega
\mapsto
\dlog t_1 \wedge \dlog t_2 \wedge \cdots \wedge \dlog t_k \wedge \omega
\in \Omega^{n+k}_X(\log \cM_X)
\end{equation}
for all $n$.
The composite of this morphism with
the inclusion
\begin{equation}
\Omega^{n+k}_X(\log \cM_X)
\simeq
\bC u^{\vo} \otimes_{\bC} \Omega^{n+k}_X(\log \cM_X)
\hookrightarrow
\bC[\vu] \otimes_{\bC} \Omega^{n+k}_X(\log \cM_X)
\end{equation}
and with the canonical surjection
$\bC[\vu] \otimes_{\bC} \Omega^{n+k}_X(\log \cM_X)
\longrightarrow
A^n_{\bC}$,
defines a morphism of $\cO_X$-modules
$\Omega^n_X(\log \cM_X)
\longrightarrow
A^n_{\bC}$,
which is compatible with the differential $d$
on the both sides.
Thus a morphism of complexes
\begin{equation}
\label{eq:73}
\Omega_X(\log \cM_X) \longrightarrow A_{\bC}
\end{equation}
is obtained.
Moreover, this morphism factors through the canonical surjection
$\Omega_X(\log \cM_X) \longrightarrow \Omega_{X/\ast}(\log (\cM_X/\bN^k))$.
Thus we obtain a morphism of filtered complexes
\begin{equation}
\label{eq:28}
\theta \colon
(\Omega_{X/\ast}(\log (\cM_X/\bN^k)), F)
\longrightarrow
(A_{\bC}, F),
\end{equation}
where $F$ on the left hand side
denotes the stupid filtration
on $\Omega_{X/\ast}(\log (\cM_X/\bN^k))$.
\end{defn}

The following lemma
shows that $(A_{\bC},F)$ is a substitute
for $(\Omega_{X/\ast}(\log (\cM_X/\bN^k)),F)$.

\begin{lem}
\label{lem:15}
The morphism $\theta$
is a filtered quasi-isomorphism.
\end{lem}
\begin{proof}
We may work in the local situation
as in \ref{para:1}.
Then we obtain the conclusion
by Lemma (3.6)(2) of \cite{KatoNakayama}
and by Corollary 4.13 of \cite{FujisawaLHSSV2}.
\end{proof}

\begin{cor}
\label{cor:3}
The morphism $\theta$
induces an isomorphism of filtered $\bC$-vector spaces
\begin{equation}
\label{eq:115}
(\coh^n(X, \Omega_{X/\ast}(\log (\cM_X/\bN^k)), F)
\overset{\simeq}{\longrightarrow}
(\coh^n(X, A_{\bC}), F)
\end{equation}
for all $n \in \bZ$.
\end{cor}

\section{Main results}
\label{sec:main-results}

In this section,
we state all the main results of this paper.

\begin{assump}
\label{assump:1}
The \ssls degeneration
$f \colon (X, \cM_X) \longrightarrow (\ast, \bN^k)$
is assumed to be projective
throughout this section.
\end{assump}

\begin{notn}
For a subset $I \subset \ski$,
we set
\begin{gather}
(A, L(I), L, F)=((A_{\bQ},L(I), L), (A_{\bC}, L(I), L, F), \alpha), \\
(\coh^n(X, A), L(I), L, F)
=((\coh^n(X,A_{\bQ}), L(I), L),
(\coh^n(X,A_{\bC}), L(I), L, F), \coh^n(X, \alpha))
\end{gather}
and
\begin{equation}
\label{eq:113}
(E_r^{p,q}(A, L(I)), L_{\rec}, F_{\rec})
=((E_r^{p,q}(A_{\bQ}, L(I)), L_{\rec}),
(E_r^{p,q}(A_{\bC}, L(I)), L_{\rec}, F_{\rec}),
E_r^{p,q}(\alpha)).
\end{equation}
For the case of $I=\ski$,
we omit $L_{\rec}$ in \eqref{eq:113}.
\end{notn}

\begin{thm}
\label{thm:1}
For all $I \subset \ski$,
the quadruple $(A, L(I), L, F)$
is a filtered cohomological $\bQ$-mixed Hodge complex on $X$
in the sense of El Zein
{\normalfont \cite[6.1.5]{ElZeinbook}}.
\end{thm}

The following is a direct consequence of the theorem above
by \cite[6.1.8 Th\'eor\`em]{ElZeinbook}.

\begin{thm}
\label{thm:6}
We have the following\textup{:}
\begin{enumerate}
\item
$(\coh^n(X,A), L[n], F)$
is a mixed Hodge structure
for all $n$.
\item
The spectral sequence
$E_r^{p,q}(A, F)$ degenerates at $E_1$-terms.
\item
The spectral sequence
$E_r^{p,q}(A, L)$ degnerates at $E_2$-terms.
\item
\label{item:9}
$(E_r^{p,q}(A, L), F_{\rec})$ is a $\bQ$-Hodge structure of weight $q$
for $r=1,2$,
and the morphism of $E_1$-terms
$d_1 \colon E_1^{p,q}(A, L) \longrightarrow E_1^{p+1,q}(A, L)$
is a morphism of $\bQ$-Hodge structures.
\end{enumerate}
Moreover, for $I \subset \ski$,
we have the following:
\begin{enumerate}[resume]
\item
\label{item:8}
The spectral sequence
$E_r^{p,q}(\gr_m^{L(I)}A, L)$
degenerates at $E_2$-terms for all $m \in \bZ$.
\item
\label{item:10}
$(E_r^{p,q}(A, L(I)), L_{\rec}[p+q], F_{\rec})$
is a $\bQ$-mixed Hodge structure
and the morphism of $E_r$-terms
$d_r \colon E_r^{p,q}(A, L(I)) \longrightarrow E_r^{p+r,q-r+1}(A, L(I))$
is a morphism of $\bQ$-mixed Hodge structures
for all $p,q,r$ with $1 \le r \le \infty$.
\item
$L(I)_m\coh^n(X, A(f))$
is a sub mixed Hodge structure of $\coh^n(X, A(f))$
for all $m$.
Moreover, we have the canonical isomorphism of mixed Hodge structures
\begin{equation}
(\gr_{-p}^{L(I)}\coh^{p+q}(X, A), L[p+q], F)
\simeq
(E_{\infty}^{p,q}(A, L(I)), L_{\rec}[p+q], F_{\rec})
\end{equation}
for all $p,q$.
\end{enumerate}
\end{thm}
%

\begin{thm}
\label{thm:3}
For any $I \subset \ski$,
the spectral sequence $E_r^{p,q}(A_{\bC}, L(I))$
degenerates at $E_2$-terms.
\end{thm}

\begin{defn}
\label{defn:14}
A morphism
$\nu_i \colon
\bC[\vu] \otimes_{\bC} \Omega^n_X(\log \cM_X)
\longrightarrow
\bC[\vu] \otimes_{\bC} \Omega^n_X(\log \cM_X)$
is defined by
$\nu_i(P \otimes \omega)=u_iP \otimes \omega$
for $i=\ki$.
This defines a morphism of complexes $\nu_i$ for $i=\ki$,
which preserves the filtrations $W(I)$ and $L(I)$
for all $I \subset \ski$.
In fact,
\begin{gather}
\nu_i(W(I)_m(\bC[\vu] \otimes_{\bC} \Omega_X(\log \cM_X)))
\subset W(I)_{m-1}(\bC[\vu] \otimes_{\bC} \Omega_X(\log \cM_X)) \\
\nu_i(L(I)_m(\bC[\vu] \otimes_{\bC} \Omega_X(\log \cM_X)))
\subset L(I)_{m-2}(\bC[\vu] \otimes_{\bC} \Omega_X(\log \cM_X))
\end{gather}
if $i \in I$.
For the filtration $F$,
we have
\begin{equation}
\nu_i(F^p(\bC[\vu] \otimes_{\bC} \Omega_X(\log \cM_X)))
\subset
F^{p-1}(\bC[\vu] \otimes_{\bC} \Omega_X(\log \cM_X))
\end{equation}
for all $p \in \bZ$.
Therefore a morphism of complexes
$\nu_i$ is induced for $i=\ki$,
which satisfies
\begin{gather}
\label{eq:9}
\nu_i(L(I)_m A_{\bC})
\subset
\begin{cases}
L(I)_{m-2} A_{\bC} \quad &\text{if $i \in I$} \\
L(I)_m A_{\bC} \quad &\text{otherwise},
\end{cases} \\
\nu_i(F^pA_{\bC}) \subset F^{p-1}A_{\bC}
\end{gather}
for all $I \subset \ski$ and $m,p \in \bZ$.
Similarly, a morphism
$\nu_i \colon \bQ[\vu] \otimes_{\bQ} \kos_X(\cM_X)^n
\longrightarrow
\bQ[\vu] \otimes_{\bQ} \kos_X(\cM_X)^n$
is defined by
$\nu_i(P \otimes \eta)=u_iP \otimes \eta$,
from which a morphism of complexes
$\nu_i \colon A_{\bQ} \longrightarrow A_{\bQ}$
is induced for $i=\ki$
These morphisms satisfy the same properties as \eqref{eq:9} for $A_{\bQ}$.
We can easily check that the diagram
\begin{equation}
\label{eq:77}
\begin{CD}
A_{\bQ} @>{\nu_i}>> A_{\bQ} \\
@V{\alpha}VV @VV{\alpha}V \\
A_{\bC} @>>{(2\pi\sqrt{-1})\nu_i}> A_{\bC}
\end{CD}
\end{equation}
is commutative.
Because of $\nu_i(L_mA_{\bC}) \subset L_{m-2}A_{\bC}$
and because $L$ is finite,
$\nu_i$ on $A_{\bC}$ is nilpotent for all $i$.
By the same reason, $\nu_i$ on $A_{\bQ}$ is also nilpotent for all $i$.
For $\vc=(c_i) \in \bC^I$,
a morphism of bifiltered complexes
$\nu_I(\vc) \colon (A_{\bC}, L, F) \longrightarrow (A_{\bC}, L[2], F[-1])$
is defined by
$\nu_I(\vc)=\sum_{i \in I}c_i\nu_i$
for $I \subset \ski$.
We use $\nu(c)$
instead of $\nu_{\ski}(c)$
for $\vc=(c_i) \in \bC^k$.
%
\end{defn}


\begin{defn}
\label{defn:26}
The morphism
\begin{equation}
\coh^q(X,\nu_i) \colon
(\coh^q(X, A_{\bC}), L, F)
\longrightarrow
(\coh^q(X, A_{\bC}), L[2], F[-1])
\end{equation}
is denoted by $N_i$ for $i=\ki$.
Moreover we set
$N_I(\vc_I)=\sum_{i \in I}c_iN_i$
for $\vc_I=(c_i)_{i \in I} \in \bC^I$.
We use $N(\vc)$ instead of $N_{\ski}(\vc_{\ski})$.
Since $\nu_i$ is nilpotent, so is $N_i$ for any $i$.
Then $N_I(c_I)$ is also nilpotent
for any $I \subset \ski$ and $c_I \in \bC^I$.
\end{defn}

\begin{rmk}
For the case where $f \colon (X, \cM_X) \longrightarrow (\ast, \bN^k)$
is the central fiber
of a morphism $g \colon \mathcal{X} \longrightarrow \Delta^k$
as in Example \ref{exa:1},
we have $N_i=(2\pi\sqrt{-1})^{-1}\log T_i$ for each $i=\ki$,
where $T_i$ denotes the monodromy automorphism
around the coordinate hyperplane $\{t_i=0\}$
(cf. \cite[(2.21) Theorem and (4.22)]{SteenbrinkLHS},
\cite[Theorem 5.19]{FujisawaLHSSV2}).
\end{rmk}
%

\begin{thm}
\label{thm:4}
For any $I \subset \ski$, $\vc_I \in (\bpR)^I$ and $l \in \bpZ$,
the morphism $N_I(\vc_I)^l$
induces an isomorphism
\begin{equation}
\gr_l^{L(I)}\coh^q(X, A_{\bC})
\overset{\simeq}{\longrightarrow}
\gr_{-l}^{L(I)}\coh^q(X, A_{\bC})
\end{equation}
for all $q \in \bZ$.
\end{thm}

\begin{thm}
\label{thm:2}
For any $I \subset \ski$, $\vc \in (\bpR)^k$
and $l \in \bpZ$,
the morphism $N(\vc)^l$ induces an isomorphism
\begin{equation}
\gr_{l+m}^L\gr_m^{L(I)}\coh^q(X, A_{\bC})
\overset{\simeq}{\longrightarrow}
\gr_{-l+m}^L\gr_m^{L(I)}\coh^q(X, A_{\bC})
\end{equation}
for all $m, q \in \bZ$.
\end{thm}

\begin{defn}
\label{notn:2}
Let $\cL$ be an invertible sheaf on $X$.
The morphism
$\dlog \colon \cO^{\ast}_X[-1] \longrightarrow \Omega_X$
induces a morphism
$\coh^2(X, \dlog) \colon
\coh^1(X, \cO^{\ast}_X) \longrightarrow \coh^2(X, \Omega_X)$.
We set $c(\cL)=\coh^2(X, \dlog)([\cL]) \in \coh^2(X, \Omega_X)$
as in \cite[(2.2.4)]{DeligneII}.
\end{defn}

\begin{defn}
\label{defn:25}
The wedge product on $\Omega_X(\log \cM_X/\bN^k)$
induces the morphism of complexes of $\bC$-sheaves
$\Omega_{X/\ast}(\log (\cM_X/\bN^k)) \otimes_{\bC} \Omega_X
\longrightarrow \Omega_{X/\ast}(\log (\cM_X/\bN^k))$.
This morphism
induces a morphism
\begin{equation}
\label{eq:107}
\coh^a(X, \Omega_{X/\ast}(\log (\cM_X/\bN^k))
\otimes_{\bC}
\coh^b(X, \Omega_X)
\longrightarrow
\coh^{a+b}(X, \Omega_{X/\ast}(\log (\cM_X/\bN^k))
\end{equation}
as in Definition \ref{defn:10}.
For $\omega \in \coh^a(X, \Omega_{X/\ast}(\log (\cM_X/\bN^k))$
and $\eta \in \coh^b(X, \Omega_X)$,
the image of $\omega \otimes \eta$ by the morphism \eqref{eq:107}
is simply denoted by
$\omega \cup \eta$
and called the cup product of $\omega$ and $\eta$.
Thus the element $c(\cL) \in \coh^2(X, \Omega_X)$
gives us a morphism
$\cup c(\cL) \colon
\coh^a(X, \Omega_{X/\ast}(\log (\cM_X/\bN^k))
\longrightarrow
\coh^{a+2}(X, \Omega_{X/\ast}(\log (\cM_X/\bN^k))$
for all $a \in \bZ$.
\end{defn}

\begin{thm}[Log hard Lefschetz theorem]
\label{thm:5}
We assume that $X$ is of pure dimension in addition.
For any ample invertible sheaf $\cL$ on $X$,
the morphism
\begin{equation}
\label{eq:108}
(\cup c(\cL))^i
\colon
\coh^{-i+\dim X}(X, \Omega_{X/\ast}(\log (\cM_X/\bN^k))
\longrightarrow
\coh^{i+\dim X}(X, \Omega_{X/\ast}(\log (\cM_X/\bN^k))
\end{equation}
is an isomorphism for all $i \in \bpZ$.
\end{thm}

\section{Proof of Theorem \ref{thm:1}}
\label{sec:proof-theor-refthm:1}

In this section,
we will prove Theorem \ref{thm:1}.
To this end,
we need to construct a residue isomorphism
as in \cite[(3.1.5)]{DeligneII}.
First, we introduce log complex manifolds $(X_{\vr}, \cM_{X_{\vr}})$
for all $\vr \in \bZ^k_{\ge \ve}$,
whose underlying complex manifolds $X_{\vr}$ are finite over $X$.
Second, we define the residue morphism
for the log de Rham complex $\Omega_X(\log \cM_X)$
in Definition \ref{defn:23},
and the residue morphism
for the Koszul complex of $(X, \cM_X)$
in Definition \ref{defn:30}.
In the construction of these two residue morphisms
the log complex manifolds $(X_{\vr}, \cM_{X_{\vr}})$ above
play a role of ``target'' spaces.
Then, we will prove Theorem \ref{thm:1}
at the end of this section.

\begin{defn}
A map
$\rf_X \colon X \longrightarrow \bZ^k$
is defined by
\begin{equation}
\rf_X(x)=\left(\rank \overline{\cM(i)}_{X,x}\gp\right)_{i=1}^k \in \bZ^k,
\end{equation}
where $\cM(i)_X$ is the monoid sheaf
given in Lemma \ref{lem:1} for each $i=\ki$.
\end{defn}

\begin{defn}
For $\vr \in \bZ^k_{\ge \ve}$,
we set
$\overline{X}_{\vr}=\{x \in X \mid \rf_X(x) \ge \vr\}$,
where $\ge$ is the partial order on $\bZ^k$
defined in \eqref{eq:3}.
\end{defn}

\begin{lem}
$\overline{X}_{\vr}$ is a closed analytic subset of $X$.
\end{lem}
\begin{proof}
Since the question is of local nature,
we may assume that
$(X, \cM_X)$ is an open neighborhood of the origin
of a local model $(U, \cM_U)$ as in \ref{para:1}.
Then we have
\begin{equation}
\label{eq:7}
\overline{X}_{\vr}
=X \cap \bigcup_{\ula \in S_{\vr}(\Lambda)}D[\ula],
\end{equation}
where $D[\ula]=\bigcap_{\lambda \in \ula}D_{\lambda}$
for $\ula \subset \Lambda$.
\end{proof}

\begin{defn}
\label{defn:28}
For $\vr \in \bZ^k_{\ge \ve}$,
the normalization of the reduced \ca subspace $\overline{X}_{\vr}$
is denoted by $X_{\vr}$.
The composite of the canonical morphism
$X_{\vr} \longrightarrow \overline{X}_{\vr}$
and the closed immersion $\overline{X}_{\vr} \hookrightarrow X$
is denoted by $a_{\vr}$.
Then $a_{\vr}$ is a finite morphism.
A log structure $\cM_{X_{\vr}}$ on $X_{\vr}$
is defined by $\cM_{X_{\vr}}=a_{\vr}^{\ast}\cM_X$.
\end{defn}

\begin{lem}
\label{lem:4}
For $\vr \in \bZ^k_{\ge \ve}$,
we have the following \textup{:}
\begin{enumerate}
\item
$X_{\vr}$ is nonsingular.
\item
\label{item:11}
For any $x \in X_{\vr}$,
there exist an open neighborhood $V$ of $x$
and a reduced \snc divisor $D_V$ on $V$,
such that the log structure $\cM_V$ is isomorphic to
the log structure $\cM_V(D_V) \oplus \bN_V^{|\vr|}$,
where the structure morphism
$\alpha \colon \cM_V(D_V) \oplus \bN_V^{|\vr|} \longrightarrow \cO_V$
is given by
\begin{equation}
\alpha(f \oplus \vv)
=
\begin{cases}
0 \qquad &\text{if $\vv \not= \vo$}, \\
f \qquad &\text{if $\vv=\vo$},
\end{cases}
\end{equation}
for $f \in \cM_V(D_V) \subset \cO_V$
and for $\vv \in \bN_V^{|\vr|}$.
\end{enumerate}
\end{lem}
\begin{proof}
We may work on a local model $(U, \cM_U)$ in \ref{para:1}.
From \eqref{eq:7} for $\overline{U}_{\vr}$,
and from the equality $U \cap D[\ula]=D[\ula]$
for any $\ula \in S_{\vr}(\Lambda)$ with $\vr \ge \ve$,
we obtain
$U_{\vr}=\coprod_{\ula \in S_{\vr}(\Lambda)}D[\ula]$.
Thus $U_{\vr}$ is nonsingular
and the restriction $a_{\vr}|_{D[\ula]}$
coincides with the canonical inclusion $D[\ula] \hookrightarrow U$.
Moreover, the log structure $\cM_{U_{\vr}}|_{D[\ula]}$
coincides with the pull-back of
$\cM_{\bC^{\Lambda} \times \bC^l}(D)$
by the closed immersion $D[\ula] \hookrightarrow \bC^{\Lambda} \times \bC^l$.
By setting 
$D_V|_{D[\ula]}=\sum_{\lambda \in \Lambda \setminus \ula}D_{\lambda}\cap D[\ula]$,
the condition \ref{item:11} is satisfied.
\end{proof}

\begin{rmk}
We note that the condition \ref{item:11} is
a special case of the condition (3.4.1) in \cite{FujisawaMHSLSD}.
\end{rmk}

\begin{lem}
\label{lem:2}
For any $\vr \in \bZ^k_{\ge \ve}$,
there exist a unique \nc divisor $D$ on $X_{\vr}$
and a unique inclusion $\cM_{X_{\vr}}(D) \hookrightarrow \cM_{X_{\vr}}$
satisfying the following$:$
\begin{enumerate}
\item
\label{item:12}
For any $x \in X_{\vr}$, there exists an open neighborhood $V$ of $x$
such that the inclusion $\cM_{X_{\vr}}(D)|_V \hookrightarrow \cM_{X_{\vr}}|_V$
induces an isomorphism
$\cM_{X_{\vr}}(D)|_V \oplus \bN^{|\vr|}_V \simeq \cM_{X_{\vr}}|_V$
as log structures.
\end{enumerate}
Moreover $\cM_{X_{\vr}}\gp/\cM_{X_{\vr}}(D)\gp$,
which is a locally free $\bZ$-module of rank $|\vr|$,
carries a symmetric bilinear form with values in $\bZ$,
whose stalk at any $x \in X_{\vr}$
coincides with the canonical bilinear form
associated to the finitely generated free monoid
$(\cM_{X_{\vr}}/\cM_{X_{\vr}}(D))_x$
\textup{(}cf. Definition \textup{\ref{defn:9}}\textup{)}.
\end{lem}
\begin{proof}
By considering the connected components of $X_{\vr}$,
we can easily deduce the conclusion
form Lemmas 3.7 and 3.10 of \cite{FujisawaMHSLSD}.
\end{proof}

\begin{defn}
\label{defn:6}
For $\vr \in \bZ^k_{\ge \ve}$,
we denote by $D_{\vr}$
the normal crossing divisor
$D$ in Lemma \ref{lem:2}
on $X_{\vr}$.
We set $L_{\vr}=\cM_{X_{\vr}}\gp/\cM_{X_{\vr}}(D)\gp$
and $\varepsilon_{\vr}=\bigwedge^{|\vr|}L_{\vr}$,
which are locally free $\bZ$-modules of rank $|\vr|$
and of rank one respectively.
The symmetric bilinear form on $L_{\vr}$
in Lemma \ref{lem:2}
induces an isomorphism
$\varepsilon_{\vr} \otimes_{\bZ} \varepsilon_{\vr} \longrightarrow \bZ$,
which is denoted by $\vartheta_{\vr}$.
\end{defn}

\begin{defn}
For $\vr \in \bZ^k_{\ge \ve}$,
the monoid subsheaf $\cM_{X_{\vr}}(D_{\vr})$
defines an increasing filtration 
on $\Omega_{X_{\vr}}(\log \cM_{X_{\vr}})$,
as in Definition \ref{defn:29},
which is denoted by $\widehat{W}$ for a while.
\end{defn}

\begin{para}
The morphism of abelian sheaves
$\bigwedge^m\dlog \colon
\bigwedge^m\cM_{X_{\vr}}\gp
\longrightarrow
\Omega^m_{X_{\vr}}(\log \cM_{X_{\vr}})$
induces a morphism of $\cO_{X_{\vr}}$-modules
$\bigwedge^m\cM_{X_{\vr}}\gp
\otimes_{\bZ} \Omega^{n-m}_{X_{\vr}}(\log D_{\vr})
\longrightarrow
\gr_m^{\widehat{W}}\Omega^n_{X_{\vr}}(\log \cM_{X_{\vr}})$,
which factors through the surjection
$\bigwedge^m\cM_{X_{\vr}}\gp
\otimes_{\bZ} \Omega^{n-m}_{X_{\vr}}(\log D_{\vr})
\longrightarrow
\bigwedge^mL_{\vr} \otimes_{\bZ} \Omega^{n-m}_{X_{\vr}}(\log D_{\vr})$.
Thus a morphism of $\cO_{X_{\vr}}$-modules
\begin{equation}
\label{eq:5}
\bigwedge^m L_{\vr} \otimes_{\bZ} \Omega^{n-m}_{X_{\vr}}(\log D_{\vr})
\longrightarrow
\gr_m^{\widehat{W}}\Omega^n_{X_{\vr}}(\log \cM_{X_{\vr}})
\end{equation}
is obtained for all $m,n \in \bZ$.
\end{para}

\begin{lem}
For $\vr \in \bZ^k_{\ge \ve}$,
The morphism \eqref{eq:5} gives us
an isomorphism of $\cO_{X_{\vr}}$-modules
\begin{equation}
\label{eq:6}
\bigwedge^m L_{\vr} \otimes_{\bZ} \Omega^{n-m}_{X_{\vr}}(\log D_{\vr})
\longrightarrow
\gr_m^{\widehat{W}}\Omega^n_{X_{\vr}}(\log \cM_{X_{\vr}})
\end{equation}
for all $m, n \in \bZ$.
Therefore,
$\widehat{W}_{|\vr|}\Omega_{X_{\vr}}(\log \cM_{X_{\vr}})
=\Omega_{X_{\vr}}(\log \cM_{X_{\vr}})$.
\end{lem}
\begin{proof}
Since the question is of local nature,
we may assume
that $\cM_{X_{\vr}}=\cM_{X_{\vr}}(D_{\vr}) \oplus \bN^{|\vr|}_{X_{\vr}}$
as in \ref{item:12}.
Then we have
$\Omega^1_{X_{\vr}}(\log \cM_{X_{\vr}})
\simeq
(\cO_{X_{\vr}} \otimes_{\bZ} L_{\vr}) \oplus \Omega^1_{X_{\vr}}(\log D_{\vr})$,
from which we obtain the conclusion easily.
\end{proof}

\begin{defn}[Residue morphism]
\label{defn:23}
Let $\vr \in \bZ^k_{\ge \ve}$.
By composing the three morphisms,
the surjection
$\Omega^n_{X_{\vr}}(\log \cM_{X_{\vr}})
\longrightarrow \gr_{|\vr|}^{\widehat{W}}\Omega^n_{X_{\vr}}(\log \cM_{X_{\vr}})$,
the inverse of the isomorphism \eqref{eq:6} for $m=|\vr|$
and the inclusion
$\varepsilon_{\vr} \otimes_{\bZ} \Omega^{n-|\vr|}_{X_{\vr}}(\log D_{\vr})
\hookrightarrow
\varepsilon_{\vr} \otimes_{\bZ} \Omega^{n-|\vr|}_{X_{\vr}}(\log \cM_{X_{\vr}})$,
we obtain a morphism of $\cO_{X_{\vr}}$-modules
$\Omega^n_{X_{\vr}}(\log \cM_{X_{\vr}})
\longrightarrow
\varepsilon_{\vr} \otimes \Omega^{n-|\vr|}_{X_{\vr}}(\log \cM_{X_{\vr}})$.
Then we have a morphism of $\cO_X$-modules
\begin{equation}
\label{eq:98}
(a_{\vr})_{\ast}\Omega^n_{X_{\vr}}(\log \cM_{X_{\vr}})
\longrightarrow
(a_{\vr})_{\ast}(\varepsilon_{\vr} \otimes_{\bZ}
\Omega^{n-|\vr|}_{X_{\vr}}(\log \cM_{X_{\vr}}))
\end{equation}
on $X$.
A morphism of $\cO_X$-modules
\begin{equation}
\label{eq:111}
\res_{\vr} \colon
\Omega^n_X(\log \cM_X)
\longrightarrow
(a_{\vr})_{\ast}
(\varepsilon_{\vr}
\otimes_{\bZ}
\Omega^{n-|\vr|}_{X_{\vr}}(\log \cM_{X_{\vr}}))
\end{equation}
is defined as the composite of the canonical morphism
$\Omega^n_X(\log \cM_X)
\longrightarrow
(a_{\vr})_{\ast}\Omega^n_{X_{\vr}}(\log \cM_{X_{\vr}})$
and the morphism \eqref{eq:98}.
It is easy to see
that these morphisms form a morphism of complexes of $\bC$-sheaves
\begin{equation}
\res_{\vr} \colon
\Omega_X(\log \cM_X)
\longrightarrow
(a_{\vr})_{\ast}
(\varepsilon_{\vr}
\otimes_{\bZ}
\Omega_{X_{\vr}}(\log \cM_{X_{\vr}})[-|\vr|])
\end{equation}
for all $\vr \in \bZ^k_{\ge \ve}$.
\end{defn}

\begin{para}[\textbf{Local description of the residue morphism}]
\label{para:2}
For $\vr \in \bZ^k_{\ge \ve}$,
we describe $\res_{\vr}$ locally.
We may work on a local model $(U, \cM_U)$.
Then
$U_{\vr}=\coprod_{\ula \in S_{\vr}(\Lambda)}D[\ula]$
and $D_{\vr}|_{D[\ula]}=D_{\Lambda \setminus \ula} \cap D[\ula]$
as in the proof of Lemma \ref{lem:4}.
For $\ula \in S_{\vr}(\Lambda)$,
we set $D_{\ula}=\sum_{\lambda \in \ula}D_{\lambda}$.
Then we have
\begin{equation}
(\prod_{\lambda \in \Lambda_i}x_{\lambda})
\cdot
\Omega^n_{\bC^{\Lambda} \times \bC^l}(\log D)
\subset W(D_{\ula})_{|\ula|-1}\Omega^n_{\bC^{\Lambda} \times \bC^l}(\log D)
\end{equation}
because $\ula \cap \Lambda_i \not= \emptyset$
for all $i=\ki$.
Therefore the Poincar\'e residue morphism for $W(D_{\ula})$
\begin{equation}
\res^{\ula} \colon
\Omega^n_{\bC^{\Lambda} \times \bC^l}(\log D)
\longrightarrow
\varepsilon(\ula)
\otimes_{\bZ}
\Omega^{n-|\ula|}_{D[\ula]}(\log D_{\Lambda \setminus \ula} \cap D[\ula])
\end{equation}
induces a morphism of $\cO_U$-modules
$\Omega^n_U(\log \cM_U)
\longrightarrow
\varepsilon(\ula)
\otimes_{\bZ}
\Omega^{n-|\ula|}_{D[\ula]}(\log D_{\Lambda \setminus \ula} \cap D[\ula])$
by the identification \eqref{eq:10}.
Composing with the canonical inclusion,
we obtain a morphism
\begin{equation}
\res^{\ula} \colon
\Omega^n_U(\log \cM_U)
\longrightarrow
\varepsilon(\ula)
\otimes_{\bZ}
\Omega^{n-|\ula|}_{D[\ula]}(\log \cM_{D[\ula]})
\end{equation}
denoted by $\res^{\ula}$ again.
Under the identification
\begin{equation}
(a_{\vr})_{\ast}(
\varepsilon_{\vr}
\otimes_{\bZ}
\Omega^{n-|\vr|}_{U_{\vr}}(\log \cM_{U_{\vr}})
)
\simeq
\bigoplus_{\ula \in S_{\vr}(\Lambda)}
\varepsilon(\ula)
\otimes_{\bZ}
\Omega^{n-|\ula|}_{D[\ula]}(\log \cM_{D[\ula]}),
\end{equation}
the equality
$\res_{\vr}=\sum_{\ula \in S_{\vr}(\Lambda)}\res^{\ula}$
can be easily checked.
\end{para}

\begin{lem}
\label{lem:5}
For $\vr \in \bZ^k_{\ge \ve}$ and $I \subset \ski$,
we have
\begin{gather}
\res_{\vr}(W_m\Omega^n_X(\log \cM_X))
\subset
(a_{\vr})_{\ast}(
\varepsilon_{\vr}
\otimes_{\bZ}
W_{m-|\vr|}\Omega^{n-|\vr|}_{X_{\vr}}(\log \cM_{X_{\vr}})) \\
\res_{\vr}(W(I)_m\Omega^n_X(\log \cM_X))=0 \quad
\text{if $|\vr_I| > m$},
\end{gather}
where $\vr_I$ is the image of $\vr \in \bZ^k$
by the projection $\bZ^k \longrightarrow \bZ^I$.
\end{lem}
\begin{proof}
Easy from the local description above.
\end{proof}

\begin{para}
For $\vq, \vr \in \bnnZ^k$
with $\vr \ge \vq+\ve$,
the morphism $\res_{\vr}$ in \eqref{eq:111}
factors through the surjection
$\Omega^n_X(\log \cM_X)
\longrightarrow
\Omega^n_X(\log \cM_X)/
\sum_{i=1}^{k}W(i)_{q_i}$
by Lemma \ref{lem:5}.
Thus a morphism of $\cO_X$-modules
\begin{equation}
\label{eq:57}
\res_{\vr} \colon
\Omega^n_X(\log \cM_X)/
\sum_{i=1}^{k}W(i)_{q_i}
\longrightarrow
(a_{\vr})_{\ast}
(\varepsilon_{\vr}
\otimes_{\bZ}
\Omega^{n-|\vr|}_{X_{\vr}}(\log \cM_{X_{\vr}}))
\end{equation}
is obtained,
which is denoted by $\res_{\vr}$ again
by abuse of notation.
\end{para}

\begin{lem}
\label{lem:25}
We have the isomorphism of complexes of $\bC$-sheaves
\begin{equation}
\sum_{\substack{\vr \ge \vq+\ve \\ |\vr|=m}}
\res_{\vr} \colon
\gr_m^W
\bigl(\Omega_X(\log \cM_X)/
\sum_{i=1}^{k}W(i)_{q_i}\bigr)
\overset{\simeq}{\longrightarrow}
\bigoplus_{\substack{\vr \ge \vq+\ve \\ |\vr|=m}}
(a_{\vr})_{\ast}
(\varepsilon_{\vr}
\otimes_{\bZ}
\Omega_{X_{\vr}})[-m]
\end{equation}
for any $\vq \in \bN^k$ and $m \in \bZ$,
under which
$W(I)_l\gr_m^W\bigl(\Omega_X(\log \cM_X)/\sum_{i=1}^{k}W(i)_{q_i}\bigr)$
is identified with
the direct sum of
$(a_{\vr})_{\ast}(\varepsilon_{\vr} \otimes_{\bZ} \Omega_{X_{\vr}})[-m]$
over the index set
\begin{equation}
\label{eq:8}
\{ \vr \in \bN^k \mid
\vr \ge \vq+\ve, |\vr|=m, |\vr_I| \le l\},
\end{equation}
for all $I \subset \ski$ and $l \in \bZ$.
\end{lem}
\begin{proof}
We may work on a local model
$(U, \cM_U)$
and use the same notation as in \ref{para:1}.
In particular, the partition
$\Lambda=\coprod_{i=1}\Lambda_i$
is associated to the semistable morphism $\varphi$ in \ref{item:3}.
Since
\begin{equation}
(\prod_{\lambda \in \Lambda_i}x_{\lambda})
\cdot
\Omega^n_{\bC^{\Lambda} \times \bC^l}(\log D)
\subset W(D_i)_0\Omega^n_{\bC^{\Lambda} \times \bC^l}(\log D)
\end{equation}
for all $i=\ki$,
the identification \eqref{eq:10}
induces an isomorphism
\begin{equation}
\label{eq:93}
\Omega_{\bC^{\Lambda} \times \bC^l}(\log D)
/\sum_{i=1}^{k}W(D_i)_{q_i}
\overset{\simeq}{\longrightarrow}
\Omega_U(\log \cM_U)
/\sum_{i=1}^{k}W(i)_{q_i}
\end{equation}
for any $\vq \in \bN^k$.
Then the local description of $\res_r$ in \ref{para:2}
implies the conclusion
from the usual Poincar\'e residue isomorphism
as in \cite[3.1]{DeligneII}
(cf. \cite[Section 3]{FujisawaLHSSV}).
\end{proof}

The following corollary will be used
in Section \ref{sec:proofs-main-theorems}.

\begin{cor}
\label{cor:1}
For $I \subset \ski$,
we have $L=L(I) \ast L(\ski \setminus I)$ on $A_{\bC}$.
\end{cor}
\begin{proof}
We set $J=\ski \setminus I$.
By the definition \eqref{eq:2} of $L(I)$ on $A_{\bC}$,
it suffices to prove the equality
$W(I) \ast W(J)=W$
on 
$\Omega_X(\log \cM_X)/\sum_{i=1}^{k}W(i)_{q_i}$
for all $\vq \in \bN^k$.
From the isomorphism \eqref{eq:93}
and the equality $W(D_I) \ast W(D_J)=W(D)$
on $\Omega_{\bC^{\Lambda} \times \bC^l}(\log D)$
for a local situation in \ref{para:1},
we have
$W_m \subset (W(I) \ast W(J))_m$ on 
$\Omega_X(\log \cM_X)/\sum_{i=1}^{k}W(i)_{q_i}$
for all $m$.
By the lemma above,
$W(I)_a \cap W(J)_b=0$ on
$\gr_m^W\bigl(\Omega_X(\log \cM_X)/\sum_{i=1}^{k}W(i)_{q_i}\bigr)$
if $a+b < m$.
Therefore $W(I)_a \cap W(J)_b \subset W_{a+b}$ on
$\Omega_X(\log \cM_X)/\sum_{i=1}^{k}W(i)_{q_i}$
for any $a, b \in \bZ$.
\end{proof}

\begin{defn}
Let $\vr \in \bZ^k_{\ge \ve}$.
Then the monoid subsheaf $\cM_{X_{\vr}}(D_{\vr})$ of $\cM_{X_{\vr}}$
gives us
an increasing filtration
$W(\cM_{X_{\vr}}(D_{\vr})_{\bQ}\gp)$
on $\kos_X(\cM_X)$
as in \cite[Definition 1.8]{FujisawaMHSLSD}.
This filtration is denoted by $\widehat{W}$ for a while
as in the case of the log de Rham complex.
\end{defn}

\begin{lem}
Let $\vr \in \bZ^k_{\ge \ve}$.
There exists an isomorphism
of complexes of $\bQ$-sheaves
\begin{equation}
\label{eq:15}
\bigwedge^mL_{\vr} \otimes \kos_{X_{\vr}}(D_{\vr})[-m]
\overset{\simeq}{\longrightarrow}
\gr_m^{\widehat{W}}\kos_{X_{\vr}}(\cM_{X_{\vr}})
\end{equation}
for all $m$.
Therefore,
$\widehat{W}_{|\vr|}\kos_{X_{\vr}}(\cM_{X_{\vr}})
=\kos_{X_{\vr}}(\cM_{X_{\vr}})$.
\end{lem}
\begin{proof}
By Proposition 1.10 of \cite{FujisawaMHSLSD}.
\end{proof}

\begin{defn}[Residue morphism for the Koszul complex]
\label{defn:30}
Let $\vr \in \bZ^k_{\ge \ve}$.
By composing the three morphisms,
the surjection
$\kos_{X_{\vr}}(\cM_{X_{\vr}})
\longrightarrow
\gr_{|\vr|}^{\widehat{W}}\kos_{X_{\vr}}(\cM_{X_{\vr}})$,
the inverse of the isomorphism \eqref{eq:15} for $m=|\vr|$
and the inclusion
$\varepsilon_{\vr} \otimes_{\bZ} \kos_{X_{\vr}}(D_{\vr})[-|\vr|]
\hookrightarrow
\varepsilon_{\vr} \otimes_{\bZ} \kos_{X_{\vr}}(\cM_{X_{\vr}})[-|\vr|]$,
we obtain a morphism of complexes of $\bQ$-sheaves
$\kos_{X_{\vr}}(\cM_{X_{\vr}})
\longrightarrow
\varepsilon_{\vr} \otimes_{\bZ} \kos_{X_{\vr}}(\cM_{X_{\vr}})[-|\vr|]$.
Then we have a morphism of complexes of $\bQ$-sheaves
\begin{equation}
\label{eq:102}
(a_{\vr})_{\ast}\kos_{X_{\vr}}(\cM_{X_{\vr}})
\longrightarrow
(a_{\vr})_{\ast}(\varepsilon_{\vr} \otimes_{\bZ}
\kos_{X_{\vr}}(\cM_{X_{\vr}}))[-|\vr|]
\end{equation}
on $X$.
A morphism of complexes of $\bQ$-sheaves
\begin{equation}
\label{eq:112}
\res^{\bQ}_{\vr} \colon
\kos_X(\cM_X)
\longrightarrow
(a_{\vr})_{\ast}
(\varepsilon_{\vr}
\otimes_{\bZ}
\kos_{X_{\vr}}(\cM_{X_{\vr}}))[-|\vr|]
\end{equation}
is defined as the composite of the canonical morphism
$\kos_X(\cM_X)
\longrightarrow
(a_{\vr})_{\ast}\kos_{X_{\vr}}(\cM_{X_{\vr}})$
and the morphism \eqref{eq:102}.
\end{defn}

\begin{para}[\textbf{The stalk of the residue morphism $\res_{\vr}^{\bQ}$}]
\label{para:4}
Now we describe $\res_{\vr}^{\bQ}$ stalkwise.
We may work at the origin $x=0$
of a local model $(U, \cM_U)$
in \ref{para:1}.
Similarly to \eqref{eq:23},
we have
\begin{equation}
(a_{\vr})_{\ast}
(\varepsilon_{\vr}
\otimes_{\bZ}
\kos_{X_{\vr}}(\cM_{X_{\vr}})[-|\vr|])^n_x
\simeq
\bigoplus_{\umu \in S_{\vr}(\Lambda)}
\bigoplus_{\unu \in S(\Lambda)}
\varepsilon(\umu)
\otimes_{\bZ}
\varepsilon(\unu)
\otimes_{\bZ}
\kos(\cO^{\ast}_{D[\umu],x})^{n-|\vr|-|\unu|}
\end{equation}
for all $n \in \bZ$.
Via the identification \eqref{eq:23},
\begin{equation}
\res^{\bQ}_{\vr, x}
(\varepsilon(\ula)
\otimes_{\bZ}
\kos(\cO^{\ast}_{U,x})^{n-|\ula|})
\subset
\bigoplus_{\substack{\umu \in S_{\vr}(\Lambda) \\ \umu \subset \ula}}
\varepsilon(\umu)
\otimes_{\bZ}
\varepsilon(\ula \setminus \umu)
\otimes_{\bZ}
\kos(\cO^{\ast}_{D[\umu],x})^{n-|\ula|}
\end{equation}
and the restriction of $\res^{\bQ}_{\vr, x}$
on the direct summand
$\varepsilon(\ula)
\otimes_{\bZ}
\kos(\cO^{\ast}_U)_x^{n-|\ula|}$
is identified with
\begin{equation}
\label{eq:24}
\sum_{\substack{\umu \in S_{\vr}(\Lambda) \\ \umu \subset \ula}}
\chi(\umu, \ula \setminus \umu)^{-1}
\otimes
\kos(a[\umu]^{\ast}_x),
\end{equation}
where $\chi(\umu, \ula \setminus \umu)$
is the isomorphism \eqref{eq:44}
and $\kos(a[\umu]^{\ast}_x)$
is the induced morphism from the canonical morphism
$a[\umu]^{\ast}_x \colon \cO^{\ast}_{U,x} \longrightarrow \cO^{\ast}_{D[\umu],x}$
for the closed immersion
$a[\umu] \colon D[\umu] \hookrightarrow U$.
\end{para}

\begin{lem}
For $\vr \in \bZ_{\ge \ve}^k$
and $I \subset \ski$,
we have
\begin{gather}
\res^{\bQ}_{\vr}(W_{|\vr|}\kos_X(\cM_X))
\subset
(a_{\vr})_{\ast}(
\varepsilon_{\vr}
\otimes_{\bZ}
\kos_{X_{\vr}}(\cO^{\ast}_{X_{\vr}}))[-|\vr|] \\
\res^{\bQ}_{\vr}(W(I)_m\kos_X(\cM_X))=0 \quad
\text{if $|\vr_I| > m$.}
\end{gather}
\end{lem}
\begin{proof}
We may work stalkwise.
Under the identification \eqref{eq:23},
$W(I)_m\kos_X(\cM_X)^n_x$
is identified with
\begin{equation}
\bigoplus_{\substack{\ula \in S(\Lambda) \\ |\ula \cap \Lambda_I| \le m}}
\varepsilon(\ula)
\otimes_{\bZ}
\kos_U(\cO^{\ast}_U)_x^{n-|\ula|},
\end{equation}
where $\Lambda_I=\coprod_{i \in I}\Lambda_i$.
Then we can easily check the conclusions.
\end{proof}

\begin{lem}
\label{lem:26}
We have the quasi-isomorphism
\begin{equation}
\sum_{\substack{\vr \ge \vq+\ve \\ |\vr|=m}}
\res^{\bQ}_{\vr} \colon
\gr_m^W\bigl(
\kos_X(\cM_X)/\sum_{i=1}^{k}W(i)_{q_i}
\bigr)
\longrightarrow
\bigoplus_{\substack{\vr \ge \vq+\ve \\ |\vr|=m}}
(a_{\vr})_{\ast}(
\varepsilon_{\vr}
\otimes_{\bZ}
\kos_{X_{\vr}}(\cO^{\ast}_{X_{\vr}}))[-m]
\end{equation}
for any $\vq \in \bN^k$ and $m \in \bZ$.
Similarly, we have the quasi-isomorphism
\begin{equation}
\sum\res^{\bQ}_{\vr} \colon
W(I)_l\gr_m^W\bigl(\kos_X(\cM_X)/\sum_{i=1}^{k}W(i)_{q_i}\bigr)
\longrightarrow
\bigoplus
(a_{\vr})_{\ast}(
\varepsilon_{\vr}
\otimes_{\bZ}
\kos_{X_{\vr}}(\cO^{\ast}_{X_{\vr}}))[-m]
\end{equation}
for all $I \subset \ski$ and $l \in \bZ$,
where the sum and the direct sum are taken
over the same index set as \eqref{eq:8}.
\end{lem}
\begin{proof}
We may work stalkwise as in \ref{para:4}.
Note that
the morphism
$\kos(a[\umu]^{\ast}_x) \colon
\kos(\cO^{\ast}_{U,x}) \longrightarrow \kos(\cO^{\ast}_{D[\umu],x})$
in \eqref{eq:24}
is a quasi-isomorphism
because both sides are canonically quasi-isomorphic to $\bQ$
by \cite[Corollary 1.15]{FujisawaMHSLSD}.
Then the conclusions follows
from the local description in \ref{para:4}.
\end{proof}

\begin{lem}
\label{lem:8}
For $\vr \in \bZ^k_{\ge \ve}$,
the diagram
\begin{equation}
\begin{CD}
\kos_X(\cM_X)
@>{\res^{\bQ}_{\vr}}>>
(a_{\vr})_{\ast}(\varepsilon_{\vr}
\otimes_{\bZ}
\kos_{X_{\vr}}(\cM_{X_{\vr}}))[-|\vr|] \\
@V{\psi_X}VV
@VV{(a_{\vr})_{\ast}
(\id \otimes (2\pi\sqrt{-1})^{-|\vr|}
\psi_{X_{\vr}}[-|\vr|]}V \\
\Omega_X(\log \cM_X)
@>>{\res_{\vr}}>
(a_{\vr})_{\ast}(\varepsilon_{\vr}
\otimes_{\bZ}
\Omega_{X_{\vr}}(\log \cM_{X_{\vr}}))[-|\vr|]
\end{CD}
\end{equation}
is commutative,
where $\psi_X$
and $\psi_{X_{\vr}}$
are the morphisms in \ref{item:5}.
\end{lem}
\begin{proof}
The commutativity of the diagram
\begin{equation}
\begin{CD}
\bigwedge^m L_{\vr}
\otimes_{\bZ}
\kos_{X_{\vr}}(D_{\vr})[-m]
@>{\simeq}>>
\gr_m^{\widehat{W}}\kos_{X_{\vr}}(\cM_{X_{\vr}}) \\
@V{\id \otimes (2\pi\sqrt{-1})^{-m}\psi_{(X_{\vr}, D_{\vr})}[-m]}VV
@VV{\gr_m^{\widehat{W}}\psi_{(X_{\vr}, \cM_{X_{\vr}})}}V \\
\bigwedge^mL_{\vr} \otimes_{\bZ} \Omega_{X_{\vr}}(\log D_{\vr})[-m]
@>>>
\gr_m^{\widehat{W}}\Omega_{X_{\vr}}(\log \cM_{X_{\vr}})
\end{CD}
\end{equation}
can be easily checked from the definition
in \cite[(2.4)]{FujisawaMHSLSD}.
The conclusion follows from the case of $m=|\vr|$.
\end{proof}

\begin{proof}[Proof of Theorem \textup{\ref{thm:1}}]
We use the notation in Remark \ref{rmk:8} for short.
Because
\begin{equation}
d_i(L_mA^n_{\bQ}) \subset L_{m-1}A^{n+1}_{\bQ}, \quad
d_i(L_mA^n_{\bC}) \subset L_{m-1}A^{n+1}_{\bC}
\end{equation}
for all $i=\ki$ and $m, n \in \bZ$,
we have
\begin{gather}
\gr_m^LA_{\bQ}
=
\bigoplus_{\vq \in \bN^k}
\bQ\vu^{\vq} \otimes_{\bC}
\gr_{m+2|\vq|+k}^W
\bigl(\kos_X(\cM_X)/\sum_{i=1}^{k}W(i)_{q_i}\bigr)[k] \\
\gr_m^LA_{\bC}
=
\bigoplus_{\vq \in \bN^k}
\bC\vu^{\vq} \otimes_{\bC}
\gr_{m+2|\vq|+k}^W
\bigl(\Omega_X(\log \cM_X)/\sum_{i=1}^{k}W(i)_{q_i}\bigr)[k]
\end{gather}
as complexes.
Therefore,
by Lemmas \ref{lem:25} and \ref{lem:26},
and by the canonical isomorphisms
\begin{gather}
\gr_m^L(L(I)_bA_{\bQ}/L(I)_aA_{\bQ})
\simeq
L(I)_b\gr_m^LA_{\bQ}/L(I)_a\gr_m^LA_{\bQ} \\
\gr_m^L(L(I)_bA_{\bC}/L(I)_aA_{\bC})
\simeq
L(I)_b\gr_m^LA_{\bC}/L(I)_a\gr_m^LA_{\bC},
\end{gather}
we have a quasi-isomorphism
\begin{equation}
\label{eq:34}
\gr_m^L(L(I)_bA_{\bQ}/L(I)_aA_{\bQ})
\longrightarrow
\bigoplus
\bQ\vu^{\vq} \otimes_{\bQ}
(a_{\vr})_{\ast}(\varepsilon_{\vr}
\otimes
\kos_{X_{\vr}}(\cO_{X_{\vr}}^{\ast}))[-m-2|\vq|]
\end{equation}
and an isomorphism
\begin{equation}
\label{eq:62}
\gr_m^L(L(I)_bA_{\bC}/L(I)_aA_{\bC})
\overset{\simeq}{\longrightarrow}
\bigoplus
\bC\vu^{\vq} \otimes_{\bC}
(a_{\vr})_{\ast}(\varepsilon_{\vr}
\otimes_{\bZ}
\Omega_{X_{\vr}})[-m-2|\vq|],
\end{equation}
under which the morphism induced by $\alpha$ is identified with
\begin{equation}
\label{eq:86}
\bigoplus
(a_{\vr})_{\ast}(
\id
\otimes
(2\pi\sqrt{-1})^{-m-|\vq|}\psi_{(X_{\vr}, \cO_{X_{\vr}}^{\ast})})[-m-2|\vq|]
\end{equation}
by Lemma \ref{lem:8},
where the direct sums \eqref{eq:34}--\eqref{eq:86}
are taken over the index set
\begin{equation}
\label{eq:52}
\{(\vq,\vr) \in \bN^k \times \bN^k \mid
\vr \ge \vq+\ve,
|\vr|=m+2|\vq|+k,
a < |\vr_I|-2|\vq_I|-|I| \le b\}.
\end{equation}
From \eqref{eq:50} and Lemma \ref{lem:25}, we have
\begin{align}
F^p\gr_m^L(L(I)_bA^n_{\bC}/L(I)_aA_{\bC})
&=
\bigoplus_{|\vq| \le n-p}
\bC\vu^{\vq} \otimes_{\bC}
\gr_m^L
(L(I)_b(A^n_{\bC})_{\vq}/L(I)_a(A^n_{\bC})_{\vq} \\
&\simeq
\bigoplus_{|\vq| \le n-p}
\bC\vu^{\vq} \otimes_{\bC}
\bigl(L(I)_b\gr_m^L(A^n_{\bC})_{\vq}/L(I)_a\gr_m^L(A^n_{\bC})_{\vq}\bigr) \\
&\simeq
\bigoplus
\bC\vu^{\vq} \otimes_{\bC}
(a_{\vr})_{\ast}
(\varepsilon_{\vr}
\otimes_{\bZ}
\Omega^{n-m-2|\vq|}_{X_{\vr}}),
\end{align}
where the direct sum in the last term
is taken over the index set
\begin{equation}
\{ (\vq, \vr) \in \bN^k \times \bN^k \mid
\vr \ge \vq+\ve, |\vq| \le n-p, |\vr|=m+2|\vq|+k,
a < |\vr_I|-2|\vq_I|-|I| \le b \}.
\end{equation}
Therefore the isomorphisms \eqref{eq:62}
induces an isomorphism of filtered complexes
\begin{equation}
\label{eq:79}
(\gr_m^L(L(I)_bA_{\bC}/L(I)_aA_{\bC}), F)
\overset{\simeq}{\longrightarrow}
\bigoplus
\bC\vu^{\vq} \otimes_{\bC}
((a_{\vr})_{\ast}
(\varepsilon_{\vr}
\otimes_{\bZ}
\Omega_{X_{\vr}})[-m-2|\vq|], F[-m-|\vq|]),
\end{equation}
where $F$ on the right hand side
is the stupid filtration
on $\varepsilon_{\vr} \otimes_{\bZ}\Omega_{X_{\vr}}$
and the index set of the direct sum on the right hand side
is the same as \eqref{eq:52}.
Because $\varepsilon_{\vr}$ admits
a positive definite symmetric bilinear form
$\theta_{\vr}$ as in Definition \ref{defn:6},
\begin{equation}
((L(I)_bA_{\bQ}/L(I)_aA_{\bQ},L),
(L(I)_bA_{\bC}/L(I)_aA_{\bC},L,F), \alpha)
\end{equation}
is a $\bQ$-\cmh complex on $X$
by \cite[(2.2.2)]{DeligneII}.
\end{proof}

\begin{rmk}
The assumption for
$f \colon (X, \cM_X) \longrightarrow (\ast,\bN^k)$ being projective
in Theorem \ref{thm:1}
can be relaxed to the assumptions that
$f \colon (X, \cM_X) \longrightarrow (\ast,\bN^k)$ is proper
and that $X_{\vr}$ is K\"ahler for all $\vr \in \bZ^k_{\ge \ve}$.
\end{rmk}

\begin{rmk}
\label{rmk:9}
By taking $a$ sufficiently small and $b$ sufficiently large
in \eqref{eq:79},
we have the isomorphism of filtered complexes
\begin{equation}
\label{eq:67}
(\gr_m^LA_{\bC}, F)
\overset{\simeq}{\longrightarrow}
\bigoplus
\bC\vu^{\vq} \otimes_{\bC}
((a_{\vr})_{\ast}
(\varepsilon_{\vr}
\otimes_{\bZ}
\Omega_{X_{\vr}})[-m-2|\vq|], F[-m-|\vq|]),
\end{equation}
for all $m \in \bZ$,
where the direct sum on the right hand side
is taken over the index set
\begin{equation}
\label{eq:61}
\{(\vr, \vq) \in \bN^k \times \bN^k
\mid \vr \ge \vq+\ve, |\vr|=m+2|\vq|+k\}.
\end{equation}
Similarly, we have the quasi-isomorphism
\begin{equation}
\label{eq:78}
\gr_m^LA_{\bQ}
\longrightarrow
\bigoplus
\bQ\vu^{\vq} \otimes_{\bQ}
(a_{\vr})_{\ast}(\varepsilon_{\vr}
\otimes_{\bZ}
\kos_X(\cM_{X_{\vr}})[-m-2|\vq|]
\end{equation}
for all $m \in \bZ$,
where the index set of the direct sum
is the same as \eqref{eq:61}.
\end{rmk}

\section{A complex $\cC(\Omega_{X_{\bullet}}(\log \cM_{X_{\bullet}}))$}
\label{sec:a Cech type complex}

In this section, we first construct a \v{C}ech type filtered complex
$(\cC(\Omega_{X_{\bullet}}(\log \cM_{X_{\bullet}})), \delta W)$
and a product on it.
Because the family of complex manifolds
$\{X_{\vr}\}_{\vr \in \bZ_{\ge \ve}^k}$ does not admit a simplicial
(or cubical) structure in general,
it is not possible to apply
the arguments in \cite[Section 2]{FujisawaPLMHS}.
Thus the construction in this section
requires some other tasks,
in which the log structures on $X$ and on $X_{\vr}$
play essential roles.
Second, we construct 
a kind of ``trace map''
$E_1^{-k,2\dim X+2k} \longrightarrow \bC$,
where $E_1^{p,q}$
denotes the $E_1$-terms of the spectral sequence
associated to
$(R\Gamma_c(X, \cC(\Omega_{X_{\bullet}}(\log \cM_{X_{\bullet}}))), \delta W)$
in this section.
The construction of this map
is similar to and slightly simplified from
the one in \cite[Definition 7.7]{FujisawaPLMHS}.

\begin{defn}
For $\vr \in \bN^k$,
we set
\begin{equation}
\bigwedge^{\otimes \vr}\overline{\cM}_X\gp
=
\bigwedge^{r_1}\overline{\cM(1)}_X\gp
\otimes_{\bZ}
\bigwedge^{r_2}\overline{\cM(2)}_X\gp
\otimes_{\bZ}
\dots
\otimes_{\bZ}
\bigwedge^{r_k}\overline{\cM(k)}_X\gp,
\end{equation}
which is regraded as a subsheaf of
$\bigwedge^{|\vr|}\overline{\cM}_X\gp$
by the inclusion
\begin{equation}
\bigwedge^{\otimes \vr}\overline{\cM}_X\gp
\ni
\vv_1 \otimes \vv_2 \otimes \dots \otimes \vv_k
\mapsto
\vv_1 \wedge \vv_2 \wedge \dots \wedge \vv_k
\in \bigwedge^{|\vr|}\overline{\cM}_X\gp.
\end{equation}
\end{defn}

\begin{para}
For $\vr \in \bZ^k_{\ge \ve}$,
the canonical morphism
$a_{\vr}^{-1}\cM_X\gp \longrightarrow \cM_{X_{\vr}}\gp$
induces morphisms
$a_{\vr}^{-1}\overline{\cM}_X\gp
\longrightarrow \cM_{X_{\vr}}\gp/\cM_{X_{\vr}}(D_{\vr})\gp
=L_{\vr}$
and
$a_{\vr}^{-1}\bigwedge^{|\vr|}\overline{\cM}_X\gp
\longrightarrow
\bigwedge^{|\vr|}L_{\vr}=\varepsilon_{\vr}$.
Thus we obtain a morphism of $\bZ$-sheaves
\begin{equation}
\label{eq:35}
\bigwedge^{|\vr|}\overline{\cM}_X\gp
\longrightarrow
(a_{\vr})_{\ast}\varepsilon_{\vr}
\end{equation}
for all $\vr \in \bZ^k_{\ge \ve}$.
\end{para}

\begin{lem}
By restricting the morphism \eqref{eq:35} to
$\bigwedge^{\otimes \vr}\overline{\cM}_X\gp$,
we obtain an isomorphism
\begin{equation}
\label{eq:36}
\bigwedge^{\otimes \vr}\overline{\cM}_X\gp
\overset{\simeq}{\longrightarrow}
(a_{\vr})_{\ast}\varepsilon_{\vr}
\end{equation}
for any $\vr \in \bZ^k_{\ge \ve}$.
\end{lem}
\begin{proof}
We may work stalkwise.
Then the local description in \ref{para:1}
implies the conclusion easily.
\end{proof}

\begin{defn}
\label{defn:31}
The image of $t_i \in \Gamma(X, \cM_X)$
by the projection $\cM_X \longrightarrow \overline{\cM}_X$
is denoted by $\overline{t}_i \in \Gamma(X, \overline{\cM}_X)$
for $i=\ki$.
Then a morphism
$\overline{t}_i \wedge \colon
\bigwedge\overline{\cM}_X \longrightarrow \bigwedge\overline{\cM}_X$
is defined by sending $\vv$ to $\overline{t}_i \wedge \vv$.
Because $\overline{t}_i \in \Gamma(X, \overline{\cM(i)}_X)$
as in Remark \ref{rmk:4},
the morphism $\overline{t}_i \wedge$
induces a morphism
$\overline{t}_i \wedge \colon
\bigwedge^{\otimes \vr}\overline{\cM}_X\gp
\longrightarrow
\bigwedge^{\otimes \vr+\ve_i}\overline{\cM}_X\gp$
for every $\vr \in \bN^k$.
Thus we obtain a morphism
\begin{equation}
\delta_i \colon
(a_{\vr})_{\ast}\varepsilon_{\vr}
\longrightarrow
(a_{\vr+\ve_i})_{\ast}\varepsilon_{\vr+\ve_i}
\end{equation}
via the isomorphism \eqref{eq:36}
for $\vr \in \bZ^k_{\ge \ve}$ and for $i=\ki$.
Trivially the equalities
\begin{equation}
\label{eq:41}
\delta_i^2=0, \quad
\delta_i\delta_j+\delta_j\delta_i=0
\end{equation}
hold for all $i,j \in \ski$.
\end{defn}


\begin{lem}
For any $\vr \in \bZ^k_{\ge \ve}$,
the canonical morphism
\begin{equation}
\label{eq:37}
(a_{\vr})_{\ast}\varepsilon_{\vr}
\otimes_{\bZ}
\Omega^n_X(\log \cM_X)
\longrightarrow
(a_{\vr})_{\ast}(
\varepsilon_{\vr}
\otimes_{\bZ}
\Omega^n_{X_{\vr}}(\log \cM_{X_{\vr}}))
\end{equation}
is surjective.
\end{lem}
\begin{proof}
It suffices to consider the stalks
at the origin $x=0$ of a local model
$(U, \cM_U)$ in \ref{para:1}.
Then
\begin{gather}
((a_{\vr})_{\ast}\varepsilon_{\vr} \otimes_{\bZ} \Omega^n_U(\log \cM_U))_x
=
\bigoplus_{\ula \in S_{\vr}(\Lambda)}
\varepsilon(\ula) \otimes_{\bZ} \Omega^n_U(\log \cM_U)_x
\label{eq:39} \\
(a_{\vr})_{\ast}(
\varepsilon_{\vr}
\otimes_{\bZ}
\Omega^n_{U_{\vr}}(\log \cM_{U_{\vr}}))_x
=
\bigoplus_{\ula \in S_{\vr}(\Lambda)}
\varepsilon(\ula) \otimes_{\bZ} \Omega^n_{D[\ula]}(\log \cM_{D[\ula]})_x
\label{eq:40}
\end{gather}
and the stalk of the morphism \eqref{eq:37}
is the direct sum of $\id \otimes a[\ula]^{\ast}$
over all $\ula \in S_{\vr}(\Lambda)$,
where
$a[\ula]^{\ast} \colon \Omega^n_U(\log \cM_U)_x \longrightarrow
\Omega^n_{D[\ula]}(\log \cM_{D[\ula]})_x$
is the surjection induced from the closed immersion
$a[\ula] \colon D[\ula] \hookrightarrow U$.
\end{proof}

\begin{lem}
\label{lem:9}
For all $i=\ki$,
the composite
\begin{equation}
\label{eq:38}
\begin{split}
(a_{\vr})_{\ast}\varepsilon_{\vr} \otimes_{\bZ} \Omega^n_X(\log \cM_X)
\xrightarrow{\delta_i \otimes \id}
(a_{\vr+\ve_i})_{\ast}&\varepsilon_{\vr+\ve_i}
\otimes_{\bZ} \Omega^n_X(\log \cM_X) \\
&\longrightarrow
(a_{\vr+\ve_i})_{\ast}(
\varepsilon_{\vr+\ve_i}
\otimes_{\bZ}
\Omega^n_{X_{\vr+\ve_i}}(\log \cM_{X_{\vr+\ve_i}})
)
\end{split}
\end{equation}
factors through the surjection \eqref{eq:37}.
\end{lem}
\begin{proof}
It is enough to consider the stalk of the morphism \eqref{eq:38}
at the origin $x=0$ of a local model
$(U, \cM_U)$ as above.
Under \eqref{eq:39} for $\vr$ and \eqref{eq:40} for $\vr+\ve_i$,
the stalk of \eqref{eq:38} at $x$ is the direct sum of
$\sum_{\lambda \in \Lambda_i \setminus (\ula \cap \Lambda_i)}
(e_{\lambda} \wedge) \otimes a[\ula \cup \{\lambda\}]^{\ast}_x$
for all $\ula \in S_{\vr}(\Lambda)$.
Because
$a[\ula \cup \{\lambda\}] \colon D[\ula \cup \{\lambda\}] \hookrightarrow U$
factors as
$D[\ula \cup \{\lambda\}] \hookrightarrow D[\ula] \hookrightarrow U$,
we obtain the conclusion.
\end{proof}

\begin{defn}
\label{defn:32}
By the lemma above,
a morphism of $\cO_X$-modules
\begin{equation}
\label{eq:14}
(a_{\vr})_{\ast}(
\varepsilon_{\vr}
\otimes_{\bZ}
\Omega^n_{X_{\vr}}(\log \cM_{X_{\vr}}))
\longrightarrow
(a_{\vr+\ve_i})_{\ast}(
\varepsilon_{\vr+\ve_i}
\otimes_{\bZ}
\Omega^n_{X_{\vr+\ve_i}}(\log \cM_{X_{\vr+\ve_i}})),
\end{equation}
is induced from \eqref{eq:38} for every $i=\ki$.
This morphism is denoted by $\delta_i$ again
by abuse of notation.
Then the same equalities as \eqref{eq:41} hold trivially.
\end{defn}

\begin{rmk}
\label{rmk:11}
We look at the stalk of the morphism \eqref{eq:14}.
As in the proof of Lemma \ref{lem:9},
it suffices to consider the stalk at the origin $x=0$
of a local model $(U, \cM_U)$.
Under \eqref{eq:40} for $\vr$ and $\vr+\ve_i$,
the stalk of \eqref{eq:14} at $x$ is the direct sum of
$\sum_{\lambda \in \Lambda_i \setminus (\ula \cap \Lambda_i)}
(e_{\lambda} \wedge) \otimes (-)|_{D[\ula \cup \{\lambda\}]}$
for all $\ula \in S_{\vr}(\Lambda)$,
where $(-)|_{D[\ula \cup \{\lambda\}]}$ denotes the restriction morphism
from $D[\ula]$ to $D[\ula \cup \{\lambda\}]$.
Therefore the equality
\begin{equation}
\label{eq:103}
\delta_i \cdot (a_{\vr})_{\ast}(\id \otimes d)
=(a_{\vr+\ve_i})_{\ast}(\id \otimes d) \cdot \delta_i
\end{equation}
holds for all $i=\ki$.
\end{rmk}

\begin{defn}
An $\cO_X$-module
$\cC(\Omega_{X_{\bullet}}(\log \cM_{X_{\bullet}}))^n$
and a morphism of $\cO_X$-modules
\begin{equation}
d \colon
\cC(\Omega_{X_{\bullet}}(\log \cM_{X_{\bullet}}))^n
\longrightarrow
\cC(\Omega_{X_{\bullet}}(\log \cM_{X_{\bullet}}))^{n+1}
\end{equation}
are defined by
\begin{equation}
\cC(\Omega_{X_{\bullet}}(\log \cM_{X_{\bullet}}))^n
=\bigoplus_{\vr \in \bZ^k_{\ge \ve}}
(a_{\vr})_{\ast}
(\varepsilon_{\vr} \otimes_{\bZ}
\Omega^{n-|\vr|+k}_{X_{\vr}}(\log \cM_{X_{\vr}}))
\end{equation}
and
\begin{equation}
d=\bigoplus_{\vr \in \bZ^k_{\ge \ve}}
((-1)^{|\vr|-k}(a_{\vr})_{\ast}(\id \otimes d)
+\sum_{i=1}^{k}\delta_i),
\end{equation}
where $d$ in the right hand side
is the differential of the complex
$\Omega_{X_{\vr}}(\log \cM_{X_{\vr}})$.
From \eqref{eq:41} and \eqref{eq:103},
the equality $d^2=0$ can be easily checked.
Thus the complex $\cC(\Omega_{X_{\bullet}}(\log \cM_{X_{\bullet}}))$
of $\bC$-sheaves on $X$ is obtained.
By setting
\begin{equation}
(\delta W)_m\cC(\Omega_{X_{\bullet}}(\log \cM_{X_{\bullet}}))^n
=\bigoplus_{\vr \in \bZ^k_{\ge \ve}}
(a_{\vr})_{\ast}
(\varepsilon_{\vr}
\otimes_{\bZ}
W_{m+|\vr|-k}\Omega^{n-|\vr|+k}_{X_{\vr}}(\log \cM_{X_{\vr}}))
\end{equation}
for $m, n \in \bZ$,
we obtain an increasing filtration $\delta W$
on the complex
$\cC(\Omega_{X_{\bullet}}(\log \cM_{X_{\bullet}}))$.
We have
\begin{equation}
\label{eq:48}
\gr_m^{\delta W}\cC(\Omega_{X_{\bullet}}(\log \cM_{X_{\bullet}}))
=\bigoplus_{\vr \in \bZ^k_{\ge \ve}}
(a_{\vr})_{\ast}
(\varepsilon_{\vr}
\otimes_{\bZ}
\gr_{m+|\vr|-k}^W\Omega_{X_{\vr}}(\log \cM_{X_{\vr}}))[-|\vr|+k]
\end{equation}
as complexes for all $m \in \bZ$,
because
$\delta_i((\delta W)_m\cC(\Omega_{X_{\bullet}}(\log \cM_{X_{\bullet}})))
\subset (\delta W)_{m-1}\cC(\Omega_{X_{\bullet}}(\log \cM_{X_{\bullet}}))$
for $i=\ki$
by the description of $\delta_i$ in Remark \ref{rmk:11}.
\end{defn}

Next task is to construct a product
on the complex $\cC(\Omega_{X_{\bullet}}(\log \cM_{X_{\bullet}}))$
as in \cite[Section 2]{FujisawaPLMHS}.

\begin{para}
For every $x \in X$,
a morphism
$\overline{\chi}(\overline{f}^{\flat}_x) \colon
\bigwedge \overline{\cM}_{X,x}\gp
\otimes_{\bZ} \bigwedge \overline{\cM}_{X,x}\gp
\longrightarrow \bigwedge \overline{\cM}_{X,x}\gp$
is induced from the semistable morphism
$\overline{f}^{\flat}_x: \bN^k \longrightarrow \overline{\cM}_{X,x}$
as in Definition \ref{defn:7}.
It is easy to see that
$\overline{\chi}(\overline{f}^{\flat}_x)
(\bigwedge^{\otimes \vr}\overline{\cM}_{X,x}\gp
\otimes
\bigwedge^{\otimes \vs}\overline{\cM}_{X,x}\gp)
\subset
\bigwedge^{\otimes \vr+\vs-\ve}\overline{\cM}_{X,x}\gp$
for all $\vr, \vs \in \bN^k_{\ge \ve}$.
\end{para}

\begin{lem}
\label{lem:28}
There exists a unique morphism
$\overline{\chi} \colon
\bigwedge\overline{\cM}_X\gp \otimes_{\bZ} \bigwedge\overline{\cM}_X\gp
\longrightarrow
\bigwedge\overline{\cM}_X\gp$
such that
$\overline{\chi}_x=\overline{\chi}(\overline{f}^{\flat}_x)$
for all $x \in X$.
\end{lem}
\begin{proof}
Since the uniqueness is trivial,
it suffices to check the existence locally.
Thus we may work over a local model 
$(U, \cM_U)$ in \ref{para:1}.
We note that there exists a chart
$\bN^{\Lambda}_U \longrightarrow \cM_U$
which induces a surjection
$\bZ^{\Lambda}_U \longrightarrow \overline{\cM}_U\gp$.
For $x \in U$,
we set
$\Lambda_x=\{\lambda \in \Lambda \mid x_{\lambda} \notin \cO^{\ast}_{U,x}\}$,
where $x_{\lambda}$ is the coordinate function
corresponding to $\lambda \in \Lambda$ as in \ref{para:1}.
The chart $\bN^{\Lambda}_U \longrightarrow \cM_U$
induces the identification
$\bN^{\Lambda_x} \overset{\simeq}{\longrightarrow} \overline{\cM}_{U,x}$
for all $x \in U$.
On the other hand,
the partition $\Lambda=\coprod_{i=1}^k\Lambda_i$ as in \ref{para:1}
induces a morphism
$\overline{\chi}(\Lambda) \colon
\bigwedge\bZ^{\Lambda} \otimes_{\bZ} \bigwedge\bZ^{\Lambda}
\longrightarrow \bigwedge\bZ^{\Lambda}$.
By Remark \ref{rmk:2},
the diagram
\begin{equation}
\begin{CD}
\bigwedge\bZ^{\Lambda} \otimes_{\bZ} \bigwedge\bZ^{\Lambda}
@>{\overline{\chi}(\Lambda)}>>
\bigwedge\bZ^{\Lambda} \\
@VVV @VVV \\
\bigwedge\bZ^{\Lambda_x} \otimes_{\bZ} \bigwedge\bZ^{\Lambda_x}
@>{\overline{\chi}(\Lambda_x)}>>
\bigwedge\bZ^{\Lambda_x} \\
@V{\simeq}VV @VV{\simeq}V \\
\bigwedge \overline{\cM}_{U,x}\gp \otimes_{\bZ} \bigwedge \overline{\cM}_{U,x}\gp
@>>{\overline{\chi}(\overline{f}^{\flat}_x)}>
\bigwedge \overline{\cM}_{U,x}\gp
\end{CD}
\end{equation}
is commutative for all $x \in U$.
Therefore the composite
\begin{equation}
\bigwedge\bZ^{\Lambda}_U \otimes_{\bZ} \bigwedge\bZ^{\Lambda}_U
\overset{\overline{\chi}(\Lambda)}{\longrightarrow}
\bigwedge\bZ^{\Lambda}_U
\longrightarrow
\bigwedge \overline{\cM}_U\gp
\end{equation}
factors through the surjection
$\bigwedge\bZ^{\Lambda}_U \otimes_{\bZ} \bigwedge\bZ^{\Lambda}_U
\longrightarrow
\bigwedge \overline{\cM}_U\gp \otimes_{\bZ} \bigwedge \overline{\cM}_U\gp$
and induces the morphism
$\bigwedge \overline{\cM}_U\gp \otimes_{\bZ} \bigwedge \overline{\cM}_U\gp
\longrightarrow \bigwedge \overline{\cM}_U\gp$
as desired.
\end{proof}

\begin{defn}
The restriction of $\overline{\chi}$ to
$\bigwedge^{\otimes \vr}\overline{\cM}_X\gp
\otimes_{\bZ}
\bigwedge^{\otimes \vs}\overline{\cM}_X\gp$
gives us a morphism
$\bigwedge^{\otimes \vr}\overline{\cM}_X\gp
\otimes_{\bZ}
\bigwedge^{\otimes \vs}\overline{\cM}_X\gp
\longrightarrow
\bigwedge^{\otimes \vr+\vs-\ve}\overline{\cM}_X\gp$
by definition.
Therefore, the morphism $\overline{\chi}$
induces a morphism
\begin{equation}
(a_{\vr})_{\ast}\varepsilon_{\vr}
\otimes_{\bZ} (a_{\vs})_{\ast}\varepsilon_{\vs}
\longrightarrow
(a_{\vr+\vs-\ve})_{\ast}\varepsilon_{\vr+\vs-\ve}
\end{equation}
via the isomorphism \eqref{eq:36}.
It is also denoted by $\overline{\chi}$ by abuse of the notation.
\end{defn}

\begin{rmk}
The equalities
\begin{equation}
\label{eq:47}
\overline{\chi} \cdot (\delta_i \otimes \id)
=
(-1)^{|\vr|-k}\overline{\chi} \cdot (\id \otimes \delta_i)
=
\delta_i \cdot \overline{\chi}
\end{equation}
can be easily checked for $i=\ki$.
\end{rmk}

\begin{lem}
For $\vr, \vs \in \bZ^k_{\ge \ve}$
and for $p,q \in \bZ$,
we define a morphism of $\cO_X$-modules
\begin{equation}
\label{eq:43}
\begin{split}
((a_{\vr})_{\ast}\varepsilon_{\vr}
\otimes_{\bZ}
\Omega^p_X(\log \cM_X))
&\otimes_{\bC}
((a_{\vs})_{\ast}\varepsilon_{\vs}
\otimes_{\bZ}
\Omega^q_X(\log \cM_X)) \\
&\longrightarrow
(a_{\vr+\vs-\ve})_{\ast}
(\varepsilon_{\vr+\vs-\ve}
\otimes_{\bZ}
\Omega^{p+q}_{X_{\vr+\vs-\ve}}(\log \cM_{X_{\vr+\vs-\ve}}))
\end{split}
\end{equation}
as the composite of the three morphisms,
the isomorphism
\begin{equation}
\begin{split}
((a_{\vr})_{\ast}\varepsilon_{\vr}
\otimes_{\bZ}
\Omega^p_X(\log \cM_X))
&\otimes_{\bC}
((a_{\vs})_{\ast}\varepsilon_{\vs}
\otimes_{\bZ}
\Omega^q_X(\log \cM_X)) \\
&\simeq
((a_{\vr})_{\ast}\varepsilon_{\vr}
\otimes_{\bZ}
(a_{\vs})_{\ast}\varepsilon_{\vs})
\otimes_{\bZ}
(\Omega^p_X(\log \cM_X)
\otimes_{\bC}
\Omega^q_X(\log \cM_X))
\end{split}
\end{equation}
exchanging the middle terms,
the morphism
$\overline{\chi} \otimes \wedge$,
and the surjection \eqref{eq:37} for $\vr+\vs-\ve \in \bZ^k_{\ge \ve}$.
Then this morphism factors through the surjection
\begin{equation}
\begin{split}
((a_{\vr})_{\ast}\varepsilon_{\vr}
\otimes_{\bZ}
&\Omega^p_X(\log \cM_X))
\otimes_{\bC}
((a_{\vs})_{\ast}\varepsilon_{\vs}
\otimes_{\bZ}
\Omega^q_X(\log \cM_X)) \\
&\longrightarrow
(a_{\vr})_{\ast}(
\varepsilon_{\vr}
\otimes_{\bZ}
\Omega^p_{X_{\vr}}(\log \cM_{X_{\vr}}))
\otimes_{\bC}
(a_{\vs})_{\ast}(
\varepsilon_{\vs}
\otimes_{\bZ}
\Omega^q_{X_{\vs}}(\log \cM_{X_{\vs}}))
\end{split}
\end{equation}
induced from the surjections \eqref{eq:37}
for $\vr$ and $\vs$.
\end{lem}
\begin{proof}
Similar to the proof of Lemmas \ref{lem:9} and \ref{lem:28}.
\end{proof}

\begin{defn}
From the lemma above,
the morphism \eqref{eq:43}
induces a morphism
\begin{equation}
\begin{split}
(a_{\vr})_{\ast}(
\varepsilon_{\vr}
\otimes_{\bZ}
\Omega^p_{X_{\vr}}(\log \cM_{X_{\vr}}))
\otimes_{\bC}
(&(a_{\vs})_{\ast}(
\varepsilon_{\vs}
\otimes_{\bZ}
\Omega^q_{X_{\vs}}(\log \cM_{X_{\vs}})) \\
&\longrightarrow
(a_{\vr+\vs-\ve})_{\ast}
(\varepsilon_{\vr+\vs-\ve}
\otimes_{\bZ}
\Omega^{p+q}_{X_{\vr+\vs-\ve}}(\log \cM_{X_{\vr+\vs-\ve}}))
\end{split}
\end{equation}
for $\vr, \vs \in \bZ^k_{\ge \ve}$ and for $p,q \in \bZ$,
which is denoted by $\overline{\chi} \otimes \wedge$
by abuse of the notation.
\end{defn}

\begin{defn}
A morphism of $\bC$-sheaves
\begin{equation}
\tau \colon
\cC(\Omega_{X_{\bullet}}(\log \cM_{X_{\bullet}}))^p
\otimes_{\bC}
\cC(\Omega_{X_{\bullet}}(\log \cM_{X_{\bullet}}))^q
\longrightarrow
\cC(\Omega_{X_{\bullet}}(\log \cM_{X_{\bullet}}))^{p+q}
\end{equation}
for $p,q \in \bZ$
is defined by
\begin{equation}
\tau
=
(-1)^{(|\vs|-k)(p-|\vr|+k)}
\prod_{i=1}^k
\frac{(r_i-1)!(s_i-1)!}{(r_i+s_i-1)!}
\overline{\chi} \otimes \wedge
\end{equation}
on the direct summand
\begin{equation}
(a_{\vr})_{\ast}
(\varepsilon_{\vr}
\otimes_{\bZ}
\Omega^{p-|\vr|+k}_{X_{\vr}}(\log \cM_{X_{\vr}}))
\otimes_{\bC}
(a_{\vs})_{\ast}
(\varepsilon_{\vs}
\otimes_{\bZ}
\Omega^{q-|\vs|+k}_{X_{\vs}}(\log \cM_{X_{\vs}}))
\end{equation}
of
$\cC(\Omega_{X_{\bullet}}(\log \cM_{X_{\bullet}}))^p
\otimes_{\bC}
\cC(\Omega_{X_{\bullet}}(\log \cM_{X_{\bullet}}))^q$.
We can check that these morphisms
define a morphism of complexes of $\bC$-sheaves
\begin{equation}
\tau \colon
\cC(\Omega_{X_{\bullet}}(\log \cM_{X_{\bullet}}))
\otimes_{\bC}
\cC(\Omega_{X_{\bullet}}(\log \cM_{X_{\bullet}}))
\longrightarrow
\cC(\Omega_{X_{\bullet}}(\log \cM_{X_{\bullet}}))
\end{equation}
by the direct computation using \eqref{eq:47}.
The inclusion
\begin{equation}
\tau(
\delta W_a\cC(\Omega_{X_{\bullet}}(\log \cM_{X_{\bullet}}))
\otimes_{\bC}
\delta W_b\cC(\Omega_{X_{\bullet}}(\log \cM_{X_{\bullet}}))
\subset
\delta W_{a+b}\cC(\Omega_{X_{\bullet}}(\log \cM_{X_{\bullet}}))
\end{equation}
for all $a, b \in \bZ$
can be easily checked from the definition above.
\end{defn}

\begin{rmk}
\label{rmk:13}
Direct computation using \eqref{eq:20} shows that
$\tau$ is compatible with the isomorphism
\begin{equation}
\cC(\Omega_{X_{\bullet}}(\log \cM_{X_{\bullet}}))
\otimes_{\bC}
\cC(\Omega_{X_{\bullet}}(\log \cM_{X_{\bullet}}))
\simeq
\cC(\Omega_{X_{\bullet}}(\log \cM_{X_{\bullet}}))
\otimes_{\bC}
\cC(\Omega_{X_{\bullet}}(\log \cM_{X_{\bullet}}))
\end{equation}
exchanging the left and right hand sides
defined in \ref{para:6}.
\end{rmk}

\begin{assump}
In the remainder of this section,
$X$ is assumed to be of pure dimension.
\end{assump}

\begin{para}
\label{para:8}
We consider
the $E_1$-terms of the spectral sequence
\begin{equation}
\label{eq:56}
E_r^{p,q}
(R\Gamma_c(X, \cC(\Omega_{X_{\bullet}}(\log \cM_{X_{\bullet}}))), \delta W).
\end{equation}
By \eqref{eq:48},
we have
\begin{equation}
\label{eq:42}
\begin{split}
E_1^{p,q}
(R\Gamma_c(X, \cC(\Omega_{X_{\bullet}}&(\log \cM_{X_{\bullet}}))), \delta W) \\
&\simeq
\bigoplus_{\vr \in \bZ^k_{\ge \ve}}
\coh^{p+q-|\vr|+k}_c(X_{\vr},
\varepsilon_{\vr} \otimes \gr_{-p+|\vr|-k}^W\Omega_{X_{\vr}}(\log \cM_{X_{\vr}}))
\end{split}
\end{equation}
for all $p,q$
because $a_{\vr}$ is a finite morphism.
\end{para}

The following lemma is a special case
of \cite[Lemma 3.23]{FujisawaMHSLSD}.

\begin{lem}
\label{lem:3}
For $\vr \in \bZ^k_{\ge \ve}$,
there exists an isomorphism of complexes of $\bC$-sheaves
\begin{equation}
\label{eq:46}
\bigoplus_{l=0}^m
(i_l)_{\ast}(i_l^{-1}(
\varepsilon_{\vr} \otimes_{\bZ}
\bigwedge^{m-l}L_{\vr}) \otimes_{\bZ}
\varepsilon^l \otimes_{\bZ} \Omega_{\widetilde{D_{\vr}}^l}[-m])
\overset{\simeq}{\longrightarrow}
\varepsilon_{\vr} \otimes_{\bZ}
\gr_m^W\Omega_{X_{\vr}}(\log \cM_{X_{\vr}})
\end{equation}
where $\widetilde{D_{\vr}}^l$, $i_l$
and $\varepsilon^l$ are defined in \textup{\cite[(3.1.4)]{DeligneII}}.
\end{lem}

\begin{para}
The composite of the inverse of \eqref{eq:46} for $m=|\vr|$,
the projection for $l=0$
and $\vartheta_{\vr} \otimes \id$
gives us a morphism
$\varepsilon_{\vr} \otimes_{\bZ}
\gr_{|\vr|}^W\Omega_{X_{\vr}}(\log \cM_{X_{\vr}})
\longrightarrow
\Omega_{X_{\vr}}[-|\vr|]$.
Then a morphism
\begin{equation}
(a_{\vr})_{\ast}(
\epsilon_{\vr} \otimes_{\bZ}
\gr_{|\vr|}^W\Omega_{X_{\vr}}(\log \cM_{X_{\vr}}))
\longrightarrow
(a_{\vr})_{\ast}\Omega_{X_{\vr}}[-|\vr|]
\end{equation}
is obtained.
By taking the direct sum for all $\vr \in \bZ^k_{\ge \ve}$,
a morphism
\begin{equation}
\label{eq:53}
\gr_k^{\delta W}\cC(\Omega_{X_{\bullet}}(\log \cM_{X_{\bullet}}))
\longrightarrow
\bigoplus_{\vr \in \bZ^k_{\ge \ve}}
(a_{\vr})_{\ast}\Omega_{X_{\vr}}[-2|\vr|+k]
\end{equation}
is obtained by \eqref{eq:48}.
Similarly, the composite of the inverse of \eqref{eq:46} for $m=|\vr|+1$,
the projection for $l=1$
and $i^{-1}(\vartheta_{\vr} \otimes \id)$
induces a morphism
\begin{equation}
\label{eq:70}
\gr_{k+1}^{\delta W}\cC(\Omega_{X_{\bullet}}(\log \cM_{X_{\bullet}}))
\longrightarrow
\bigoplus_{\vr \in \bZ^k_{\ge \ve}}
(a_{\vr} \cdot i)_{\ast}
(\varepsilon \otimes_{\bZ} \Omega_{\widetilde{D}_{\vr}}[-2|\vr|+k-1])
\end{equation}
by \eqref{eq:48} again,
where we use $\widetilde{D_{\vr}}$, $i$ and $\varepsilon$
instead of $\widetilde{D_{\vr}}^1$, $i_1$ and $\varepsilon^1$
in \eqref{eq:46} respectively.
\end{para}

\begin{lem}
We have the isomorphisms
\begin{gather}
E_1^{-k,2\dim X+2k}
(R\Gamma_c(X, \cC(\Omega_{X_{\bullet}}(\log \cM_{X_{\bullet}}))), \delta W)
\simeq
\bigoplus_{\vr \in \bZ^k_{\ge \ve}}
\coh^{2\dim X_{\vr}}_c
(X_{\vr}, \Omega_{X_{\vr}})
\label{eq:54} \\
E_1^{-k-1,2\dim X+2k}
(R\Gamma_c(X, \cC(\Omega_{X_{\bullet}}(\log \cM_{X_{\bullet}}))), \delta W)
\simeq
\bigoplus_{\vr \in \bZ^k_{\ge \ve}}
\coh^{2\dim X_{\vr}-2}_c
(\widetilde{D_{\vr}}, \varepsilon \otimes_{\bZ} \Omega_{\widetilde{D_{\vr}}})
\label{eq:71}
\end{gather}
induced from the morphisms \eqref{eq:53}
and \eqref{eq:70} respectively.
\end{lem}
\begin{proof}
Combining \eqref{eq:42} and \eqref{eq:46},
we have
\begin{equation}
\begin{split}
E_1^{-k-1,2\dim X+2k}
(R\Gamma_c&(X, \cC(\Omega_{X_{\bullet}}(\log \cM_{X_{\bullet}}))), \delta W) \\
&\simeq
\bigoplus_{\vr \in \bZ^k_{\ge \ve}}
\bigoplus_{l=0}^{|\vr|+1}
\coh^{2\dim X-2|\vr|+2k-2}_c
(\widetilde{D_{\vr}}^l,
i_l^{-1}(
\varepsilon_{\vr} \otimes_{\bZ}
\bigwedge^{|\vr|+1-l}L_{\vr}) \otimes_{\bZ}
\varepsilon^l \otimes_{\bZ} \Omega_{\widetilde{D_{\vr}}^l}).
\end{split}
\end{equation}
Because the inequalities
$2\dim X-2|\vr|+2k-2 \le 2\dim \widetilde{D_{\vr}}^l=2\dim X-2|\vr|+2k-2l$
implies $l \le 1$,
we obtain \eqref{eq:71}
from the equality $\rank L_{\vr}=|\vr|$.
Similar argument shows \eqref{eq:54}.
\end{proof}

\begin{lem}
\label{lem:12}
Let $X=\bigcup_{\alpha \in A}V_{\alpha}$ be an open covering of $X$.
Then the canonical morphism
\begin{equation}
\label{eq:72}
\begin{split}
\bigoplus_{\alpha \in A}
E_1^{-k-1,2\dim X+2k}
(R\Gamma_c&(V_{\alpha},
\cC(\Omega_{X_{\bullet}}(\log \cM_{X_{\bullet}}))), \delta W) \\
&\longrightarrow
E_1^{-k-1,2\dim X+2k}
(R\Gamma_c(X, \cC(\Omega_{X_{\bullet}}(\log \cM_{X_{\bullet}}))), \delta W)
\end{split}
\end{equation}
is surjective.
\end{lem}
\begin{proof}
By the isomorphisms \eqref{eq:71}
for $X$ and $V_{\alpha}$,
the morphism \eqref{eq:72} is identified with
the direct sum of the morphisms
\begin{equation}
\label{eq:116}
\bigoplus_{\alpha \in A}
\coh_c^{2\dim \widetilde{D_{\vr}}}
(i^{-1}(V_{\alpha}), \varepsilon \otimes_{\bZ} \Omega_{\widetilde{D_{\vr}}})
\longrightarrow
\coh_c^{2\dim \widetilde{D_{\vr}}}
(\widetilde{D_{\vr}}, \varepsilon \otimes_{\bZ} \Omega_{\widetilde{D_{\vr}}}),
\end{equation}
induced from the surjection
$\bigoplus_{\alpha \in A}(j_{\alpha})_{!}
(\varepsilon \otimes_{\bZ} \Omega_{\widetilde{D_{\vr}}})|_{i^{-1}(V_{\alpha})}
\longrightarrow
\varepsilon \otimes_{\bZ} \Omega_{\widetilde{D_{\vr}}}$,
where $j_{\alpha}$ denotes the open immersion
$i^{-1}(V_{\alpha}) \hookrightarrow D_{\vr}$
for every $\alpha \in A$.
Therefore the morphism \eqref{eq:116} is surjective.
\end{proof}

\begin{defn}
\label{notn:5}
We set
$\epsilon(a)=(-1)^{a(a-1)/2}$
for $a \in \bZ$
as in \cite[(3.3)]{Guillen-NavarroAznarCI}
and \cite[\RomI-14]{Peters-SteenbrinkMHS}.
\end{defn}

\begin{defn}
\label{defn:15}
A morphism
$\Theta \colon
E_1^{-k,2\dim X+2k}
(R\Gamma_c(X, \cC(\Omega_{X_{\bullet}}(\log \cM_{X_{\bullet}}))), \delta W)
\longrightarrow \bC$
is defined by
\begin{equation}
\label{eq:19}
\Theta
=\bigoplus_{\vr \in \bZ^k_{\ge \ve}}
\epsilon(|\vr|-k)(2\pi\sqrt{-1})^{|\vr|-k}\int_{X_{\vr}}
\colon
\bigoplus_{\vr \in \bZ^k_{\ge \ve}}
\coh^{2\dim X_{\vr}}_c
(X_{\vr}, \Omega_{X_{\vr}})
\longrightarrow
\bC
\end{equation}
via the isomorphism \eqref{eq:54}.
\end{defn}

\begin{lem}
\label{lem:13}
$\Theta \cdot d_1=0$,
where $d_1$ is the morphism of $E_1$-terms
of the spectral sequence \eqref{eq:56}.
\end{lem}
\begin{proof}
In this proof,
we use $E_1^{p,q}$
instead of
$E_1^{p,q}
(R\Gamma_c(X, \cC(\Omega_{X_{\bullet}}(\log \cM_{X_{\bullet}}))), \delta W)$
for short.
Since Lemma \ref{lem:12} enables us
to compute $\Theta \cdot d_1$ locally,
we may work on a local model
$(U, \cM_U)$ in \ref{para:1}.
We use the notation in \ref{para:1}.
In addition, we fix an total order on $\Lambda$
and identify $\varepsilon(\ula)$ with $\bZ$
by fixing the base
$\ve_{\lambda_1} \wedge \ve_{\lambda_2} \wedge \dots \wedge \ve_{\lambda_m}$
for $\ula=\{\lambda_1, \lambda_2, \dots, \lambda_m\}$
with $\lambda_1 < \lambda_2 < \dots < \lambda_m$.
Then
\begin{gather}
E_1^{-k,2\dim U+2k}
\simeq
\bigoplus_{\vr \in \bZ^k_{\ge \ve}}
\bigoplus_{\ula \in S_{\vr}(\Lambda)}
\coh^{2\dim D[\ula]}_c
(D[\ula], \Omega_{D[\ula]}) \\
E_1^{-k-1,2\dim U+2k}
\simeq
\bigoplus_{\vr \in \bZ^k_{\ge \ve}}
\bigoplus_{\ula \in S_{\vr}(\Lambda)}
\bigoplus_{\lambda \in \Lambda \setminus \ula}
\coh^{2\dim D[\ula]-2}_c
(D[\ula] \cap D_{\lambda}, \Omega_{D[\ula] \cap D_{\lambda}})
\end{gather}
by \eqref{eq:54} and \eqref{eq:71}.
We fix
$\vr \in \bZ^k_{\ge \ve}, \ula \in S_{\vr}(\Lambda)$
and $\lambda \in \Lambda \setminus \ula$.
For any
\begin{equation}
\omega
\in
\coh^{2\dim D[\ula]-2}_c
(D[\ula] \cap D_{\lambda}, \Omega_{D[\ula] \cap D_{\lambda}})
\end{equation}
we have
\begin{equation}
\begin{split}
d_1(\omega)
=(-1)^{|\vr|-k+j}&\gamma(\omega)
+(-1)^j\omega \\
&\in
\coh^{2\dim D[\ula]}_c
(D[\ula], \Omega_{D[\ula]})
\oplus
\coh^{2\dim D[\ula]-2}_c
(D[\ula] \cap D_{\lambda}, \Omega_{D[\ula] \cap D_{\lambda}})
\end{split}
\end{equation}
as in the proof of \cite[Lemma 7.10]{FujisawaPLMHS},
where $\gamma$ denotes the Gysin morphism
for $D[\ula] \cap D_{\lambda}$ in $D[\ula]$
as in \cite[4.2]{FujisawaPLMHS}
and where
$\ula \cup \{\lambda\}
=\{\lambda_0, \lambda_1, \dots, \lambda=\lambda_j, \dots, \lambda_m\}$
with $\lambda_0 < \lambda_1 < \dots < \lambda_m$.
Then
\begin{equation}
\begin{split}
\Theta(d_1(\omega))
&=(-1)^{|\vr|-k+j}\epsilon(|\vr|-k)(2\pi\sqrt{-1})^{|\vr|-k}
\int_{D[\ula]}\gamma(\omega) \\
&\phantom{(-1)^{|\vr|-k+j}\epsilon(|\vr|-k)}
+(-1)^j\epsilon(|\vr|-k+1)(2\pi\sqrt{-1})^{|\vr|-k+1}
\int_{D[\ula] \cap D_{\lambda}}\omega \\
&=(-1)^{|\vr|-k+j}\epsilon(|\vr|-k)(2\pi\sqrt{-1})^{|\vr|-k}
(\int_{D[\ula]}\gamma(\omega)
+(2\pi\sqrt{-1})\int_{D[\ula] \cap D_{\lambda}}\omega) \\
&=0
\end{split}
\end{equation}
by $\epsilon(a+1)=(-1)^a\epsilon(a)$
and by \cite[\S 2 (b)]{GriffithsSchmid}.
\end{proof}

\section{Products}
\label{sec:constr-bilin-forms}

In this section,
we construct two products;
one is the morphism
$A_{\bC} \otimes_{\bC} \Omega_X \longrightarrow A_{\bC}$
in Definition \ref{defn:22}
and the other
$A_{\bC} \otimes_{\bC} A_{\bC}
\longrightarrow \cC(\Omega_{X_{\bullet}}(\log \cM_{X_{\bullet}}))[k]$
in Definition \ref{defn:13}.
The construction of the first one
is straightforward.
To define the second,
we use the morphism $\tau$
on $\cC(\Omega_{X_{\bullet}}(\log \cM_{X_{\bullet}}))$.

\begin{defn}
\label{defn:22}
Morphisms of $\bC$-sheaves
given by
\begin{equation}
(\bC[\vu] \otimes_{\bC} \Omega^p_X(\log \cM_X)) \otimes_{\bC} \Omega^q_X
\ni
(P \otimes \omega) \otimes \eta
\mapsto
P \otimes \omega \wedge \eta
\in
\bC[\vu] \otimes_{\bC} \Omega^{p+q}_X(\log \cM_X)
\end{equation}
for all $p,q$
define a morphism of complexes
$(\bC[\vu] \otimes_{\bC} \Omega_X(\log \cM_X)) \otimes_{\bC} \Omega_X
\longrightarrow
\bC[\vu] \otimes_{\bC} \Omega_X(\log \cM_X)$,
which sends
$W(I)_m(\bC[\vu] \otimes_{\bC} \Omega_X(\log \cM_X)) \otimes_{\bC} \Omega_X$
and 
$L(I)_m(\bC[\vu] \otimes_{\bC} \Omega_X(\log \cM_X)) \otimes_{\bC} \Omega_X)$
to $W(I)_m(\bC[\vu] \otimes_{\bC} \Omega_X(\log \cM_X))$
and $L(I)_m(\bC[\vu] \otimes_{\bC} \Omega_X(\log \cM_X)))$
for all $I \subset \ski$ and $m \in \bZ$ respectively.
Thus a morphism of complexes
\begin{equation}
\label{eq:11}
\overline{\Psi} \colon
A_{\bC} \otimes_{\bC} \Omega_X \longrightarrow A_{\bC}
\end{equation}
is induced.
This morphism satisfies
$\overline{\Psi}(L(I)_mA_{\bC} \otimes_{\bC} \Omega_X)
\subset L(I)_mA_{\bC}$
for all $I \subset \ski$ and $m \in \bZ$.
The morphism
$\gr_m^LA_{\bC} \otimes_{\bC} \Omega_X
\longrightarrow \gr_m^LA_{\bC}$
induced from $\overline{\Psi}$
is denoted by $\gr_m^L\overline{\Psi}$.
The morphism $\overline{\Psi}$
induces a morphism
\begin{equation*}
\coh^{a,b}(X, \overline{\Psi})
\colon
\coh^a(X, A_{\bC}) \otimes_{\bC} \coh^b(X, \Omega_X)
\longrightarrow
\coh^{a+b}(X, A_{\bC})
\end{equation*}
as in Definition \ref{defn:10}.
For $\omega \in \coh^a(X, A_{\bC})$ and $\eta \in \coh^b(X, \Omega_X)$,
the element
$\coh^{a,b}(X, \overline{\Psi})(\omega \otimes \eta) \in \coh^{a+b}(X, A_{\bC})$
is simply denoted by $\omega \cup \eta$
and called the cup product of $\omega$ and $\eta$.
\end{defn}

\begin{rmk}
It is trivial that the diagram
\begin{equation}
\begin{CD}
\Omega_{X/\ast}(\log (\cM_X/\bN^k)) \otimes_{\bC} \Omega_X
@>>> \Omega_{X/\ast}(\log (\cM_X/\bN^k)) \\
@V{\theta \otimes \id}VV @VV{\theta}V \\
A_{\bC} \otimes_{\bC} \Omega_X @>>{\overline{\Psi}}> A_{\bC}
\end{CD}
\end{equation}
is commutative,
where the top horizontal arrow
is the morphism defined in Definition \ref{defn:25},
and $\theta$ is the morphism defined in \eqref{eq:28}.
Therefore the morphism $\cup c(\cL)$ in \eqref{eq:108}
is identified with the morphism
\begin{equation}
\cup c(\cL) \colon
\coh^a(X, A_{\bC}) \longrightarrow \coh^{a+2}(X, A_{\bC})
\end{equation}
via the isomorphisms \eqref{eq:115}
induced by $\theta$.
\end{rmk}

The following lemma
computes $\gr_m^L\overline{\Psi}$
via the isomorphism \eqref{eq:67}.

\begin{lem}
\label{lem:7}
For $\vq, \vr \in \bN^k$
satisfying the conditions in \eqref{eq:61},
and for
\begin{equation}
\vu^{\vq} \otimes \vv \otimes \omega
\in
(a_{\vr})_{\ast}(\varepsilon_{\vr}
\otimes_{\bZ}
\Omega^{p-m-2|\vq|}_{X_{\vr}}), 
\quad
\eta
\in
\Omega^q_X,
\end{equation}
the image of
$(\vu^{\vq} \otimes \vv \otimes \omega) \otimes \eta$
by the morphism $\gr_m^L\overline{\Psi}$
via the identification \eqref{eq:67} is
\begin{equation}
\vu^{\vq} \otimes \vv \otimes \omega \wedge (a_{\vr})^{\ast}\eta
\in
(a_{\vr})_{\ast}(\varepsilon_{\vr}
\otimes_{\bZ}
\Omega^{p+q-m-2|\vq|}_{X_{\vr}}).
\end{equation}
\end{lem}
\begin{proof}
By the 
direct computation
in the local situation \ref{para:1}.
\end{proof}

\begin{defn}
A morphism of $\cO_X$-modules
\begin{equation}
\cC(\dlog t_i\wedge) \colon
\cC(\Omega_{X_{\bullet}}(\log \cM_{X_{\bullet}}))^n
\longrightarrow
\cC(\Omega_{X_{\bullet}}(\log \cM_{X_{\bullet}}))^{n+1}
\end{equation}
is defined by
$\cC(\dlog t_i\wedge)
=
\bigoplus_{\vr \in \bZ^k_{\ge \ve}}
(-1)^{|\vr|-k}(a_{\vr})_{\ast}(\id \otimes \dlog t_i\wedge)$
for all $n$.
Then these morphisms define a morphism of filtered complexes
\begin{equation}
\cC(\dlog t_i\wedge) \colon
(\cC(\Omega_{X_{\bullet}}(\log \cM_{X_{\bullet}})), \delta W)
\longrightarrow
(\cC(\Omega_{X_{\bullet}}(\log \cM_{X_{\bullet}}))[1], \delta W[-1]).
\end{equation}
The equality
$\cC(\dlog t_i \wedge) \cdot \cC(\dlog t_j \wedge)
+\cC(\dlog t_j \wedge) \cdot \cC(\dlog t_i \wedge)=0$
holds for all $i, j \in \ski$.
\end{defn}

\begin{defn}
For $\vq \in \bN^k$,
a morphism of $\cO_X$-modules
\begin{equation}
\res_{\vq+\ve} \colon
\Omega^{n+k}_X(\log \cM_X)/
\sum_{i=1}^{k}W(i)_{q_i}
\longrightarrow
\cC(\Omega_{X_{\bullet}}(\log \cM_{X_{\bullet}}))^n
\end{equation}
is obtained
as the composite of the morphism $\res_{\vq+\ve}$ in \eqref{eq:57}
and the inclusion
$(a_{\vq+\ve})_{\ast}
(\varepsilon_{\vq+\ve}
\otimes_{\bZ}
\Omega^{n-|\vq|}_{X_{\vq+\ve}}(\log \cM_{X_{\vq+\ve}}))
\hookrightarrow
\cC(\Omega_{X_{\bullet}}(\log \cM_{X_{\bullet}}))^n$.
By the identification in \eqref{eq:58},
a morphism of $\cO_X$-modules
\begin{equation}
\res=
\bigoplus_{\vq \in \bN^k}
(-1)^{k|\vq|}\res_{\vq+\ve} \colon
A^n_{\bC}
\longrightarrow
\cC(\Omega_{X_{\bullet}}(\log \cM_{X_{\bullet}}))^n
\end{equation}
is defined for all $n \in \bZ$.
Lemma \ref{lem:5}
implies
$\res(L_mA_{\bC}) \subset
(\delta W)_m\cC(\Omega_{X_{\bullet}}(\log \cM_{X_{\bullet}}))^n$
for all $m \in \bZ$.
\end{defn}

\begin{defn}
\label{defn:13}
A morphism of $\bC$-sheaves
$\Psi \colon
A^p_{\bC} \otimes_{\bC} A^q_{\bC}
\longrightarrow
\cC(\Omega_{X_{\bullet}}(\log \cM_{X_{\bullet}}))^{p+q}$
is defined by
$\Psi=\tau \cdot (\res \otimes \res)$
for all $p,q$.
Moreover, we set
\begin{equation}
\widetilde{\Psi}
=\cC(\dlog t_1 \wedge) \cdots \cC(\dlog t_k \wedge) \cdot \Psi
\colon
A^p_{\bC} \otimes_{\bC} A^q_{\bC}
\longrightarrow
\cC(\Omega_{X_{\bullet}}(\log \cM_{X_{\bullet}}))^{p+q+k}
\end{equation}
for all $p,q$.
\end{defn}

\begin{lem}
The morphism
$\widetilde{\Psi}$
gives us a morphism of filtered complexes
\begin{equation}
\widetilde{\Psi} \colon
(A_{\bC} \otimes_{\bC} A_{\bC}, L)
\longrightarrow
(\cC(\Omega_{X_{\bullet}}(\log \cM_{X_{\bullet}}))[k], (\delta W)[-k]),
\end{equation}
where the filtration $L$ on $A_{\bC} \otimes_{\bC} A_{\bC}$
is defined
as in Definition \textup{\ref{defn:11}}.
\end{lem}
\begin{proof}
By definition,
we clearly have
$\widetilde{\Psi}(L_m(A_{\bC} \otimes_{\bC} A_{\bC})^n)
\subset \delta W_{m+k}\cC(\Omega_{X_{\bullet}}(\log \cM_{X_{\bullet}}))^{n+k}$
for all $m,n \in \bZ$.
The following lemma
implies that $\widetilde{\Psi}$ is a morphism of complexes.
\end{proof}

\begin{lem}
For $\vr \in \bZ^k_{\ge \ve}$,
\begin{equation}
\res_{\vr+\ve_i} \cdot \dlog t_i \wedge
=
(-1)^{|\vr|+1}(a_{\vr+\ve_i})_{\ast}(\id \otimes \dlog t_i\wedge)
\cdot \res_{\vr+\ve_i}
+\delta_i \cdot \res_{\vr}
\end{equation}
for all $i=\ki$.
\end{lem}
\begin{proof}
We may work in the local situation \ref{para:1}.
Then the proof is similar to \cite[Lemma 3.9]{FujisawaPLMHS}.
\end{proof}

\begin{rmk}
\label{rmk:15}
The composite of the canonical morphism
$\Omega_X \longrightarrow \Omega_X(\log \cM_X)$
and the morphism \eqref{eq:73}
gives us a morphism of complexes
$\Omega_X \longrightarrow A_{\bC}$.
We can easily check that the diagram
\begin{equation}
\begin{CD}
A_{\bC}^p \otimes \Omega_X^q @>{\overline{\Psi}}>> A_{\bC}^{p+q} \\
@VVV @VV{\res}V \\
A_{\bC}^p \otimes_{\bC} A_{\bC}^q
@>>{\Psi}> \cC(\Omega_{X_{\bullet}}(\log \cM_{X_{\bullet}}))^{p+q}
\end{CD}
\end{equation}
is commutative,
where the left vertical arrow is the tensor product of the identity
and the morphism $\Omega_X \longrightarrow A_{\bC}$ above.
However, we will not use this commutativity in this paper.
\end{rmk}

\begin{rmk}
The morphism
$\Omega_{X/\ast}(\log (\cM_X/\bN^k)) \otimes_{\bC}
\Omega_{X/\ast}(\log (\cM_X/\bN^k))
\longrightarrow \Omega_{X/\ast}(\log (\cM_X/\bN^k))$
defined by taking the wedge product
is compatible with $\Psi$, that is,
the diagram
\begin{equation}
\begin{CD}
\Omega^p_{X/\ast}(\log (\cM_X/\bN^k)) \otimes_{\bC}
\Omega^q_{X/\ast}(\log (\cM_X/\bN^k))
@>>>
\Omega^{p+q}_{X/\ast}(\log (\cM_X/\bN^k)) \\
@V{\theta \otimes \theta}VV @VV{\theta}V \\
A_{\bC}^p \otimes_{\bC} A_{\bC}^q @.
A_{\bC}^{p+q} \\
@V{\Psi}VV @VV{\res}V \\
\cC(\Omega_{X_{\bullet}}(\log \cM_{X_{\bullet}}))^{p+q}
@=
\cC(\Omega_{X_{\bullet}}(\log \cM_{X_{\bullet}}))^{p+q}
\end{CD}
\end{equation}
is commutative,
where $\theta$ is the morphism defined in \eqref{eq:28}.
Here we omit the proof
because this fact is not needed later.
\end{rmk}

\begin{rmk}
By Remark \ref{rmk:13},
$\Psi$ is compatible with the isomorphism
$A_{\bC} \otimes_{\bC} A_{\bC} \simeq A_{\bC} \otimes_{\bC} A_{\bC}$
exchanging the left and right hand sides
defined in \ref{para:6}.
\end{rmk}

\begin{para}
Here, we compute the morphism
\begin{equation}
\label{eq:60}
\gr_{-a}^LA_{\bC} \otimes_{\bC} \gr_a^LA_{\bC}
\longrightarrow
\gr_k^{\delta W}\cC(\Omega_{X_{\bullet}}(\log \cM_{X_{\bullet}}))[k]
\longrightarrow
\bigoplus_{\vr \in \bZ^k_{\ge \ve}}
(a_{\vr})_{\ast}\Omega_{X_{\vr}}[-2|\vr|+2k]
\end{equation}
given by the composition of $\gr_{-a,a}^L\widetilde{\Psi}$
and the morphism \eqref{eq:53} shifted by $k$.
\end{para}

\begin{lem}
\label{lem:14}
For $(\vq, \vr), (\vq', \vr') \in \bN^k \times \bN^k$
satisfying the conditions in \eqref{eq:61}
for $m=-a$ and for $m=a$ respectively,
the restriction of the morphism \eqref{eq:60}
on the direct summand
\begin{equation}
\bC\vu^{\vq} \otimes_{\bC}
(a_{\vr})_{\ast}
(\varepsilon_{\vr} \otimes_{\bZ} \Omega_{X_{\vr}})[a-2|\vq|]
\otimes_{\bC}
(\bC \vu^{\vq'}
\otimes_{\bC}
(a_{\vr'})_{\ast}
((\varepsilon_{\vr'} \otimes_{\bZ} \Omega_{X_{\vr'}}))[-a-2|\vq'|]
\end{equation}
via the isomorphism \eqref{eq:67} is zero unless
$\vr=\vr'=\vq+\vq'+\ve$.
For the case of $\vr=\vr'=\vq+\vq'+\ve$,
it coincides with the composite
of the following five morphisms of complexes;
the isomorphism
\begin{equation}
\begin{split}
\bC\vu^{\vq} \otimes_{\bC}
(a_{\vr})_{\ast}
(\varepsilon_{\vr} \otimes_{\bZ} \Omega_{X_{\vr}})&[a-2|\vq|]
\otimes_{\bC}
\bC \vu^{\vq'}
\otimes_{\bC}
(a_{\vr})_{\ast}
(\varepsilon_{\vr} \otimes_{\bZ} \Omega_{X_{\vr}})[-a-2|\vq'|] \\
&\simeq
((a_{\vr})_{\ast}
(\varepsilon_{\vr} \otimes_{\bZ} \Omega_{X_{\vr}})
\otimes_{\bC}
(a_{\vr})_{\ast}
(\varepsilon_{\vr} \otimes_{\bZ} \Omega_{X_{\vr}}))[-2|\vr|+2k]
\end{split}
\end{equation}
given by \eqref{eq:80},
the canonical morphism shifted by $-2|\vr|+2k$
\begin{equation}
\begin{split}
((a_{\vr})_{\ast}
(\varepsilon_{\vr} \otimes_{\bZ} \Omega_{X_{\vr}})
\otimes_{\bC}
&(a_{\vr})_{\ast}
(\varepsilon_{\vr} \otimes_{\bZ} \Omega_{X_{\vr}}))[-2|\vr|+2k] \\
&\longrightarrow
(a_{\vr})_{\ast}
((\varepsilon_{\vr} \otimes_{\bZ} \Omega_{X_{\vr}})
\otimes_{\bC}
(\varepsilon_{\vr} \otimes_{\bZ} \Omega_{X_{\vr}}))[-2|\vr|+2k],
\end{split}
\end{equation}
the isomorphism
\begin{equation}
(a_{\vr})_{\ast}
((\varepsilon_{\vr} \otimes_{\bZ} \Omega_{X_{\vr}})
\otimes_{\bC}
(\varepsilon_{\vr} \otimes_{\bZ} \Omega_{X_{\vr}}))[-2|\vr|+2k]
\simeq
(a_{\vr})_{\ast}
(\varepsilon_{\vr} \otimes_{\bZ} \varepsilon_{\vr}
\otimes_{\bZ}
\Omega_{X_{\vr}} \otimes_{\bC} \Omega_{X_{\vr}})[-2|\vr|+2k]
\end{equation}
induced by exchanging the middle terms,
the morphism
\begin{equation}
\begin{split}
(-1)^{|\vq|}(a_{\vr})_{\ast}
(&\vartheta_{\vr} \otimes \wedge)[-2|\vr|+2k] \\
&\colon
(a_{\vr})_{\ast}
(\varepsilon_{\vr} \otimes_{\bZ} \varepsilon_{\vr}
\otimes_{\bZ}
\Omega_{X_{\vr}} \otimes_{\bC} \Omega_{X_{\vr}})[-2|\vr|+2k]
\longrightarrow
(a_{\vr})_{\ast}
\Omega_{X_{\vr}}[-2|\vr|+2k],
\end{split}
\end{equation}
and the inclusion
$(a_{\vr})_{\ast}
\Omega_{X_{\vr}}[-2|\vr|+2k],
\hookrightarrow
\bigoplus_{\vs \in \bZ^k_{\ge \ve}}
(a_{\vs})_{\ast}\Omega_{X_{\vs}}[-2|\vs|+2k]$.
\end{lem}
\begin{proof}
Since the question is of local nature,
we can apply
the same argument as in the proof of \cite[Lemma 6.13]{FujisawaPLMHS}.
\end{proof}

\section{A bilinear form on $V_{\bC}$}
\label{sec:bilin-form}

In this section,
we define a bilinear form
on $V_{\bC}=\bigoplus_{a,b}E_1^{a,b}(A_{\bC}, L)$
by using the product
$\widetilde{\Psi}$
and the morphism $\Theta$
constructed in Sections \ref{sec:constr-bilin-forms}
and \ref{sec:a Cech type complex} respectively.
Then we can check that
$E_1(A_{\bC}, L)$
satisfies the conditions
to be a polarized differential $\bZ \oplus \bZ^k$-graded
Hodge-Lefschetz module,
which will be introduced in the next section.

\begin{assump}
In this section,
the \ssls degeneration
$f \colon (X, \cM_X) \longrightarrow (\ast, \bN^k)$
is assumed to be projective
and $X$ to be of pure dimension.
\end{assump}

\begin{defn}
\label{defn:20}
A finite dimensional filtered $\bC$-vector space $(V_{\bC}, F)$
is defined by
\begin{equation}
(V_{\bC}, F)=\bigoplus_{a,b \in \bZ}(E_1^{a,b}(A_{\bC}, L), F).
\end{equation}
The direct sum of the morphisms $d_1$
of the $E^1$-terms
gives us an endomorphism of $(V_{\bC},F)$
denoted by the same letter $d_1$.
Moreover, we set
\begin{equation}
V_{\bQ}=
\image(\bigoplus_{a,b \in \bZ}E^{a,b}_1(\alpha)
\colon
\bigoplus_{a,b \in \bZ}E_1^{a,b}(A_{\bQ}, L)
\longrightarrow
\bigoplus_{a,b \in \bZ}E_1^{a,b}(A_{\bC}, L)
=V_{\bC}),
\end{equation}
which is a finite dimensional $\bQ$-subspace of $V_{\bC}$
with the property $\bC \otimes_{\bQ}V_{\bQ}=V_{\bC}$.
By definition, $V_{\bQ}$ is preserved by $d_1$.
\end{defn}

\begin{para}
By \eqref{eq:67},
\begin{equation}
\label{eq:74}
(V_{\bC}, F)
\simeq
\bigoplus
\bC \vu^{\vq} \otimes_{\bC}
(\coh^j(X_{\vr}, \varepsilon_{\vr} \otimes_{\bZ} \Omega_{X_{\vr}}),
F[-|\vr|+|\vq|+k])
\end{equation}
where the direct sum is taken over the index set
\begin{equation}
\label{eq:68}
\{(\vq, \vr, j) \in \bN^k \times \bN^k \times \bZ
\mid \vr \ge \vq+\ve\}
\end{equation}
and the filtration $F$ on the right hand side
is the usual Hodge filtration on
$\coh^j(X_{\vr}, \varepsilon_{\vr} \otimes_{\bZ} \Omega_{X_{\vr}})$.
In particular,
\begin{equation}
\label{eq:75}
V_{\bC}=\bigoplus_{\vr \in \bN^k, j \in \bZ}
\bC[\vu]/(u_1^{r_1}, \dots, u_k^{r_k})
\otimes_{\bC}
\coh^j(X_{\vr}, \varepsilon_{\vr} \otimes_{\bZ} \Omega_{X_{\vr}}),
\end{equation}
as $\bC$-vector spaces.
\end{para}

\begin{defn}
\label{defn:18}
Under the identification \eqref{eq:74},
a filtered $\bC$-subspace $(V_{\bC}^{j_0, \vj}, F)$ of $(V_{\bC},F)$
is defined by
\begin{equation}
\label{eq:64}
(V_{\bC}^{j_0, \vj}, F)
\simeq
\bigoplus_{-\vr+2\vq+\ve=\vj}
\bC \vu^{\vq} \otimes_{\bC}
(\coh^{j_0+\dim X-|\vr|+k}
(X_{\vr}, \varepsilon_{\vr} \otimes_{\bZ} \Omega_{X_{\vr}}),
F[-|\vr|+|\vq|+k])
\end{equation}
for $j_0 \in \bZ$ and for $\vj \in \bZ^k$.
Moreover, a $\bQ$-subspace $V_{\bQ}^{j_0,\vj}$ of $V_{\bQ}$
is defined by
$V_{\bQ}^{j_0,\vj}=V_{\bQ} \cap V_{\bC}^{j_0,\vj}$.
\end{defn}

\begin{rmk}
By definition, we have
\begin{equation}
\label{eq:4}
(V_{\bC}, F)=\bigoplus_{j_0 \in \bZ, \vj \in \bZ^k}(V_{\bC}^{j_0, \vj}, F),
\qquad
(E_1^{a,b}(A_{\bC}, L), F)
=
\bigoplus_{|\vj|=a}(V_{\bC}^{a+b-\dim X, \vj}, F).
\end{equation}
Moreover, $(V_{\bQ}^{j_0,\vj}, (V_{\bC}^{j_0,\vj}, F))$
is a $\bQ$-Hodge structure
of weight $j_0-|\vj|+\dim X$.
In fact, we have an identification
as $\bQ$-Hodge structures
\begin{equation}
\label{eq:87}
V_{\bQ}^{j_0,\vj}
\simeq
\bigoplus_{-\vr+2\vq+\ve=\vj}
\bQ\vu^{\vq} \otimes_{\bQ}
\coh^{j_0+\dim X-|\vr|+k}
(X_{\vr}, \varepsilon_{\vr} \otimes_{\bZ} \bQ)(-|\vr|+|\vq|+k)
\end{equation}
by the canonical quasi-isomorphism
$\bQ \simeq \kos_{X_{\vr}}(\cO^{\ast}_{X_{\vr}})$
as in \ref{item:41} (cf. \cite[Corollary 1.15]{FujisawaMHSLSD})
and by \eqref{eq:86},
where $(-|\vr|+|\vq|+k)$ stands for the Tate twist as usual.
\end{rmk}

\begin{rmk}
\label{rmk:10}
Let $I \subset \ski$.
The filtration $L(I)$ on $A_{\bC}$ induces
the filtration $L(I)$ on $E_1^{a,b}(A_{\bC}, L)$.
Thus we obtain a filtration $L(I)$ on $V_{\bC}$.
We have
$L(I)_lV_{\bC}=\bigoplus_{|\vj_I| \ge -l}V_{\bC}^{j_0,\vj}$
by Lemma \ref{lem:25}.
By definition $d_1$ preserves $L(I)$ for all $I$.
\end{rmk}

\begin{lem}
\label{lem:18}
There exist the unique endomorphisms
$d'_i$ of $(V_{\bC},F)$ for $i=\ki$
such that
\begin{enumerate}
\item
\label{item:13}
$d_1=\sum_{i=1}^{k}d'_i$
\item
\label{item:14}
$d'_i(V_{\bC}^{j_0,\vj}) \subset V_{\bC}^{j_0+1, \vj+\ve_i}$
for all $j_0 \in \bZ$ and $\vj \in \bZ^k$.
\end{enumerate}
They satisfy
$d'_id'_j+d'_jd'_i=0$ for all $i,j \in \ski$.
Moreover, they preserve the subspace $V_{\bQ}$ and
$d'_i \colon V_{\bQ}^{j_0,\vj} \longrightarrow V_{\bQ}^{j_0+1,\vj+\ve_i}$
is a morphism of $\bQ$-Hodge structures for all $i \in \ski$.
\end{lem}
\begin{proof}
Since $d_1$ preserves the filtration $L(i)$ for all $i$,
we have
$d_1(V_{\bC}^{j_0, \vj}) \subset
\bigoplus_{|\vj'|=|\vj|+1, \vj' \ge \vj}V_{\bC}^{j_0+1, \vj'}$
by the second equality of \eqref{eq:4}.
The conditions $\vj' \ge \vj$ and $|\vj'|=|\vj|+1$
imply $\vj'=\vj+\ve_i$ for some $i \in \ski$.
Thus we obtain the unique morphisms $d_i'$
satisfying \ref{item:13} and \ref{item:14}.
Because $d_1$ preserves $V_{\bQ}$ and $F$,
then so does $d'_i$ for each $i \in \ski$.
Therefore
$d'_i \colon V_{\bQ}^{j_0,\vj} \longrightarrow V_{\bQ}^{j_0+1,\vj+\ve_i}$
is a morphism of $\bQ$-Hodge structures.
The equality $d_1^2=0$
implies $d'_id'_j+d'_jd'_i=0$ for all $i,j \in \ski$.
\end{proof}

\begin{defn}
\label{defn:16}
The morphism
$\nu_i \colon (A_{\bC}, L, F) \longrightarrow (A_{\bC}, L[2], F[-1])$
induces a morphism
\begin{equation}
E_r(\nu_i) \colon
(E_r^{a,b}(A_{\bC}, L), F) \longrightarrow (E_r^{a+2,b-2}(A_{\bC}, L), F[-1])
\end{equation}
for $i=\ki$.
By taking direct sum for all $a,b \in \bZ$,
we obtain
$E_1(\nu_i)
\colon (V_{\bC}, F) \longrightarrow (V_{\bC}, F[-1])$.
We set $l_i=(2\pi\sqrt{-1})E_1(\nu_i)$ for $i=\ki$.
\end{defn}

\begin{lem}
\label{lem:19}
The following holds:
\begin{enumerate}
\item
\label{item:29}
$l_i(V_{\bQ}) \subset V_{\bQ}$ for all $i \in \ski$.
\item
\label{item:35}
$l_i(V_{\bC}^{j_0,\vj}) \subset V_{\bC}^{j_0,\vj+2\ve_i}$
for all $i \in \ski$.
\item
\label{item:38}
$l_i \colon
(V_{\bQ}^{j_0,\vj}, (V_{\bC}^{j_0,\vj}, F))
\longrightarrow
(V_{\bQ}^{j_0,\vj+2\ve_i}, (V_{\bC}^{j_0,\vj+2\ve_i}, F[-1]))$
is a morphism of $\bQ$-Hodge structures.
\item
\label{item:30}
$l_id'_j=d'_jl_i$ for all $i,j \in \ski$.
\item
\label{item:36}
$l_il_j=l_jl_i$ for all $i,j \in \ski$.
\item
\label{item:37}
For any $i \in \ski$,
$\vj=(j_1, \dots, j_k) \in \bZ^k$ with $j_i > 0$,
and $j_0 \in \bZ$,
the morphism
$l_i^{j_i} \colon
V^{j_0, -\vj}_{\bC} \longrightarrow V^{j_0, -\vj+2j_i\ve_i}_{\bC}$
is an isomorphism.
\end{enumerate}
\end{lem}
\begin{proof}
\ref{item:29} follows from \eqref{eq:77}.
By definition,
$E_1(\nu_i)$ is identified with
$\bigoplus (u_i \cdot) \otimes \id$
via the isomorphism \eqref{eq:75}.
where $(u_i \cdot)$ denotes the morphsim
defined by the multiplication by $u_i$ in $\bC[\vu]$.
Therefore we obtain
\ref{item:35}, \ref{item:38}, \ref{item:36} and \ref{item:37}.
Since $E_1(\nu_i)$ commutes with $d_1$ by definition,
we obtain \ref{item:30}.
\end{proof}

\begin{notn}
\label{notn:7}
We take an ample invertible sheaf $\cL$ on $X$.
Then the cohomology class $c(\cL) \in \coh^2(X, \Omega_X)$
is defined in \ref{notn:2}.
For any $\vr \in \bZ^k_{\ge \ve}$,
we set $\cL_{\vr}=a_{\vr}^{\ast}\cL$,
which is an ample invertible sheaf on $X_{\vr}$
because $a_{\vr} \colon X_{\vr} \longrightarrow X$ is finite.
Moreover the usual first Chern class
$c_1(\cL_{\vr}) \in \coh^2(X_{\vr}, \bZ)$
is sent to $-(2\pi\sqrt{-1})^{-1}a_{\vr}^{\ast}c(\cL)$
by the morphism induced from
$\bZ \hookrightarrow \bC \simeq \Omega_{X_{\vr}}$
as in \cite[(2.2.5)]{DeligneII}.
We usually identify
$c_1(\cL_{\vr})$
and $-(2\pi\sqrt{-1})^{-1}a_{\vr}^{\ast}c(\cL)$
in $\coh^2(X_{\vr}, \Omega_{X_{\vr}})$.
\end{notn}

\begin{defn}
\label{defn:17}
The morphism $\overline{\Psi}$ in \eqref{eq:11}
induces a morphism
\begin{equation}
E_r(\overline{\Psi}) \colon
E_r^{a,b}(A_{\bC}, L) \otimes_{\bC} \coh^d(X, \Omega_X)
\longrightarrow E_r^{a,b+d}(A_{\bC}, L)
\end{equation}
as in Definition \ref{defn:12},
where $\Omega_X$ is equipped with the trivial filtration.
By using $c(\cL) \in \coh^2(X, \Omega_X)$ above,
a morphism
$l_0 \colon E_1^{a,b}(A_{\bC}, L) \longrightarrow E_1^{a,b+2}(A_{\bC}, L)$
is defined by
$l_0(\omega)=-(2\pi\sqrt{-1})^{-1}E_1(\overline{\Psi})(\omega \otimes c(\cL))$
for $\omega \in E_1^{a,b}(A_{\bC}, L)$.
\end{defn}

\begin{lem}
\label{lem:10}
We have the following:
\begin{enumerate}
\item
\label{item:32}
$l_0(V_{\bQ}) \subset V_{\bQ}$.
\item
\label{item:1}
$l_0(V_{\bC}^{j_0,\vj}) \subset V_{\bC}^{j_0+2,\vj}$
for all $p, j_0 \in \bZ, \vj \in \bZ^k$.
\item
\label{item:31}
$l_0 \colon
(V_{\bQ}^{j_0,\vj},(V_{\bC}^{j_0,\vj}, F))
\longrightarrow
(V_{\bQ}^{j_0+2,\vj}, (V_{\bC}^{j_0+2,\vj}, F[1]))$
is a morphism of $\bQ$-Hodge structures.
\item
\label{item:33}
$l_0d'_j=d'_jl_0$ for all $i \in \ski$.
\item
\label{item:15}
$l_0l_i=l_il_0$ for all $i \in \ski$.
\item
\label{item:28}
$l_0^{j_0} \colon V_{\bC}^{-j_0, \vj} \longrightarrow V_{\bC}^{j_0, \vj}$
is an isomorphism
for all $j_0 \in \bpZ$ and $\vj \in \bZ^k$.
\end{enumerate}
\end{lem}
\begin{proof}
Under the identification \eqref{eq:75},
$l_0$ is identified with
$\id \otimes \cup c_1(\cL_{\vr})$ by Lemma \ref{lem:7},
where $\cup$ denotes the cup product
induced from the morphism
$\varepsilon_{\vr} \otimes_{\bZ} \Omega_{X_{\vr}} \otimes \Omega_{X_{\vr}}
\xrightarrow{\id \otimes \wedge}
\varepsilon_{\vr} \otimes_{\bZ} \Omega_{X_{\vr}}$.
Then \ref{item:1} and \ref{item:15} are trivial.
The commutativity of $l_0$ with $d_1$ by \eqref{eq:106}
together with \ref{item:1}
implies \ref{item:33}.
Because $c_1(\cL_{\vr}) \in \coh^2(X_{\vr}, \bZ)$,
we obtain \ref{item:32}
via the identification \eqref{eq:87}.
From the Hodge theory for
$\coh^{\ast}(X_{\vr}, \varepsilon_{\vr} \otimes_{\bZ} \Omega_{X_{\vr})}$
we obtain \ref{item:31} and \ref{item:28}.
Here we note that $\dim X_{\vr}=\dim X-|\vr|+k$.
\end{proof}

\begin{defn}
\label{defn:2}
As in Definition \ref{defn:12},
the morphism $\widetilde{\Psi}$
induces a morphism
\begin{equation}
E_r(\widetilde{\Psi}) \colon
E_r^{a,b}(A_{\bC}, L) \otimes_{\bC} E_r^{c,d}(A_{\bC}, L)
\longrightarrow
E_r^{a+c-k,b+d+2k}(
\cC(\Omega_{X_{\bullet}}(\log \cM_{X_{\bullet}})), \delta W)
\end{equation}
because
$E_r^{p,q}(
\cC(\Omega_{X_{\bullet}}(\log \cM_{X_{\bullet}}))[k], (\delta W)[-k])
=
E_r^{p-k,q+2k}(
\cC(\Omega_{X_{\bullet}}(\log \cM_{X_{\bullet}})), \delta W)$.
Then we define a morphism
$S \colon
E_1^{a,b}(A_{\bC}, L) \otimes_{\bC} E_1^{c,d}(A_{\bC}, L)
\longrightarrow \bC$ by
\begin{equation}
S=
\begin{cases}
\epsilon(-a-b)\Theta \cdot E_1(\widetilde{\Psi})
&\qquad \text{if $a+c=0$ and $b+d=2\dim X$} \\
0 &\qquad \text{otherwise}, 
\end{cases}
\end{equation}
where $\Theta$ is the morphism defined in Definition \ref{defn:15}
and $\epsilon(-a-b)$ is given in Definition \ref{notn:5}.
Then $S$ induces a bilinear form
$V_{\bC} \otimes_{\bC} V_{\bC} \longrightarrow \bC$,
which is denoted by the same letter $S$.
\end{defn}

\begin{lem}
\label{lem:22}
$S \cdot (d_1 \otimes \id)=S \cdot (\id \otimes d_1)$.
on $E_1^{a,b}(A_{\bC}, L) \otimes_{\bC} E_1^{c,d}(A_{\bC}, L)$
\end{lem}
\begin{proof}
By definition, we may assume $a+c=-1, b+d=2\dim X$.
From \eqref{eq:106}, we have
\begin{equation}
E_1(\widetilde{\Psi}) \cdot (d_1 \otimes \id)
+(-1)^{a+b}E_1(\widetilde{\Psi}) \cdot (\id \otimes d_1)
=(-1)^kd_1 \cdot E_1(\widetilde{\Psi}),
\end{equation}
where
$d_1$ on the right hand side
is the morphism of $E_1$-terms for
$(\cC(\Omega_{X_{\bullet}}(\log \cM_{X_{\bullet}}), \delta W)$.
Because $\Theta \cdot d_1=0$ by Lemma \ref{lem:13},
the conclusion is obtained
from $\epsilon(-a-b)=(-1)^{a+b+1}\epsilon(-a-b-1)$.
\end{proof}

\begin{lem}
\label{lem:6}
The restriction of $S$
to the direct summand
\begin{equation}
\label{eq:85}
(\bC\vu^{\vq} \otimes_{\bC}
\coh^j(X_{\vr}, \varepsilon_{\vr} \otimes_{\bZ} \Omega_{X_{\vr}}))
\otimes_{\bC}
(\bC\vu^{\vq'} \otimes_{\bC}
\coh^{j'}(X_{\vr'}, \varepsilon_{\vr'} \otimes_{\bZ} \Omega_{X_{\vr'}}))
\end{equation}
of $V_{\bC} \otimes_{\bC} V_{\bC}$ via the identification \eqref{eq:75}
is zero 
unless $\vr=\vr'=\vq+\vq'+\ve$ and $j+j'=2\dim X_{\vr}$.
For the case of $\vr=\vr'=\vq+\vq'+\ve$ and $j+j'=2\dim X_{\vr}$,
the restriction of $S$
to the direct summand \eqref{eq:85} coincides with
\begin{equation}
(-1)^{|\vq|}\epsilon(-j)(2\pi\sqrt{-1})^{|\vr|-k}
\int_{X_{\vr}} \cdot \coh^{j,j'}(X_{\vr}, \vartheta_{\vr} \otimes \wedge),
\end{equation}
where $\vartheta_{\vr} \otimes \wedge$
denotes the composite
\begin{equation}
(\varepsilon_{\vr} \otimes_{\bZ} \Omega_{X_{\vr}})
\otimes_{\bC}
(\varepsilon_{\vr} \otimes_{\bZ} \Omega_{X_{\vr}})
\simeq
(\varepsilon_{\vr} \otimes_{\bZ} \varepsilon_{\vr})
\otimes_{\bZ}
(\Omega_{X_{\vr}} \otimes_{\bC} \Omega_{X_{\vr}})
\xrightarrow{\vartheta_{\vr} \otimes \wedge}
\Omega_{X_{\vr}}
\end{equation}
by abuse of notation.
\end{lem}
\begin{proof}
Note that
$\bC\vu^{\vq} \otimes_{\bC}
\coh^j(X_{\vr}, \varepsilon_{\vr} \otimes_{\bZ} \Omega_{X_{\vr}})
\subset E_1^{a,b}(A_{\bC}, L)$
for $a=2|\vq|-|\vr|+k$ and $b=j-2a+2|\vq|=j-a+|\vr|-k$
as in the second equality in \eqref{eq:4}.
From Lemma \ref{lem:14},
we obtain the conclusion
because
$(-1)^{a(a+b)}\epsilon(-a-b)\epsilon(|\vr|-k)
=(-1)^{(|\vr|-k)(j+|\vr|-k)}\epsilon(-j-|\vr|+k)\epsilon(|\vr|-k)
=\epsilon(-j)$.
\end{proof}

\begin{cor}
$S$ is $(-1)^{\dim X}$-symmetric.
\end{cor}

\begin{lem}
\label{lem:20}
We have the following:
\begin{enumerate}
\item
\label{item:16}
$S(V_{\bC}^{j_0, \vj} \otimes_{\bC} V_{\bC}^{j_0',\vj'})=0$
unless $j_0+j_0'=0$ and $\vj+\vj'=\vo$.
\item
\label{item:17}
$S \cdot (l_i \otimes \id)+S \cdot (\id \otimes l_i)=0$
for all $i \in \ski$.
\item
\label{item:18}
$S \cdot (l_0 \otimes \id)+S \cdot (\id \otimes l_0)=0$.
\item
\label{item:19}
$S \cdot (d'_i \otimes \id)=S \cdot (\id \otimes d'_i)$
for all $i \in \ski$.
\item
\label{item:40}
$S(F^pV_{\bC} \otimes_{\bC} F^qV_{\bC})=0$
if $p+q > \dim X$.
\end{enumerate}
\end{lem}
\begin{proof}
For \ref{item:16},
it suffices to consider the cases of
\begin{gather}
\vr=2\vq-\vj+\ve, \vr'=2\vq'-\vj'+\ve,
\vr=\vr'=\vq+\vq'+\ve, \\
j_0-|\vr|+\dim X+k+j_0'-|\vr|+\dim X+k=2(\dim X-|\vr|+k)
\end{gather}
in \eqref{eq:64},
by Lemma \ref{lem:6}.
Then these equalities imply
$\vj+\vj'=\vo$ and $j_0+j_0'=0$.
Since $E_1(\nu_i)$ is identified with
the morphism $\bigoplus (u_i \cdot) \otimes \id$
under the isomorphism \eqref{eq:75},
we can easily check \ref{item:17}
from Lemma \ref{lem:6}.
Similarly, \ref{item:1} implies \ref{item:18}
by $\epsilon(-j-2)=-\epsilon(-j)$.
The equality
$S \cdot (d_1 \otimes \id)=S \cdot (\id \otimes d_1)$ in Lemma \ref{lem:22}
combined with \ref{item:16}
implies \ref{item:19}.
We can easily check \ref{item:40}
by \eqref{eq:64} and Lemma \ref{lem:6}.
\end{proof}

\begin{lem}
\label{lem:11}
$S(V_{\bQ} \otimes V_{\bQ}) \subset \bQ$.
\end{lem}
\begin{proof}
Under the isomorphism \eqref{eq:74},
$V_{\bQ}$ is identified with
$\bigoplus \bQ \vu^{\vq} \otimes_{\bQ}
\coh^j(X_{\vr}, \varepsilon_{\vr} \otimes_{\bZ} \bQ)(-|\vr|+|\vq|+k)$
by \eqref{eq:87}.
Therefore Lemma \ref{lem:6} implies the conclusion.
\end{proof}

\begin{defn}
For $j_0 \in \bN$ and $\vj \in \bN^k$,
we set
\begin{equation}
V^{-j_0, -\vj}_{\bC, 0}
=V^{-j_0, -\vj}_{\bC} \cap \bigcap_{i=0}^k \kernel(l_i^{j_i+1}), \qquad
V_{\bQ,0}^{-j_0,-\vj}
=V_{\bQ} \cap V^{-j_0, -\vj}_{\bC, 0}.
\end{equation}
Then, together with the induced filtration $F$ on $V^{-j_0, -\vj}_{\bC, 0}$,
the data $(V^{-j_0, -\vj}_{\bQ, 0}, (V^{-j_0, -\vj}_{\bC, 0}, F))$
is a $\bQ$-Hodge structure of weight $-j_0+|\vj|+\dim X$.
\end{defn}

\begin{lem}
\label{lem:21}
The bilinear form
$S \cdot (\id \otimes C l_0^{j_0}l_1^{j_1} \dots l_k^{j_k})$
on $V^{-j_0, -\vj}_{\bQ, 0}$
is symmetric and positive definite,
where $C$ denotes the Weil operator
of a $\bQ$-Hodge structure
$(V^{-j_0, -\vj}_{\bQ, 0}, (V^{-j_0, -\vj}_{\bC, 0},F))$.
\end{lem}
\begin{proof}
Since $E_1(\nu_i)$ is identified with
the morphism $\bigoplus (u_i \cdot) \otimes \id$
via the isomorphism \eqref{eq:75},
the equality as $\bQ$-Hodge structures
\begin{equation}
V^{-j_0, -\vj}_{\bQ,0}
=\bQ \vu^{\vo}
\otimes_{\bQ}
\coh^{-j_0-|\vj|+\dim X}(X_{\vj+\ve},
\varepsilon_{\vj+\ve} \otimes_{\bZ} \bQ)(-|\vj|)
\cap \kernel(l_0^{j_0+1})
\end{equation}
can be easily seen.
We note that
$l_1^{j_1} \dots l_k^{j_k}$
is identified with the multiplication
by $\vu^{\vj} \otimes (2\pi\sqrt{-1})^{|\vj|}$.
Then Lemma \ref{lem:6} and the classical Hodge theory on $X_{\vj+\ve}$
imply the conclusion
because
$l_0$ is identified with the cup product
$\cup (2\pi\sqrt{-1})c_1(\cL_{\vj+\ve})$
on $\coh^{\ast}(X_{\vj+\ve}, \varepsilon_{\vj+\ve} \otimes_{\bZ} \bQ)$
by Lemma \ref{lem:7}.
\end{proof}

\begin{rmk}
In fact, we can check that the bilinear form
$(2\pi\sqrt{-1})^{j_0-|\vj|-\dim X}
S \cdot (\id \otimes l_0^{j_0}l_1^{j_1} \dots l_k^{j_k})$
is a polarization of the $\bQ$-Hodge structure
$V^{-j_0,-\vj}_{\bQ, 0}$
in the sense of Deligne
\cite[D\'efinition (2.1.15)]{DeligneII}.
\end{rmk}

\section{Multi-graded Hodge-Lefschetz modules}
\label{sec:multi-graded-hodge}

In this section,
we introduce the notion of a multi-graded Hodge-Lefschetz modules,
which slightly generalize the notion
of a bigraded Hodge-Lefschetz module
in \cite[Section 4]{Guillen-NavarroAznarCI}
(cf. \cite[Section 4]{SaitoMorihikoMHP} and \cite[11.3.2]{Peters-SteenbrinkMHS}).
Then, we prove Proposition \ref{prop:1},
which is a key tool for the proofs
of Theorems \ref{thm:4}, \ref{thm:2} and \ref{thm:5}
in Section \ref{sec:proofs-main-theorems}.

\begin{defn}
Let $A$ be a finite set.
A $\bZ^A$-graded Lefschetz module
$(V, \{l_a\}_{a \in A})$
consists of
a finite dimensional $\bZ^A$-graded $\bR$-vector space
$V=\bigoplus_{\vj \in \bZ^A}V^{\vj}$
and a family of endomorphisms $l_a$ of $V$
satisfying the following conditions:
\begin{enumerate}
\item
$l_al_b=l_bl_a$ for all $a,b \in A$.
\item
$l_a(V^{\vj}) \subset V^{\vj+2\ve_a}$ for all $a \in A$.
\item
\label{item:20}
For all $a \in A$,
the morphism
$l_a^{j_a}: V^{-\vj} \longrightarrow V^{-\vj+2j_a\ve_a}$
is an isomorphism
for all $\vj=\sum_{a \in A}j_a\ve_a \in \bZ^A$
with $j_a > 0$.
\end{enumerate}
A $\bZ^A$-graded Lefschetz module $(V, \{l_a\}_{a \in A})$
is called a $\bZ^A$-graded Hodge-Lefschetz module
if $V^{\vj}$ is an $\bR$-Hodge structure of certain weight
and $l_a \colon V^{\vj} \longrightarrow V^{\vj+2\ve_a}$
is a morphism of $\bR$-Hodge structures of certain type
(cf. \cite[(1.2) Definition]{GriffithsSchmid},
\cite[1.2.9]{ElZeinbook})
for all $\vj \in \bZ^A$ and $a \in A$.
We set
$V^{-\vj}_0 =V^{-\vj} \cap \bigcap_{a \in A}\kernel(l_a^{j_a+1})$
for $\vj \in \bN^A$.
Then $V^{-\vj}_0$ is a sub $\bR$-Hodge structure of $V^{-\vj}$.
Taking direct sum of the Weil operator of $V^{\vj}$
for all $\vj \in \bZ^A$,
we obtain an endomorphism $C$ of $V$.
\end{defn}

\begin{rmk}
\label{rmk:6}
As in \cite[(4.1)]{Guillen-NavarroAznarCI},
the $\bZ^A$-graded Lefschetz modules
correspond bijectively to the finite dimensional representations
of $SL(2,\bR)^A \simeq SL(2,\bR)^{|A|}$.
We set
\begin{equation}
w=
\left(
\begin{array}{rc}
0 & 1 \\
-1 & 0
\end{array}
\right)
\in SL(2,\bR).
\end{equation}
Moreover, $\vw_A \in SL(2,\bR)^A$
is the image of $w$ by the diagonal map $SL(2,\bR) \hookrightarrow SL(2,\bR)^A$.
\end{rmk}

\begin{defn}
For a $\bZ^A$-graded Hodge-Lefschetz module
$(V, \{l_a\}_{a \in A})$,
a polarization is an $\bR$-linear map
$S: V \otimes_{\bR} V \longrightarrow \bR$
satisfying the following conditions:
\begin{enumerate}
\item
\label{item:21}
$S(V^{\vj} \otimes_{\bR} V^{\vj'})=0$
if $\vj+\vj' \not= 0$.
\item
$S \colon V^{-\vj} \otimes_{\bR} V^{\vj} \longrightarrow \bR$
is a morphism of $\bR$-Hodge structures of certain type.
\item
\label{item:22}
$S \cdot (l_a \otimes \id)+S \cdot (\id \otimes l_a)=0$
for all $a \in A$.
\item
\label{item:23}
The bilinear form $S \cdot (\id \otimes C \prod_{a \in A}l_a^{j_a})$
on $V^{-\vj}_0$
is symmetric and positive definite
for all $\vj=\sum_{a \in A}j_a\ve_a \in \bN^A$.
\end{enumerate}
\end{defn}

\begin{rmk}
\label{rmk:7}
Under the conditions \ref{item:21} and \ref{item:22},
the condition \ref{item:23} is equivalent to the condition
that the bilinear form on $V$
defined by $S(x \otimes C\vw_A y)$ is symmetric and positive definite.
We can check this equivalence by computation
similar to \cite[(4.3) Proposition]{Guillen-NavarroAznarCI}.
Note that $C$ commutes with the action of $\vw_A$.
In fact, $C$ commutes with the action of $SL(2,\bR)^A$
because $C$ preserves the $\bZ^A$-grading of $V$
and commutes with $l_a$ for all $a \in A$.
\end{rmk}

Next, we define the notion of
a differential of a polarized $\bZ^A$-graded Hodge-Lefschetz module.
Because one distinguished component of $\bZ^A$
plays a special role
for the notion of a differential,
we replace $\bZ^A$ by $\bZ \oplus \bZ^A$
in the definition below.

\begin{defn}
A differential
of a polarized $\bZ \oplus \bZ^A$-graded Hodge-Lefschetz module
\begin{equation}
(V=\bigoplus_{j_0 \in \bZ, \vj \in \bZ^A}V^{j_0, \vj},
\{l_0, \{l_a\}_{a \in A}\}, S),
\end{equation}
is a family of $\bR$-linear maps
$d_a: V \longrightarrow V$ for $a \in A$
satisfying the following conditions:
\begin{enumerate}
\item
$d_a(V^{j_0, \vj}) \subset V^{j_0+1, \vj+\ve_a}$
for $a \in A$.
\item
$d_a \colon V^{j_0,\vj} \longrightarrow V^{j_0+1,\vj+\ve_a}$
is a morphism of $\bR$-Hodge structures of certain type.
\item
\label{item:24}
$d_ad_b+d_bd_a=0$ for all $a,b \in A$.
\item
$d_al_0=l_0d_a$ and $d_al_b=l_bd_a$ for all $a, b \in A$.
\item
\label{item:25}
$S \cdot (d_a \otimes \id)=S \cdot (\id \otimes d_a)$
for all $a \in A$.
\end{enumerate}
\end{defn}

\begin{rmk}
For the case of $|A|=1$,
a polarized differential $\bZ \oplus \bZ$-graded Hodge-Lefschetz module
is nothing but a polarized differential bigraded Hodge-Lefschetz module
in \cite[Section 4]{Guillen-NavarroAznarCI}.
\end{rmk}

\begin{defn}
Let $(V, \{l_0, \{l_a\}_{a \in A}\}, S, \{d_a\}_{a \in A})$
be a polarized differential $\bZ \oplus \bZ^A$-graded Hodge-Lefschetz module.
For $B \subset A$
and for $\vc=\sum_{a \in B}c_a \ve_a \in \bR^B$,
we set $d_B=\sum_{a \in B}d_a$
and $l_B(\vc)=\sum_{a \in B}c_al_a$.
Then $d_B^2=0$ by \ref{item:24}.
Moreover, by setting
\begin{equation}
\label{eq:83}
V^{j_0,j_1,\vj'}
=\bigoplus_{|\vj_B|=j_1, \vj_{A \setminus B}=\vj'}
V^{j_0, \vj},
\end{equation}
for $j_0, j_1 \in \bZ$ and $\vj' \in \bZ^{A \setminus B}$,
we have
$V=\bigoplus V^{j_0,j_1,\vj'}$,
where the direct sum is taken
over all $j_0, j_1 \in \bZ, \vj' \in \bZ^{A \setminus B}$.
Then $d_B(V^{j_0,j_1,\vj'}) \subset V^{j_0+1,j_1+1,\vj'}$
and $\coh(V, d_B)=\kernel(d)/\image(d)$
carries the natural direct sum decomposition
\begin{equation}
\label{eq:81}
\coh(V, d_B)
=\bigoplus_{j_0, j_1, \vj'}
\coh(V, d_B)^{j_0,j_1,\vj'}
\end{equation}
by setting
\begin{equation}
\label{eq:91}
\coh(V, d_B)^{j_0,j_1,\vj'}
=
V^{j_0,j_1,\vj'} \cap\kernel(d_B)
\bigl/
V^{j_0, j_1, \vj'} \cap \image(d_B).
\end{equation}
Because of \ref{item:24}--\ref{item:25},
the morphisms $l_0$, $l_B(\vc)$, $l_a$, $d_a$ ($a \in A \setminus B$)
and $S$ commute with $d_B$
and descend to $\coh(V, d_B)$,
denoted by the same letters.
We set $d_1=0$ on $\coh(V, d_B)$.
\end{defn}

\begin{prop}
\label{prop:1}
Equipped with the direct sum decomposition \eqref{eq:81},
\begin{equation}
\label{eq:84}
(\coh(V, d_B), \{l_0, l_B(\vc), \{l_a\}_{a \in A \setminus B}\},
S, \{d_1, \{d_a\}_{a \in A \setminus B}\})
\end{equation}
is a polarized differential $\bZ \oplus \bZ \oplus \bZ^{A \setminus B}$-graded
Hodge-Lefschetz module
if $c_a > 0$ for all $a \in B$.
\end{prop}
\begin{proof}
Because $c_a > 0$,
the condition \ref{item:23} is satisfied
for $\{c_al_a\}_{a \in B} \cup \{l_a\}_{a \in A \setminus B}$.
Therefore we may assume $\vc=\ve_B$
by replacing $l_a$ with $c_al_a$.

First, we treat the case of $A=B$.
In this case, the $\bZ \oplus \bZ$-grading \eqref{eq:83} for $B=A$
corresponds to the representation of $SL(2,\bR) \times SL(2,\bR)$
induced from the inclusion
$SL(2,\bR) \times SL(2,\bR) \hookrightarrow SL(2,\bR) \times SL(2,\bR)^A$,
where the first factor is the identity of $SL(2,\bR)$
and the second factor $SL(2,\bR) \hookrightarrow SL(2,\bR)^A$
is the diagonal map.
Then the action of $(w,w) \in SL(2,\bR) \times SL(2,\bR)$ on $V$
is the same as the action of $(w, w_A) \in SL(2,\bR) \times SL(2,\bR)^A$.
Therefore $(V=\bigoplus V_A^{j_0, j_1}, \{l_0, l_A\}, S, d_A)$
is a polarized differential bigraded Hodge-Lefschetz module.
By applying \cite[(4.5) Th\'eor\`em]{Guillen-NavarroAznarCI},
we obtain \ref{item:20} and \ref{item:23}
for $\coh(V, d_A)$.

Next, we treat the general case.
Note that $l_a$ on $\coh(V, d_B)$ for $a \in A \setminus B$
trivially satisfies the condition \ref{item:20}.
Moreover, the endomorphism $C$ commutes with $d_B$
and descends to $\coh(V, d_B)$,
which coincides with the endomorphism $C$ of $\coh(V,d_B)$.
The $\bZ \oplus \bZ^B$-grading
$V=\bigoplus_{j_0, \vj'}(\bigoplus_{\vj_B=\vj'} V^{j_0,\vj})$
gives us a $\bZ \oplus \bZ^B$-graded Lefschetz module
$(V, \{l_0, \{l_a\}_{a \in B}\})$,
which corresponds to the representation
of $SL(2,\bR) \times SL(2,\bR)^B$
induced from the injection
defined by
$SL(2,\bR) \times SL(2,\bR)^B \ni (g_0, g)
\mapsto (g_0, g, \id)
\in SL(2,\bR) \times SL(2,\bR)^B \times SL(2,\bR)^{A \setminus B}
\simeq SL(2,\bR) \times SL(2,\bR)^A$.
By using
$\vw'=(\id, \id, \vw_{A \setminus B})
\in SL(2,\bR) \times SL(2,\bR)^B \times SL(2,\bR)^{A \setminus B}$,
we set
$S_B(x \otimes y)=S(x \otimes \vw' y)$
for $x,y \in V$.
Then the bilinear form $S_B$
satisfies the condition \ref{item:23}
as in Remark \ref{rmk:7}
and
$(V, \{l_0, \{l_a\}_{a \in B}\}, S_B, \{d_a\}_{a \in B})$
is a polarized differential $\bZ \oplus \bZ^B$-graded Hodge-Lefschetz module.
Therefore
$(\coh(V, d_B), \{l_0, l_B(\ve_B)\}, S_B)$
is a polarized bigraded Hodge-Lefschetz module as proved above.
Thus $l_0, l_B$ satisfy the condition \ref{item:20}
for $\coh(V, d_B)$
equipped with the direct sum decomposition \eqref{eq:81}.
Because
$S(x \otimes C(w,\vw_A)y)=S_B(x \otimes C(w,\vw_B)y)$,
the bilinear form $S$ on $\coh(V, d_B)$
satisfies the desired condition \ref{item:23}.
\end{proof}

\section{Proof of Theorems \ref{thm:3}, \ref{thm:4}, \ref{thm:2} and \ref{thm:5}}
\label{sec:proofs-main-theorems}

First, we prove the following lemma,
which slightly generalize Lemma 3.17 of \cite{Fujino-Fujisawa}.

\begin{lem}
\label{lem:24}
Let
$((A_{\bQ},W^f,W), (A_{\bC},W^f,W,F), \alpha)$
be a filtered $\bQ$-mixed Hodge complex
and $\nu: A_{\bC} \longrightarrow A_{\bC}$
a morphism of complexes
preserving the filtration $W^f$
and satisfying the condition
$\nu(W_mA_{\bC}) \subset W_{m-2}A_{\bC}$
for all $m$.
If the morphism
$\coh^n(\gr_m^{W^f}\nu)^l$
induces an isomorphism
\begin{equation}
\gr_l^{W[-m]}\coh^n(\gr_m^{W^f}A_{\bC})
\overset{\simeq}{\longrightarrow}
\gr_{-l}^{W[-m]}\coh^n(\gr_m^{W^f}A_{\bC})
\end{equation}
for all $l \in \bpZ$ and $m,n \in \bZ$,
then we have the following\textup{:}
\begin{enumerate}
\item
\label{item:26}
The spectral sequence $E_r^{p,q}(A_{\bC}, W^f)$ degenerates at $E_2$-terms.
\item
\label{item:27}
The morphism $\coh^n(\nu)^l$
induces an isomorphism
\begin{equation}
\gr_{l+m}^W\gr_m^{W^f}\coh^n(A_{\bC})
\overset{\simeq}{\longrightarrow}
\gr_{-l+m}^W\gr_m^{W^f}\coh^n(A_{\bC})
\end{equation}
for all $l \in \bpZ$ and $m,n \in \bZ$.
\end{enumerate}
\end{lem}
\begin{proof}
In this proof,
we write $E_r^{p,q}=E_r^{p,q}(A_{\bC}, W^f)$ for short.
The morphism of $E_r$-terms
$d_r \colon E_r^{p,q} \longrightarrow E_r^{p+r,q-r+1}$
is strictly compatible with
$W_{\rec}$ on the left hand side and $W_{\rec}[1]$ on the right
by \cite[6.1.8 Th\'eor\`em]{ElZeinbook}.
On the other hand,
the morphism $\nu$ induces a morphism of the spectral sequences
$E_r(\nu) \colon E_r^{p,q} \longrightarrow E_r^{p,q}$.
Via the identification
$E_1^{p,q} \simeq \coh^{p+q}(\gr_{-p}^{W^f}A_{\bC})$,
the assumption implies that
$E_1(\nu)^l$ induces an isomorphism
$\gr_l^{W[p]}E_1^{p,q}
\overset{\simeq}{\longrightarrow}
\gr_{-l}^{W[p]}E_1^{p,q}$
for all $l \in \bpZ$ and $p,q \in \bZ$.
Then the strictness of $d_1$ above
implies that $E_2(\nu)^l$
induces an isomorphism
$\gr_l^{W_{\rec}[p]}E_2^{p,q}
\overset{\simeq}{\longrightarrow}
\gr_{-l}^{W_{\rec}[p]}E_2^{p,q}$
for all $l \in \bpZ$,
that is, $W_{\rec}[p]$ is the monodromy weight filtration
of $E_2(\nu)$ on $E_2^{p,q}$
for all $p,q \in \bZ$.
Because $d_2$ commutes with $E_2(\nu)$,
the monodromy weight filtration of $E_2(\nu)$
is preserved by $d_2$.
Namely, $d_2 \colon E_2^{p,q} \longrightarrow E_2^{p+2,q-1}$
preserves $W_{\rec}[p]$ on the left hand side
and $W_{\rec}[p+2]$ on the right.
Therefore
\begin{equation}
d_2((W_{\rec})_mE_2^{p,q})
=d_2(W_{\rec}[p]_{m+p}E_2^{p,q})
\subset (W_{\rec}[p+2])_{m+p}E_2^{p+2,q-1}
=(W_{\rec}[1])_{m-1}E_2^{p+2,q-1}
\end{equation}
for all $m \in \bZ$.
Thus we obtain $d_2=0$ on $E_2^{p,q}$ for all $p,q \in \bZ$,
because of the strict compatibility of $d_2$
with $W_{\rec}$ on $E_2^{p,q}$ and $W_{\rec}[1]$ on $E_2^{p+2,q-1}$.
Repeating this procedure inductively,
we obtain $d_r=0$ for all $r \ge 2$.
Once \ref{item:26} is obtained,
\ref{item:27} follows from
Lemma 3.17 of \cite{Fujino-Fujisawa}.
\end{proof}

\begin{proof}[Proof of
Theorems \textup{\ref{thm:3}}, \textup{\ref{thm:4}},
\textup{\ref{thm:2}} and \textup{\ref{thm:5}}]
\label{para:12}
A \ssls degeneration
$f \colon (X, \cM_X) \longrightarrow (\ast, \bN^k)$
is assumed to be projective.
Moreover, we may assume
that $X$ is of pure dimension
by considering the connected components.
We fix $I \subset \ski$ and $\vc=(c_i)_{i=1}^k \in (\bpR)^k$,
and set $J=\ski \setminus I$.
Morphisms of complexes
$\nu(\vc), \nu_J(\vc_J) \colon A_{\bC} \longrightarrow A_{\bC}$
are defined by $\nu(\vc)=\sum_{i=1}^{k}c_i\nu_i$
and $\nu_J(\vc_J)=\sum_{i \in J}c_i\nu_i$.
Recall that these morphisms induces $N(\vc)$ and $N_J(\vc_J)$
on $\coh^{\ast}(X, A_{\bC})$.

Let $V_{\bQ}$ and $V_{\bQ}^{j_0,\vj}$
be as in Definitions \ref{defn:20} and \ref{defn:18}.
We set $V_{\bR}=\bR \otimes V_{\bQ}$
and $V_{\bR}^{j_0,\vj}=\bR \otimes V_{\bQ}^{j_0,\vj}$.
Then $V_{\bC}=\bC \otimes_{\bR} V_{\bR}$
and $V_{\bC}^{j_0,\vj}=\bC \otimes_{\bR} V_{\bR}^{j_0,\vj}$.
Thus we obtain
\begin{equation}
\label{eq:89}
(V_{\bR}=\bigoplus_{j_0 \in \bZ, \vj \in \bZ^k}V_{\bR}^{j_0,\vj},
\{l_0, \{l_i\}_{i=1}^k\}, S, \{d'_i\}_{i=1}^k),
\end{equation}
which is
a polarized differential $\bZ \oplus \bZ^k$-graded Hodge-Lefschetz module
by Lemmas \ref{lem:18}, \ref{lem:19}, \ref{lem:10},
\ref{lem:20}, \ref{lem:11} and \ref{lem:21}.
We set
\begin{equation}
V_{\bR}^{j_0,j_1,j_2}
=\bigoplus_{|\vj_J|=j_1, |\vj_I|=j_2}V_{\bR}^{j_0,\vj}, \quad
V_{\bC}^{j_0,j_1,j_2}
\bigoplus_{|\vj_J|=j_1, |\vj_I|=j_2}V_{\bC}^{j_0,\vj}, \quad
d'_J=\sum_{i \in J}d'_i, \quad
d'_I=\sum_{i \in I}d'_i,
\end{equation}
and $l_I(\vc_I)=\sum_{i \in I}c_il_i$ for $I \subset \ski$.
We use $l(\vc)$ instead of $l_{\ski}(\vc)$.
Then we have
\begin{equation}
V_{\bC}=\bigoplus_{j_0,j_1,j_2 \in \bZ}V_{\bC}^{j_0,j_1,j_2}, \quad
d'_J(V_{\bC}^{j_0,j_1,j_2}) \subset V_{\bC}^{j_0+1,j_1+1,j_2}, \quad
d'_I(V_{\bC}^{j_0,j_1,j_2})
\subset V_{\bC}^{j_0+1,j_1,j_2+1}
\end{equation}
and $d_1=d'_J+d'_I$
by definition.

Since $L(I) \ast L(J)=L$ on $A_{\bC}$
by Corollary \ref{cor:1},
we have the identifications
\begin{equation}
\label{eq:101}
\bigoplus_{p+q=a}E_1^{p,q+b}(\gr_{-q}^{L(I)}A_{\bC}, L(J))
\simeq
E_1^{a,b}(A_{\bC}, L)
\simeq
\bigoplus_{|\vj|=a}V_{\bC}^{a+b-\dim X, \vj}
\end{equation}
for all $a,b$,
which induces
\begin{equation}
\bigoplus_{p+q=a, p \ge -l}E_1^{p,q+b}(\gr_{-q}^{L(I)}A_{\bC}, L(J))
\simeq
L(J)_lE_1^{a,b}(A_{\bC}, L)
\simeq
\bigoplus_{|\vj|=a, |\vj_J| \ge -l}V_{\bC}^{a+b-\dim X, \vj}
\end{equation}
for all $l$
as in \ref{para:9} and Remark \ref{rmk:10}.
Therefore
\begin{equation}
\label{eq:95}
V_{\bC}^{j_0,j_1,j_2}
\simeq
E_1^{j_1,j_0-j_1+\dim X}(\gr_{-j_2}^{L(I)}A_{\bC}, L(J))
\end{equation}
for all $j_0,j_1,j_2 \in \bZ$.
We denote the morphism of $E_1$-terms
of $E_r^{p,q}(\gr_m^{L(I)}A_{\bC}, L(J))$
by $\widetilde{d_1}$ for a while.
Then we obtain
$\widetilde{d_1} \colon
V_{\bC}^{j_0,j_1,j_2} \longrightarrow V_{\bC}^{j_0+1,j_1+1,j_2}$
via the identification \eqref{eq:95}.
On the other hand,
the morphism
\begin{equation}
\gamma \colon
(\gr_m^{L(I)}A_{\bC}, L(J))
\longrightarrow
(\gr_{m-1}^{L(I)}A_{\bC}[1], L(J))
\end{equation}
in the filtered derived category
as in \ref{para:10}
induces a morphism
$E_1(\gamma) \colon
V_{\bC}^{j_0,j_1,j_2} \longrightarrow V_{\bC}^{j_0+1,j_1,j_2+1}$
via the identification \eqref{eq:95}.
Because $d_1=\widetilde{d_1}+E_1(\gamma)$ by Lemma \ref{lem:23},
$\widetilde{d_1}=d'_J$ and $E_1(\gamma)=d'_I$.
Since $L(J)=L[-m]$ on $\gr_m^{L(I)}A_{\bC}$
by Corollary \ref{cor:1},
we have
\begin{equation}
\label{eq:94}
\begin{split}
\bigoplus_{|\vj'|=j_2}\bC \otimes_{\bR} \coh(V_{\bR}, d'_J)^{j_0,j_1,\vj'}
&\simeq
V_{\bC}^{j_0,j_1,j_2} \cap \kernel(d'_J)
\bigl/V_{\bC}^{j_0,j_1,j_2} \cap \image(d'_J) \\
&\simeq
E_2^{j_1,j_0-j_1+\dim X}(\gr_{-j_2}^{L(I)}A_{\bC}, L(J)) \\
&\simeq
\gr_{-j_1}^{L[j_2]}\coh^{j_0+\dim X}(X, \gr_{-j_2}^{L(I)}A_{\bC})
\end{split}
\end{equation}
from \eqref{eq:91},
the equality $d'_J=\widetilde{d_1}$
and the $E_2$-degeneracy \ref{item:8}.
In particular,
\begin{equation}
\label{eq:96}
\bC \otimes_{\bR} \coh(V_{\bR}, d_1)^{j_0,j_1}
\simeq
\gr_{-j_1}^L\coh^{j_0+\dim X}(X, A_{\bC})
\end{equation}
as the case of $I=\emptyset$.
In the identification \eqref{eq:94},
the morphism $\id \otimes d'_I$
on the first term
$\bigoplus\bC \otimes_{\bR} \coh(V_{\bR}, d'_J)^{j_0,j_1,\vj'}$
is identified with $\gr_{-j_1}^{L[j_2]}\coh^{j_0+\dim X}(X,\gamma)$
on the last term.
Moreover, under the identification
$\coh^{j_0+\dim X}(X, \gr_{-j_2}^{L(I)}A_{\bC})
\simeq E_1^{j_2, j_0-j_2+\dim X}(A_{\bC}, L(I))$,
the morphism $\coh^{j_0+\dim X}(X,\gamma)$
is identified with the morphism of $E_1$-terms
of the spectral sequence $E_r^{p,q}(A_{\bC}, L(I))$.
By \ref{item:10}, the morphism of $E_1$-terms
$E_1^{j_2, j_0-j_2+\dim X}(A_{\bC}, L(I))
\longrightarrow E_1^{j_2+1, j_0-j_2+\dim X}(A_{\bC}, L(I))$
is strictly compatible with $L[j_2]$ and $L[j_2+1]$,
we obtain the identification
\begin{equation}
\label{eq:100}
\bC \otimes_{\bR} \coh(\coh(V_{\bR}, d'_J), d'_I)^{j_0,j_1,j_2}
\simeq
\gr_{-j_1}^{L_{\rec}[j_2]}E_2^{j_2,j_0-j_2+\dim X}(A_{\bC}, L(I))
\end{equation}
for all $j_0,j_1,j_2 \in \bZ$.

In the identifications \eqref{eq:94}--\eqref{eq:100},
\begin{equation}
\label{eq:99}
(\coh(V_{\bR}, d'_J), \{l_0,l_J(\vc_J),\{l_i\}_{i \in I}\},
S, \{d_1=0, \{d'_i\}_{i \in I}\})
\end{equation}
is a $\bZ \oplus \bZ \oplus \bZ^I$-graded
Hodge-Lefschetz module,
\begin{equation}
(\coh(V_{\bR}, d_1),
l_0, l(\vc), S)
\end{equation}
is a $\bZ \oplus \bZ$-graded Hodge-Lefschetz module,
and further,
\begin{equation}
(\coh(\coh(V_{\bR},d'_J), d'_I),\{l_0, l_J(\vc_J), l_I(\vc_I)\}, S)
\end{equation}
is a $\bZ \oplus \bZ \oplus \bZ$-graded Hodge-Lefschetz module
by Proposition \ref{prop:1}.

Under the identification \eqref{eq:96},
the morphism $l_0$ is identified with the morphism induced by
$(2\pi \sqrt{-1})(\cup c(\cL))$.
Therefore
$(\cup c(\cL))^i$ induces an isomorphism
\begin{equation}
\gr_l^L\coh^{-i+\dim X}(X, A_{\bC})
\overset{\simeq}{\longrightarrow}
\gr_l^L\coh^{i+\dim X}(X, A_{\bC})
\end{equation}
for all $i \in \bpZ$ and $l \in \bZ$.
Hence we obtain Theorem \ref{thm:5}.

Under the identification \eqref{eq:94},
$l_J(\vc_J)$ on the first term
$\bigoplus \bC \otimes_{\bR} \coh(V_{\bR}, d'_J)^{j_0,j_1,\vj'}$
is identified with the morphism induced from
$(2\pi\sqrt{-1})\coh^{j_0+\dim X}(X, \gr_{-j_2}^{L(I)}\nu_J(\vc_J))$.
Because
$\gr_m^{L(I)}\nu_J(\vc_J)=\gr_m^{L(I)}\nu(\vc)$
on $\gr_m^{L(I)}A_{\bC}$,
the morphism $\coh^j(X, \gr_m^{L(I)}\nu(\vc))^l$
induces an isomorphism
\begin{equation}
\gr_l^{L[-m]}\coh^j(X, \gr_m^{L(I)}A_{\bC})
\overset{\simeq}{\longrightarrow}
\gr_{-l}^{L[-m]}\coh^j(X, \gr_m^{L(I)}A_{\bC})
\end{equation}
for all $l \in \bpZ$ and $m \in \bZ$.
Therefore we obtain Theorems \ref{thm:3} and \ref{thm:2}
by Lemma \ref{lem:24}.

Under the identification \eqref{eq:100},
the morphism $l_I(\vc_I)$ is identified with
$\gr_{-j_2}^{L_{\rec}[j_2]}E_2(\nu_I(\vc_I))$.
Moreover, $L(J)=L[-m]$ on $\gr_m^{L(I)}A_{\bC}$ implies
$L(J)=L[p]$ on $E_1^{p,q}(A_{\bC}, L)$
and $L(J)_{\rec}=L_{\rec}[p]$ on $E_2^{p,q}(A_{\bC}, L(I))$.
Then
$\gr_m^{L(J)_{\rec}}E_2(\nu_I(\vc_I))^l$
induces an isomorphism
\begin{equation}
\gr_m^{L(J)_{\rec}}E_2^{-l,i+l}(A_{\bC}, L(I))
\overset{\simeq}{\longrightarrow}
\gr_m^{L(J)_{\rec}}E_2^{l,i-l}(A_{\bC}, L(I))
\end{equation}
for all $l \in \bpZ$ and $i, m \in \bZ$.
Thus $E_2(\nu_I(\vc_I))^l$
induces an isomorphism
\begin{equation}
E_2^{-l,i+l}(A_{\bC}, L(I))
\overset{\simeq}{\longrightarrow}
E_2^{l,i-l}(A_{\bC}, L(I))
\end{equation}
for all $l \in \bpZ$ and $i \in \bZ$.
Therefore we obtain Theorem \ref{thm:4}
by the $E_2$-degeneracy in Theorem \ref{thm:3}.
\end{proof}

\providecommand{\bysame}{\leavevmode\hbox to3em{\hrulefill}\thinspace}
\providecommand{\MR}{\relax\ifhmode\unskip\space\fi MR }
\providecommand{\MRhref}[2]{%
  \href{http://www.ams.org/mathscinet-getitem?mr=#1}{#2}
}
\providecommand{\href}[2]{#2}

\end{document}